\documentclass[10pt]{amsart}
\usepackage{amsfonts,url}
\usepackage{todonotes}
\usepackage[margin=20mm]{geometry}
\usepackage{amssymb, amsmath, amsthm, amscd, accents, mathtools}
\usepackage[T1]{fontenc}
\usepackage{hyperref}
\usepackage{enumitem}
\allowdisplaybreaks

\makeatletter

\renewcommand*{\descriptionlabel}[1]{%
  \let\orglabel\label
  \let\label\@gobble
  \phantomsection
  \edef\@currentlabel{#1\unskip}%
  \let\label\orglabel
  \hspace\labelsep \upshape\bfseries #1 %
}
\makeatother

\let\oldtocsection=\tocsection

\let\oldtocsubsection=\tocsubsection

\renewcommand{\tocsection}[2]{\hspace{0em}\oldtocsection{#1}{#2}}
\renewcommand{\tocsubsection}[2]{\hspace{1em}\oldtocsubsection{#1}{#2}}

 \newcommand*{\cm}{}
 \newcommand*{\cb}{}

\newcommand*{\erj}{}
\newcommand*\ar{}
\newcommand{\arr}{}

\newtheorem{lemma}{Lemma}
\newtheorem{theorem}[lemma]{Theorem}
\newtheorem{proposition}[lemma]{Proposition}
\newtheorem{corollary}[lemma]{Corollary}

\theoremstyle{definition}
\newtheorem{definition}[lemma]{Definition}
\newtheorem{remark}[lemma]{Remark}
\newtheorem{example}[lemma]{Example}
\numberwithin{lemma}{section}
\numberwithin{equation}{section}

\DeclareMathOperator{\Div}{div}

\DeclareMathOperator*{\esssup}{ess\, sup}
\newcommand*\Lip{{\rm Lip}_{1,1}(\mR^d)}

\newcommand*\mR{\mathbb{R}}
\newcommand*\mL{\mathcal{L}}
\newcommand*\mP{\mathcal{P}(\mR^d)}
\newcommand{\dm}[1]{\frac{\delta{#1}}{\delta m}}
\newcommand*\wtm{\widetilde{m}}
\newcommand*\wtmu{\widetilde{\mu}}
\newcommand*\eps{\varepsilon}
\newcommand*\alow{\underline{\alpha}}

\subjclass[2020]{35Q89, 35S10, 35A01, 35A08, 35K08, 49L12, 45K05, 35K61, 46E15, 46G05}

\title[Master equation with nonlocal diffusions]{
The master equation for mean field game systems with 
fractional and nonlocal diffusions
}
\author{Espen R. Jakobsen}
\address{ Norwegian University of Science and Technology, H\o{}gskoleringen 1, 7034 Trondheim, Norway}
\email{espen.jakobsen@ntnu.no}
\author{Artur Rutkowski}
\address{ Norwegian University of Science and Technology, H\o{}gskoleringen 1, 7034 Trondheim, Norway}
\address{Faculty of Pure and Applied Mathematics,
Wroc\l aw University of Science and Technology,
Wyb. Wyspia\'nskiego 27, 50-370 Wroc\l aw, Poland}
\email{artur.rutkowski@pwr.edu.pl}
\begin{document}
\begin{abstract}
    We prove existence and uniqueness of  classical solutions of the master equation for mean field game (MFG) systems with fractional and nonlocal diffusions. We cover a large class of L\'evy diffusions of order greater than one, including 
    purely nonlocal, local, and even mixed local-nonlocal operators. In the process we prove refined well-posedness results for the 
    MFG systems, results that include the mixed local-nonlocal case. We also show various auxiliary results on viscous Hamilton-Jacobi equations, linear parabolic equations, and linear forward-backward systems that may be of independent interest. This includes a rigorous treatment of certain equations and systems with data and solutions in the duals of H\"older spaces $C^\gamma_b$ on the whole of $\mR^d$. We do not assume existence of any moments for the initial distributions of players. 
%
     In a future work we will use the results of this paper to prove the convergence of $N$-player games to mean field games as $N\to\infty$.
\end{abstract}
\maketitle\vspace{-40pt}
\tableofcontents
\section{Introduction}
The goal of this paper is to show well-posedness for the master equation associated to the 
mean field game (MFG) system
\begin{align}
    \begin{cases}
    -\partial_t u  - \mL u + H(x,u,Du) = F(x,m(t))\quad &\textrm{in}\ (t_0,T)\times \mR^d,\\
    \partial_t m  - \mL^* m - \Div (mD_pH(x,u,Du)) = 0\quad &\textrm{in}\ (t_0,T)\times \mR^d,\\
    m(t_0) = m_0,\qquad u(T,x) = G(x,m(T)).\label{eq:MFG}
    \end{cases}
\end{align}
Here $\mL$ is a nonlocal or mixed local-nonlocal (or even just local) diffusion operator, a constant coefficient L\'evy operator defined for $u\in C_b^2(\mR^d)$ as
\begin{equation}\label{eq:operator}
    \mL u (x) = B\cdot Du(x) + \Div(A Du(x))+ \int_{\mR^d}(u(x+z) - u(x) - Du(x)\cdot z \textbf{1}_{B(0,1)}(z))\, \nu(dz),
\end{equation}
 with L\'evy triplet $(B,A,\nu)\in \mR^d \times \mR^{d\times d}\times \mathcal M_+(\mR^d)$ where 
  $A$ 
 is a positive semi-definite symmetric $d\times d$ matrix, 
 and $\nu$ is a 
 nonnegative Borel measure satisfying $\int_{\mR^d} (1 \wedge |z|^2)\nu(dz) < \infty$, called the L\'evy measure.  By $\mL^*$ we denote the formal adjoint of $\mL$, cf. Subsection~\ref{sec:hk}. The system describes a Nash equilibrium of a mean field game, a limit of $N$-player games as $N\to\infty$ under suitable assumptions. We refer to \cite{MR4309434} for a heuristic derivation of \eqref{eq:MFG} from the related control problem. The first equation of \eqref{eq:MFG} is a backward in time viscous Hamilton--Jacobi (H--J) equation for the value function $u$ of a generic player. The second equation is a forward in time Fokker--Planck (F--P) equation for 
 the distribution of players $m$. The equations are nonlinear and coupled. The coupling terms $F$ and $G$ are nonlocal in the sense that they depend only on $m(t)$ instead of $m(t,x)$. The operator $\mL$ corresponds to the so-called individual/idiosyncratic noise, we do not study cases with  common noise in this paper.

MFG systems were first introduced by Lasry and Lions \cite{MR2295621} and Huang, Malham\'e, and Caines \cite{huang2006large} as a way to study games with a large number of players. Since then, the topic grew immensely and found its applications, e.g., in finance, economy, energy management, crowd and traffic modelling, and mathematical biology. There are two main approaches to MFGs: probabilistic \cite{huang2006large,MR3752669,MR3753660} and PDE \cite{MR2295621,MR4214773}. This paper is written within the PDE setting.  We refer to the books \cite{MR4214773,MR3967062,MR3752669,MR3753660,MR3559742} for derivations of the models, discussion of the different approaches,  
and further reading. 
MFGs with non-Gaussian noise and corresponding nonlocal diffusion operators $\mL$ are still a rather new research direction with relatively few works, especially compared to MFGs with Gaussian noise/local operators. Cesaroni et al. \cite{MR3912635} studied stationary (time-independent) systems on the torus driven by the fractional Laplacian $(-\Delta)^{\alpha/2}$ with $\alpha\in (1,2)$. In \cite{MR3934106} Cirant and Goffi work with time-dependent equations, allowing all $\alpha \in (0,2)$, but the supercritical exponents $\alpha\in (0,1]$ required an entirely different approach and weaker type of solutions. Recently, Ye, Zou, and Liu \cite{fracnonsep} gave a well-posedness result for subcritical fractional time-dependent MFG systems on torus, with non-separable Hamiltonians. Some 
mixed local-nonlocal diffusions were studied in the context of Bertrand and Cournot MFGs by Graber, Ignazio, and Neufeld \cite{MR4223351}. \arr Benazzoli, Campi, and Di Persio \cite{MR4158808} have studied mean field games from the probabilistic point of view for diffusions with jumps of finite intensity.  We also mention here the recent work of Cavallazzi \cite{2022arXiv221201079C} on the propagation of chaos for stable driven McKean--Vlasov SDEs.
One of the key references for our work is \cite{MR4309434} by Ersland and Jakobsen. The paper provides well-posedness for the MFG systems in the whole Euclidean space, driven by more general nonlocal operators of order greater than $1$, i.e. nondegenerate operators with a sufficiently strong regularizing effect to produce smooth solutions of the $u$-equation in \eqref{eq:MFG}.  
The assumption on the order of the operator was given in \cite{MR4309434} in terms of the heat kernel\ar{. We use a slightly weaker version, see} \ref{eq:K} below. It allows very general operators, e.g., $(-\Delta)^{\alpha/2}$ with $\alpha\in (1,2]$, mixed local-nonlocal operators, strongly anisotropic stable operators like $ (-\partial^2_{x_1})^{\alpha_1/2}+(-\partial^2_{x_2})^{\alpha_2/2} +\ldots +(-\partial^2_{x_d})^{\alpha_d/2}$,  tempered operators used in finance like the CGMY model, strongly non-symmetric operators, and operators generating L\'evy processes with no finite moments,
 see Section~\ref{sec:examples} below. While most analytic results on MFGs are formulated on bounded domains \ar{\cm(but 
see e.g. \cite{MR4509653})}, this paper considers the whole space $\mR^d$ without any moment conditions on $m$. 

Note that it is important for us to work without moment conditions to cover important \arr classes of purely nonlocal operators on $\mR^d$. The moments of $m$ are determined by the moments of \arr the initial measure $m_0$ and  the driving noise generated by $\mL$, where the latter is determined by \arr the moments of $\nu$ at infinity \cite{MR3185174}.  As opposed to the local case with Gaussian noise having  finite moments of all orders, processes generated by nonlocal operators may have infinite moments. %
A canonical example is 
 \arr the isotropic $\alpha$-stable process generated by the fractional Laplacian $\mL=-(-\Delta)^{\alpha/2}$, $\alpha\in (0,2)$,  which has infinite moments of all orders $\geq\alpha$.




By a solution to the MFG system \eqref{eq:MFG} we mean a pair $(u,m)$ such that $u$ solves the Hamilton--Jacobi equation in the classical sense and $m\in C([t_0,T],\mP)$ solves the Fokker--Planck equation in the distributional sense. Because we do not assume moments, we work in $\mP$, the space of probability measures, endowed with the Kantorovich--Rubinstein (or bounded-Lipschitz) metric $d_0$, see Section~\ref{sec:d0}. 
The existence and uniqueness of solutions to \eqref{eq:MFG} is given in Theorem~\ref{th:MFGwp}. Despite the similar setting this is a different result than the one in \cite{MR4309434}, where the authors pursued a unified approach for local and nonlocal couplings, which produced two artifacts: the need to consider smooth solutions in the Fokker--Planck equation and the restriction to purely nonlocal operators. Still, many of the results and arguments of \cite{MR4309434} are applicable in our work. \arr We obtain slightly stronger H\"older estimates  for the Hamilton--Jacobi equation than in \cite{MR4309434} by using in full the regularity given by the heat kernel. Similar to \cite{MR4309434} we do not assume any moment conditions on $m_0$ or $\nu$ at infinity.  Instead we use the fact that probability measures $m_0$ and $\nu|_{B(0,1)^c}$ (properly normalized) \arr both have a certain generalized moment. 
We then show that this moment is 
 uniformly bounded for the solution $m(t)$ of the Fokker--Planck equation and hence get 
tightness of $\{m(t): t\in [t_0,T]\}$, a property crucial for compactness arguments. The preservation of the generalized moment for 
MFG systems  in the whole space was observed by Chowdhury, Jakobsen, and Krupski \cite{2021arXiv210406985C}. Note that this is not an issue on the torus or other bounded domains,  as every probability measure there has all the moments. However, other papers  on the whole space adopt some moment assumptions due to using Wasserstein $d_1$ or $d_2$ distances,  or arguments from probability and stochastic calculus, which usually require second moments. \arr Note that the probabilistic approach 
based on 
a McKean--Vlasov FBSDE reformulation of the MFG, 
has been quite successful in dealing with problems on $\mR^d$, see e.g. \cite{MR3091726} or \cite[Chapter~3]{MR3752669} and references therein. 

Our Hamiltonian is separable 
of the form $H(x,Du) - F(x,m)$. Models with more general non-separable Hamiltonians $\widetilde{H}(x,Du,m)$ have been considered in 
e.g. \cite{2021arXiv210503926A,MR4509653,2021arXiv210112362G,MR4686386,2022arXiv220110762M,fracnonsep}. 
On the other hand, our existence and uniqueness result for the MFG system allows $u$ in the Hamiltonian, which to our knowledge is new. The main obstacle in this case is the uniqueness -- with $u$-dependent Hamiltonian it requires a modified version of the Lasry--Lions monotonicity argument, which we give in Lemma~\ref{lem:LL}. Although it is sufficient for uniqueness for the MFG system, Lemma~\ref{lem:LL} does not seem strong enough to get stability results needed for the master equation. Thus, in the further results we switch to Hamiltonians depending only on $x$ and $Du$.

The main result of this work, Theorem~\ref{th:ME}, gives the existence and uniqueness for the master equation associated with the MFG system \eqref{eq:MFG}:
\begin{align}\label{eq:master}
    \begin{cases}
    \partial_t U(t,x,m)\hspace{-10pt} &=-\mL_x U(t,x,m) + H(x,D_xU(t,x,m)) - \int_{\mR^d} \mL_y \dm{U}(t,x,m,y) \, m(dy) - F(x,m)\\ &\quad+ \int_{\mR^d} D_y\dm{U}(t,x,m,y) D_pH(y,D_yU(t,y,m)) \, m(dy)\quad \mathrm{in}\ (0,T)\times \mR^d\times \mP,\\
    U(T,x,m)\hspace{-10pt} &= G(x,m)\quad \mathrm{in}\ \mathbb{R}^d\times \mP.
    \end{cases}
\end{align}
The equation is set in the infinite dimensional space $\mP$ of probability measures. The expression $\dm{U}$ is the derivative in this space; it is somewhat similar to the Gateaux/Fr\'echet derivative, but here we do not have a linear structure, see Section~\ref{sec:dmUdef} for the definition and properties of $\dm{U}$. The master equation contains full information about the MFG systems with different initial times $t_0$ and measures $m_0$.
To be precise, for any point $(t_0,x)$ and measure $m_0$, the solution $U$ of the master equation is given as
\begin{align}U(t_0,x,m_0) := u(t_0,x)\label{eq:U}\end{align}
where $(u,m)$ solves \eqref{eq:MFG} on $[t_0,T]$ with initial measure $m_0$. The function $U$ is sometimes called the master field, or in analogy with the theory of FBSDEs, the decoupling field \cite{MR4214773,MR3753660}.

The master equation was introduced by P.-L. Lions in his lectures at Coll\`ege de France, and the first result about its well-posedness was given by Chassagneux, Crisan and Delarue \cite{2014arXiv1411.3009C}. The key reference for us is the seminal work of Cardaliaguet, Delarue, Lasry, and Lions \cite{MR3967062}, in which the master equation \eqref{eq:master} is studied from a PDE/analytical point of view.
The purpose of the master equation is to provide a way of quantifying how Nash equilibria in finite-player games converge to Nash equilibria in mean field games,
i.e., it facilitates solving the so-called convergence problem. In recent years there have been several works in this direction, see e.g. \cite{MR3967062,MR4135293,MR3753660,2022arXiv220315583D,2022arXiv220307882R}. We also expect this approach to work in the nonlocal case, 
and we will address this problem in a future work using results of the current paper. 


Aside from \cite[Chapter~3]{MR3967062}, the most relevant references using the deterministic setting, are the works on bounded domains, by Ricciardi \cite{MR4420941,2022arXiv220307882R} for the Neumann condition and by Di Persio, Garbelli, and Ricciardi \cite{2022arXiv220315583D} for the Dirichlet condition. Here we also mention Graber and Sircar \cite{2021arXiv211107020J} who investigate the master equation for Cournot MFGs on the half-line, Ambrose and M\'{e}sz\'{a}ros \cite{2021arXiv210503926A} who consider non-separable Hamiltonians and the Sobolev space setting, and Cecchin and Pelino \cite{MR4013871} who work on the whole space with initial measures being a finite sum of Dirac deltas. \arr Recently, Bansil and M\'esz\'aros \cite{2024arXiv240305426B} have given a framework for obtaining well-posedness for the master equation by appropriately shifting the Hamiltonian. 
There are also studies of the master equation in numerous different settings, e.g., with common noise \cite{MR3967062,MR4451309,MR4079435,MR4214776,2021arXiv210112362G}, on finite state spaces \cite{MR4275225,MR4214776}, with a major player \cite{2020arXiv200110406C,MR3851543}, and many more.

Our work is the first study of  master equations with nonlocal diffusion operators. Our results also include the case of mixed local-nonlocal operators. For the MFG systems in general, along with \cite{MR4223351}, it is the first
to study mixed local-nonlocal operators. Whereas \cite{MR4223351} considers one-dimensional diffusions and the nonlocal part is treated as a lower order perturbation, our setting is much more general and works without any dominating local diffusion. 
 It is also the first paper to study local MFG systems on the whole space without any moment conditions on the initial distributions of players.

 On the structural level, the proof of Theorem~\ref{th:ME} follows the program of \cite[Chapter~3]{MR3967062}. The majority of the work goes into establishing the existence and regularity of $\dm{U}$ for $U$ defined in \eqref{eq:U}. The first step is the result on the stability with respect to $m_0$ for the MFG system in Theorem~\ref{th:Lip}, which yields Lipschitz continuity in $m$ for $U$. Then, in order to access and study $\dm{U}$, we linearize the MFG system \eqref{eq:MFG}, by subtracting two systems with different initial measures $m_0$. This motivates the generic linear system result in Theorem~\ref{th:linsyst}. Taking derivatives in the ambient variable $y$ (e.g., $D_y \dm{U}$) corresponds to differentiating the initial condition in the linearized system, and therefore we need to allow some of the data and solutions in Theorem~\ref{th:linsyst} to belong to negative order H\"older spaces (like distributional derivatives of Dirac measures). This is a major technical issue producing strong assumptions, as this lack of regularity needs to be compensated by higher regularity of the other coefficients in the system. Ricciardi \cite{MR4420941} proposed an improvement in this regard, exploiting the fact that some of the coefficients need not be regular uniformly in time. For the short-time existence, Cardaliaguet, Cirant, and Porretta \cite{2020arXiv200110406C} managed to separate the regularity of the coefficients in the linearized system, but because of the forward-backward nature of the problem, their results cannot be applied to arbitrarily long (but finite) time intervals that we study here. We note that while working in the whole space we need to handle many technical issues arising from the fact that the duals of H\"older spaces are much more complicated on unbounded sets. For example, there are functionals with support ``at infinity'' like the Banach limits, see Definition~\ref{ex:Banach}. Furthermore, a new compactness argument is needed in the proof of Theorem~\ref{th:linsyst}, because in the whole space the H\"older spaces are not embedded in their duals and the Arzel\`a--Ascoli approach used in previous works can no longer be used. To overcome this we work in an $L^1$-setting and use the Kolmogorov--Riesz theorem.

\arr We note that the probabilistic approach has been effective for solving master equations with local diffusions in the whole space \cite{MR3753660,2014arXiv1411.3009C}. 
By a 
lifting procedure 
$D_y\dm{U}$ is obtained directly as a Fr\'echet derivative and $\dm{U}$ is no longer needed in the analysis. 
Stochastic calculus then seems to allow for data of lower regularity than we use here, see \cite[Section~5.1.5]{MR3753660}. Some significant
developments of this theory are likely still needed to cover the case of jump diffusions and diffusions without moments. 

The results used for well-posedness of the master equation involve viscous Hamilton--Jacobi and linear parabolic problems in many different settings. In the case of $\mL = \Delta$ and other local operators, there is a well-established extensive literature, e.g., \cite{MR0241822,MR1465184}. Regularity for parabolic equations with nonlocal operators was also intensively studied over the last decades, see e.g., \cite{MR4056997,MR3803717,MR2019032,MR2121115,MR1246036}, but it is much more difficult to find the desired theorems under the right assumptions. Therefore, in this paper we formulate and prove several results for single equations: apart from the classical solution setting (Lemma~\ref{lem:cd}, Theorem~\ref{th:HJ}), which was largely present in \cite{MR4309434}, we also discuss the $L^1$ setting (Lemma~\ref{lem:FPL1}), the probability measure setting (Theorem~\ref{th:FP}), and the setting of the negative order H\"{o}lder spaces (Lemma~\ref{lem:l1cneg}). These results are quite 
general and could be of independent interest.

Due to the excessive amount of technical details and the complicated structure of the proof,  we found that some results in the previous papers require modifications of the proofs or formulations. The details and references are given in Section~\ref{sec:err}.  Of course, this does not affect the high value of the original and sometimes groundbreaking ideas which these works contributed. 
In the work of Cardaliaguet, Delarue, Lasry, and Lions \cite{MR3967062}, which is a basis for many further papers (including this one), the  necessary amendments are limited to  using slightly stronger assumptions and correcting some of the intermediate results by weakening their conclusions. We made a substantial effort to make the arguments rigorous while keeping the assumptions fairly mild.


The paper is organized as follows: in Section~\ref{sec:Holder} we discuss the H\"older spaces, Section~\ref{sec:d0} is devoted to the Kantorovich--Rubinstein metric, in Section~\ref{sec:dmUdef} we introduce the measure derivatives, and in Section~\ref{sec:hk} the operator $\mL$ and its heat kernel. Assumptions are formulated in Section~\ref{sec:assume} and examples are given in Section~\ref{sec:examples}. In Section~\ref{sec:err} we discuss the previous mistakes. Section~\ref{sec:single} contains the results about the single equations. In Section~\ref{sec:mfg} we prove well-posedness of the MFG system. Section~\ref{sec:further} gives further regularity for the MFG system, including the existence of $\dm{U}$. Finally, in Section~\ref{sec:master} we prove the main result -- well-posedness of classical solution of the master equation with nonlocal diffusion.

\subsection{Summary of the main results and contributions.}\label{sec:contributions}
\begin{itemize}
\item 
Well-posedness of the master equation -- Theorem~\ref{th:ME}:
\begin{itemize}
    \item First such result for nonlocal diffusions.
    \item Allowing mixed local-nonlocal operators.
    \item Allowing local diffusion  operators in the whole space without moment conditions.
    \item Rigorous treatment of duals of H\"older spaces in an unbounded domain, new also for the Laplacian.
\end{itemize}
\item Well-posedness of the MFG system -- Theorem~\ref{th:MFGwp}.
\begin{itemize}
    \item  First such result for mixed local-nonlocal operators.
    \item No moment conditions imposed for $m_0$ in the case of local diffusions.
    \item More optimal Schauder estimates and general (non-smooth) $m_0$ compared to \cite{MR4309434}.
    \item Hamiltonians depending on $u$.
\end{itemize}
\item Auxiliary results for Hamilton--Jacobi equations and linear equations with divergence and non-divergence form drift with quite general L\'evy diffusions -- Lemma~\ref{lem:cd}, Lemma~\ref{lem:FPL1}, Theorem~\ref{th:FP}, Theorem~\ref{th:HJ}, Lemma~\ref{lem:l1cneg}.
\item Observing and fixing some  oversights  in  previous work -- Section~\ref{sec:err}.
\end{itemize}



\section{Preliminaries}\label{sec:prelim}
Constants (usually denoted by $c$ or $C$) may change value from line to line. We define the family of cut-off functions $\chi_R\colon \mR^d \to [0,1]$ for $R\geq 1$, as functions satisfying: $\chi_R = 1$ on $B(0,R)$, $\chi_R = 0$ on $B(0,R+1)^c$, and $\|\chi_R\|_{C^2_b(\mR^d)} \leq C$, with $C$ independent of $R$.
In several results we use approximation by convolution with a standard mollifier $\eta$, that is, a non-negative radially decreasing smooth function supported in the unit ball such that $\int \eta = 1$. For $\eps>0$ we denote  $\eta_\eps(x) = \eps^{-d}\eta(x/\eps)$. Note that $\|D\eta_\eps\|_{L^1(\mR^d)} \leq \eps^{-1}$ and $\|D^2\eta_\eps\|_{L^1(\mR^d)} \leq \eps^{-2}$.
Spaces $L^p(\mR^d)$, $p\in [1,\infty]$ are the usual spaces of equivalence classes of measurable functions with the $p$-th power integrable (or essentially bounded for $p=\infty$).

\subsection{H\"older and related function spaces} 
\begin{remark}
    This subsection is quite long and might be skipped on the first reading. We explain the intricacies of the duals of H\"older spaces on unbounded domains. Furthermore, we address in detail the issue of passing to the limit inside $C^{\alpha}$--$(C^{\alpha})^*$ pairings, and for this sake we define iterated spaces, see Lemma~\ref{lem:dmf} and Remark~\ref{rem:dmf} below.
\end{remark}
By $C_b(\mR^d)$ we denote the space of bounded continuous real-valued functions on $\mR^d$. If $k\in \{0,1,2,\ldots\}$, then $C^k_b(\mR^d)$ is the space of functions with continuous and bounded derivatives up to the order $k$ (in particular $C^0_b(\mR^d) = C_b(\mR^d)$). The norm is given by
\label{sec:Holder}
\begin{align*}
    \|u\|_{C^k_b(\mR^d)} = \sum\limits_{|\beta|\leq k} \|\partial^\beta u\|_{\infty}.
\end{align*}
We denote by $C^k_0(\mR^d)$ the space of functions in $C^k_b(\mR^d)$ whose derivatives up to order $k$ vanish at infinity. The space $C^k_c(\mR^d)$ consists of compactly supported functions in $C^k_b(\mR^d)$.

Let $\sigma\in (0,1)$. We define the H\"{o}lder seminorm of a function $u\colon \mR^d\to \mR$ as 
\begin{align*}
    |u|_{C^{\sigma}(\mR^d)} = \sup\limits_{x\neq y} \frac{|u(x) - u(y)|}{|x-y|^{\sigma}}.
\end{align*}
The H\"{o}lder space of order $k+\sigma$ for $k\in \{0,1,2,\ldots\}$ is
\begin{align*}
    C^{k+\sigma}_b(\mR^d) = \big\{u\in C^k_b(\mR^d): |\partial^{\beta}u|_{C^{\sigma}(\mR^d)} <\infty \ {\rm for\ all}\ |\beta|=k\big\}.
\end{align*}
Here, the norm is given by
\begin{align*}
    \|u\|_{C^{k+\sigma}_b(\mR^d)} = \|u\|_{C^k_b(\mR^d)} + \sum\limits_{|\beta|=k}|\partial^\beta u|_{C^\sigma(\mR^d)}.
\end{align*}
It is well-known that H\"{o}lder spaces are Banach spaces with the above norm.\medskip

The negative order H\"{o}lder spaces are defined as the duals, i.e., for $\gamma\geq 0$
    we let $C^{-\gamma}_b(\mR^d) := (C^{\gamma}_b(\mR^d))^*$ and
\begin{align*}
    \|\rho\|_{C^{-\gamma}_b(\mR^d)} = \sup\limits_{\|\phi\|_{C^\gamma_b(\mR^d)}\leq 1} \frac{\langle \rho,\phi\rangle}{\|\phi\|_{C^\gamma_b(\mR^d)}}.
\end{align*}
Note that $C^{-0}_b(\mR^d)\neq C^0_b(\mR^d)$.

\arr Some of the $\rho\in C^{-\gamma}_b(\mR^d)$ can be identified with an integrable function. Indeed, if we let $f\in L^1(\mR^d)$ and
\begin{align}\label{eq:represent}
    \langle \rho, \phi \rangle = \int_{\mR^d} \phi(x) f(x)\, dx,
\end{align}
then $\rho\in C^{-\gamma}_b(\mR^d)$ for any $\gamma\geq 0$ and $\|\rho\|_{C^{-\gamma}_b(\mR^d)}\leq C(\gamma,d) \|f\|_{L^1(\mR^d)}$. This gives an embedding of $L^1(\mR^d)$ into $C^{-\gamma}_b(\mR^d)$. Similar embedding holds for the space of finite signed Borel measures $\mathcal{M}(\mR^d)$ with the total variation norm $\|\cdot\|_{TV}$. Furthermore, if $0\leq\gamma_1\leq\gamma_2$ and $\rho \in C^{-\gamma_1}_b(\mR^d)$, then we have
\begin{align*} \|\rho|_{C^{\gamma_2}_b(\mR^d)}\|_{C^{-\gamma_2}_b(\mR^d)} \leq \|\rho\|_{C^{-\gamma_1}_b(\mR^d)},
\end{align*}
although this does not yield an \textit{embedding} between the spaces, because the restriction map is not injective. Note that in contrast with bounded domains, on the whole space we do not have the \ar{natural} embedding of $C_b$ into $C^{-0}_b$, because the integral $\int_{\mR^d} fg$ for $f,g\in C_b(\mR^d)$ may not exist.

Contrary to the spaces over compact sets, not all functionals in $C^{-\gamma}_b(\mR^d)$ can be represented using integrals against (countably additive) measures. For example, the space $C^{-0}_b(\mR^d)$ can be identified with \textit{finitely} additive set functions, but they are much more difficult to work with, see Dunford and Schwartz \cite[Theorem~IV.6.1.2]{MR1009162}. Note that the space $C^{-\gamma}_b(\mathbb{R}^d)$ does not fall into any family 
of distributions, because the functionals may depend on the behavior of function at infinity.
\begin{definition}
    \label{ex:Banach}We call $L\in C^{-0}_b(\mR)$ a Banach limit, if $L(f) = \lim_{x\to\infty} f(x)$ whenever the limit exists.
    \end{definition}
    The existence of Banach limits follows from the Hahn--Banach theorem. The restriction of a Banach limit to $C_c^\infty(\mR)$ is the zero distribution. Furthermore, a Banach limit cannot be represented as integration against a signed measure.

If $f\in C^{\gamma}_b(\mR^d)$, then we let
$\langle \rho f,\phi\rangle = \langle \rho,f\phi\rangle$ and we have
\begin{align*}
    \|\rho f\|_{C^{-\gamma}_b(\mR^d)} \leq C(\gamma,d)\|\rho\|_{C^{-\gamma}_b(\mR^d)}\|f\|_{C^{\gamma}_b(\mR^d)}.
\end{align*}
For smooth $f\in L^1(\mR^d)$ with all derivatives integrable we let $\overline{f}(x) = f(-x)$ and we define
\begin{align*}
    \langle \rho\ast f,\phi\rangle = \langle \rho, \phi\ast \overline{f}\rangle,
\end{align*}
for all $\phi$ for which the right-hand side makes sense.
\begin{lemma}\label{lem:cnegapprox}
Let $\rho\in C^{-\gamma}_b(\mR^d)$ for some $\gamma \geq 0$ and assume that $f\in C^{\infty}(\mathbb{R}^d)$ with $f,Df,D^2f,\ldots\in L^1(\mR^d)$. Then $\rho\ast f\in C^{-0}_b(\mR^d)$. Furthermore, for $f=\eta_\eps$ we have $\rho\ast\eta_\eps \mathop{\longrightarrow}\limits^{\eps\to 0^+} \rho$ in $C^{-\gamma-\delta}_b(\mR^d)$ for every $\delta > 0$.
\end{lemma}
\begin{proof}
  Let $\phi\in C_b(\mR^d)$. Then
\begin{align*}
      |\langle \rho\ast f,\phi\rangle| = |\langle \rho, \phi\ast \overline{f}\rangle| \leq \|\rho\|_{C^{-\gamma}_b(\mR^d)} \|\phi\ast \overline{f}\|_{C^{\gamma}_b(\mR^d)} \leq \|\rho\|_{C^{-\gamma}_b(\mR^d)} \|\phi\ast \overline{f}\|_{C^{\lceil\gamma\rceil}_b(\mR^d)} \leq \|\rho\|_{C^{-\gamma}_b(\mR^d)}\|\phi\|_{\infty}\sum\limits_{k=0}^{\lceil \gamma\rceil} \|D^kf\|_{L^1(\mR^d)}.
  \end{align*}
  Dividing by $\|\phi\|_{\infty}$ implies that $\rho\ast f\in C^{-0}_b(\mR^d)$. For the second part of the statement let $\phi \in C^{\gamma+\delta}_b(\mR^d)$. Then,
  \begin{align*}|\langle \rho\ast\eta_\eps - \rho,\phi\rangle| = |\langle \rho,\phi\ast \eta_\eps - \phi\rangle| \leq \|\phi\ast \eta_\eps - \phi\|_{C^{\gamma}_b(\mR^d)} \|\rho\|_{C^{-\gamma}_b(\mR^d)}\leq C\eps^\delta \|\phi\|_{C^{\gamma+\delta}_b(\mR^d)}\|\rho\|_{C^{-\gamma}_b(\mR^d)},
\end{align*}
with $C$ independent of $\phi$ and $\eps$. It follows that $\rho\ast\eta_\eps \to \rho$ in $C^{-\gamma-\delta}_b(\mR^d)$ as $\eps\to 0^+$.
\end{proof}
\begin{remark}\label{rem:Banach}
In the above result we cannot infer that $\rho\ast f$ can be represented as integration against a measure. For example if $L$ is a Banach limit, then it is easy to see that $\langle L\ast\eta_\eps, \phi\rangle = \lim_{x\to\infty} \phi(x)$ if the limit exists. Hence, $L\ast\eta_\eps$ is also a Banach limit.
\end{remark}
The existence of functionals like the Banach limit motivates the need to consider a subclass of $C^{-\gamma}_b(\mR^d)$ which can be represented as integrals with respect to countably additive measures. 
\begin{definition}\label{def:MR}
Let $n\in \{0,1,\ldots\}$. We say that $\rho \in C^{-n}_b(\mR^d)$ is measure representable if there exist $\mu_\alpha\in \mathcal{M}(\mR^d)$, $|\alpha|\leq n$, such that for any $\phi\in C^n_b(\mR^d)$ we have
\begin{align}\label{eq:intrep}
    \langle \rho, \phi\rangle  = \sum\limits_{|\alpha|\leq n} \int_{\mR^d} \partial^\alpha \phi(x)\, \mu_\alpha(dx).
\end{align}
We also let
\begin{align*}
    \mathfrak{C}^{-n}_b(\mR^d) = \{\rho \in C^{-n}_b(\mR^d): \rho\ \textrm{is measure representable}\}.
\end{align*}
\end{definition}
\noindent The representation in \eqref{eq:intrep} is in general not unique.
\begin{lemma}\label{lem:intrep}
The class $\mathfrak{C}^{-n}_b(\mR^d)$ is a norm-closed subset of $C^{-n}_b(\mathbb{R}^d)$. Consequently, $(\mathfrak{C}^{-n}_b(\mR^d),\|\cdot\|_{C^{-n}_b(\mR^d)})$ is a Banach space. Furthermore, every $\rho \in \mathfrak{C}^{-n}_b(\mR^d)$ has a representation such that $\|\rho\|_{C^{-n}_b(\mR^d)} = \sum\limits_{|\alpha|\leq n}\|\mu_\alpha\|_{TV}.$
\end{lemma}
\arr Note that in Lemma~\ref{lem:intrep} we do not make a claim about uniqueness of the representation.  The proof is postponed to Appendix~\ref{sec:c0cb}.
\begin{lemma}\label{lem:cnegconv}
Let $\rho\in \mathfrak{C}^{-n}_b(\mR^d)$ for some $n\in \{0,1,\ldots\}$ and assume that $f\in C^{\infty}(\mathbb{R}^d)$ with $f,Df,D^2f,\ldots\in L^1(\mR^d)$. Then $ \rho\ast f\in L^1(\mR^d)\cap C^\infty_b(\mR^d)$ in the sense of \eqref{eq:represent} and $\|\rho\ast f\|_{L^1(\mR^d)} \leq C\|\rho\|_{C^{-n}_b(\mR^d)}$.
\end{lemma}
\begin{proof}
Let $\rho\in \mathfrak{C}^{-n}_b(\mR^d)$ be represented as \eqref{eq:intrep} with $\|\rho\|_{C^{-n}_b(\mR^d)} = \sum\limits_{|\alpha|\leq n} \|\mu_\alpha\|_{TV}$, see Lemma~\ref{lem:intrep}. Then for every $\phi\in L^\infty(\mathbb{R}^d)$ we have
\begin{align*}
    \langle \rho\ast f,\phi\rangle = \sum\limits_{|\alpha|\leq n} \int_{\mR^d} \partial^\alpha (\phi\ast f)(x)\, \mu_\alpha(dx) =  \sum\limits_{|\alpha|\leq n} \int_{\mR^d}  \phi\ast \partial^\alpha f(x)\, \mu_\alpha(dx) = \int_{\mR^d} \phi(x) \sum\limits_{|\alpha|\leq n} (\overline{\partial^\alpha f} \ast \mu_\alpha)(dx).
\end{align*}
By assumptions on $f$, the last measure is absolutely continuous with a smooth integrable density. Furthermore, by the dual characterization of the $L^1$ norm and the fact that $\|\rho\|_{C^{-n}_b(\mR^d)} = \sum\limits_{|\alpha|\leq n} \|\mu_\alpha\|_{TV}$, we get
\begin{align*}
    \|\rho\ast f\|_{L^1(\mR^d)} = \bigg\|\sum\limits_{|\alpha|\leq n} (\overline{\partial^\alpha f} \ast \mu_\alpha)(dx)\bigg\|_{L^1(\mR^d)} \leq \sum\limits_{|\alpha|\leq n} \|\overline{\partial^\alpha f}\|_{L^1(\mR^d)}  \|\mu_\alpha\|_{TV} \leq C\|\rho\|_{C^{-n}_b(\mR^d)}.
\end{align*}
\end{proof}
We will use various types of differentiable/H\"{o}lder functions with two spatial or space-time variables. First, following \cite{MR3967062}, for $m,n\in \{0,1,\ldots\}$, let $C^{m,n}_{b,x,y}(\mR^{2d})$ be the space of functions $u\colon \mR^{2d}\to \mR$, such that for all $0\leq k \leq m$, $0\leq l\leq n$ the derivatives $D^{k}_xD^{l}_y u$ exist and are bounded and continuous, with the norm
\begin{align*}
    \|u\|_{C^{m,n}_{b,x,y}(\mR^{2d})} = \sum\limits_{k=0}^{m}\sum\limits_{l=0}^{n} \|D^{k}_xD^{l}_y u\|_{\infty}.
\end{align*}
   We also consider the space $C^{m+\alpha}_b(\mR^d,C^{n+\beta}_b(\mR^d))$ for $m,n\in \{0,1,\ldots\}$ and $\alpha,\beta\in (0,1)$, endowed with the norm
\begin{align}\label{eq:2varholder}
    \|u\|_{C^{m+\alpha}_b(\mR^d,C^{n+\beta}_b(\mR^d))} = \|u\|_{C^{m,n}_{b,x,y}(\mR^{2d})} + \sum\limits_{k=0}^{m}\sup\limits_{x\in \mR^d}\|D_x^k u(x,\cdot)\|_{C^{n+\beta}_b(\mR^d)} +  \sup\limits_{h\neq 0}\bigg\|\frac{D^{m}_xu(x+h,\cdot) - D^{m}_xu(x,\cdot)}{|h|^{\alpha}}\bigg\|_{C^{n+\beta}_b(\mR^d)}.
\end{align}
%
In this space we can differentiate under functionals of class $(C^{-n-\beta}_{b})_y$ with respect to the parameter $x$. 
 
\begin{lemma}\label{lem:dmf}
If $f\in C^{m+\alpha}_b(\mR^d,C^{n+\beta}_b(\mR^d))$ with $m,n\in \{0,1,\ldots\}$, $\alpha\in(0,1)$, $\beta\in [0,1)$, then for any $\phi\in C^{-n-\beta}_b(\mR^d)$ the map $x\mapsto \phi(f(x,\cdot))$ belongs to $C^{m+\alpha}_b(\mR^d)$ with
\begin{align}\label{eq:phifbound}
    \|\phi(f)\|_{C^{m+\alpha}_b(\mR^d)} \leq \|\phi\|_{C^{-n-\beta}_b(\mR^d)}\|f\|_{C^{m+\alpha}_b(\mR^d,C^{n+\beta}_b(\mR^d))},
\end{align}and we have $D^k_x\phi(f(x,\cdot)) = \phi(D^k_xf(x,\cdot))$ for $1\leq k\leq m$.
\end{lemma}
\begin{proof}
First, note that  for $1\leq k\leq m-1$, $1\leq i\leq d$, and $h\neq 0$ we have
\begin{align*}
    \bigg\|\frac{D^{k-1}_xf(x+he_i,\cdot) - D^{k-1}_xf(x,\cdot)}{h} - \partial_{x_i}D^{k-1}_xf(x,\cdot)\bigg\|_{C^{n+\beta}_b(\mR^d)} &\leq \bigg\|\int_0^1 (D^k_xf(x+\lambda he_i,\cdot) - D^{k}_x f(x,\cdot))\, d\lambda\bigg\|_{C^{n+\beta}_b(\mR^d)}\\ &\leq C|h|.
\end{align*}
For $k=m$ we get a similar estimate with $|h|^{\alpha}$ on the right-hand side.
Therefore, since $\phi\in C^{-n-\beta}_b(\mR^d)$, we obtain
\begin{align*}
    \partial_{x_i}D^{k-1}_x\phi(f(x,\cdot)) = \lim\limits_{h\to 0} \phi\bigg(\frac{D^{k-1}_xf(x+he_i,\cdot) - D^{k-1}_xf(x,\cdot)}{h}\bigg) = \phi(\partial_{x_i}D^{k-1}_xf(x,\cdot)),\quad 1\leq k\leq m,
\end{align*}
where for $k=m$ we used the fact that $\alpha>0$. The fact that $x\mapsto \phi(f(x,\cdot))$ belongs to $C^{m+\alpha}_b(\mR^d)$ with \eqref{eq:phifbound} now follows immediately from the definition of $\|f\|_{C^{m+\alpha}_b(\mR^d,C^{n+\beta}_b(\mR^d))}$.
\end{proof}

\begin{remark}\label{rem:dmf}
In \cite[p. 29]{MR3967062} two variable H\"{o}lder spaces $C^{m+\alpha,n+\beta}$ are defined in a similar way as our $C^{m,n}_{b,x,y}$-spaces by adding the following assumption on the highest order derivatives $D^{m}_xD^{n}_y$:
\begin{align}\label{eq:2holder}
    \sup\limits_{(x_1,y_1)\neq (x_2,y_2)} \frac{|D^{m}_xD^{n}_yu(x_1,y_1) - D^m_xD^n_y u(x_2,y_2)|}{|x_1 - x_2|^\alpha + |y_1-y_2|^\beta}<\infty.
\end{align}
In this space the highest order $y$-H\"{o}lder norms are bounded, but not necessarily continuous in $x$. This is insufficient for the proof of Lemma~\ref{lem:dmf} and one needs to consider $C^{m+\alpha}_b(\mR^d,C^{n+\beta}_b(\mR^d))$ instead.
\end{remark}

For $T_1<T_2$ and $m\in \{1,2,\ldots\}$ we let $C^{m,2m}_b([T_1,T_2]\times \mR^d)$ be the usual parabolic space \cite{MR0241822} of smooth functions with the norm
\begin{align*}
    \|u\|_{C^{m,2m}_b([T_1,T_2]\times \mR^d)} = \sum\limits_{\substack{0\leq k\leq m\\ 0\leq l\leq 2m\\ 2k+l\leq 2m}}\|\partial_t^k D^{l}_x u\|_{\infty}.
\end{align*}
We will also use other types of iterated spaces.  Let $X$ be a Banach space and define
\begin{align*}
    &B([0,T],X) = \{u\colon [0,T]\to X\ | \ \sup\limits_{t\in [0,T]} \|u(t)\|_{X} < \infty\},\\
    &L^\infty([0,T],X) = \{u\colon [0,T] \to X\  | \ u \ {\rm is\  measurable\ and}\ \mathop{{\rm ess\, sup}}\limits_{t\in [0,T]} \|u(t)\|_{X} < \infty\},\\
    &L^p([0,T],X) = \{u\colon [0,T] \to X\  | \ u \ {\rm is\  measurable\  and}\ \int_0^T\|u(t)\|_{X}^p\, dt < \infty\}.
\end{align*}
For the definition of measurability for Banach space-valued functions we refer, e.g., to \cite{MR3617205}. The last two spaces above are sometimes called the Bochner spaces. We use the space $B([0,T],X)$ to indicate that the function is defined and bounded for \textit{every} $t\in[0,T]$.

We recall some results concerning product measurability for $L^p([0,T],L^1(\mR^d))$, $p\in [1,\infty]$. First note that for any $p\in (1,\infty)$ we have
\begin{align}\label{eq:Bochner}
    C([0,T],L^1(\mR^d)) \subseteq L^\infty([0,T],L^1(\mR^d)) \subseteq L^p([0,T],L^1(\mR^d)) \subseteq L^1([0,T],L^1(\mR^d)).
\end{align}
By \cite[Theorems 4.4.1 and 4.4.3]{MR1463946} (see also \cite[Proposition~1.2.25]{MR3617205}), for every $u\in L^1([0,T],L^1(\mR^d))$ there exists $\widetilde{u}\in L^1([0,T]\times \mR^d)$ (in particular, $\widetilde{u}$ is product measurable) such that for all $t\in[0,T]$ we have $\widetilde{u}(t,\cdot) = u(t)$ a.e. Therefore, by identifying $u$ with $\widetilde{u}$ we can use Fubini's theorem for functions in $L^1([0,T],L^1(\mR^d))$, hence for all the functions in spaces appearing in \eqref{eq:Bochner}.

We also note that the pairing of $\rho\in L^1([0,T],C^{-\gamma}_b(\mR^d))$ with $\phi\in L^{\infty}([0,T],C^{\gamma}_b(\mR^d))$ makes sense, in particular, the function $t\mapsto \langle \rho(t),\phi(t)\rangle$ is measurable, see \cite[IV.1]{MR0453964}. Functions of this type will often appear in Section~\ref{sec:further}.

\subsection{Space of probability measures}\label{sec:d0}
\begin{definition}
Let $$\Lip = \bigg\{\phi\in C_b(\mR^d) : \|\phi\|_{\infty}\leq 1,\ \sup\limits_{x,y\in \mR^d}\frac{|\phi(x) - \phi(y)|}{|x-y|} \leq 1\bigg\}.$$
Following \cite[Vol.~2, Section~8.3]{MR2267655},  for $m,m'\in \mP$ we define the Kantorovich--Rubinstein  distance $d_0$ as
$$d_0(m,m') = \sup\limits_{\phi\in \Lip} \bigg|\int_{\mR^d}\phi(x)\, (m'-m)(dx)\bigg|.$$
This metric is often called the bounded-Lipschitz metric, and should not be confused with Wasserstein metric $\mathcal W_1$.\footnote{The $\mathcal W_1$-metric is sometimes also called the Kantorovich--Rubinstein metric. The Kantorovich--Rubinstein duality gives a characterization of $\mathcal W_1$ that is similar to $d_0$, but $\mathcal W_1$ is a strictly stronger metric, requiring also convergence of first moments.} 
Below, whenever we consider $\mP$ as a metric space, we mean $(\mP,d_0)$. \ar{The above definition makes sense also for $m,m'\in \mathcal{M}(\mR^d)$, and $m\mapsto d_0(m,0)$ is a norm in $\mathcal{M}(\mR^d)$.} 
\end{definition}
\begin{remark}\label{rem:d0weak}
\noindent\begin{enumerate}[label=(\alph*)]
    \item Convergence of probability measures in $d_0$ is equivalent to their weak convergence (tested on $C_b(\mR^d)$ functions), see, e.g., Dudley \cite[Chapters 11.2 and 11.3]{MR1932358}. Consequently, by Prokhorov's theorem, tight subsets of $\mP$ are precompact.
\item By \cite[Theorems 4.19 and 17.23]{MR1321597} both $\mP$ and $C([0,T],\mP)$ are complete.
\item In the Schauder-type fixed point arguments one usually needs to have a locally convex topological vector space containing $\mP$. Here, the space $(\mathcal M(\mR^d), d_0)$ shall serve this purpose. We note in passing that $(\mathcal M(\mR^d), d_0)$ is not complete, see \cite[Chapter~8.3]{MR2267655}, but the Schauder--Tychonoff theorem only requires a normed space, see, e.g., \cite[Theorem~10.1]{MR3967045}.
\item Convergence of \textit{signed} measures in $d_0$ does not imply weak convergence or tightness, take for example $m_n = \delta_{n+\frac 1n} - \delta_{n}$.  
\end{enumerate}
\end{remark}
\begin{lemma}\label{lem:molltight}
Let $f$ satisfy the assumptions of Lemma~\ref{lem:cnegconv} and assume that there exist $\gamma\geq 0$ and $n\in \{0,1,\ldots\}$ such that $\rho\in C([0,T],C^{-\gamma}_b(\mR^d))$ and $\rho(t)\in \mathfrak{C}^{-n}_b(\mR^d)$ for every $t\in [0,T]$. Then $\rho\ast f\in C([0,T],L^1(\mR^d))$ and $\{\rho(t)\ast f: t\in [0,T]\}$ is a tight family.
\end{lemma}
\begin{proof}
  Since $\rho(t)\in \mathfrak{C}^{-n}_b(\mR^d)$, by Lemma~\ref{lem:cnegconv} we get that $\rho(t)\ast f\in L^1(\mR^d)$ for all $t\in[0,T]$. Furthermore, for every $s,t\in[0,T]$ and $\phi\in L^\infty(\mR^d)$,
  \begin{align*}
      |\langle \rho(t)\ast f - \rho(s)\ast f,\phi\rangle| \leq \|\rho(t) - \rho(s)\|_{C^{-\gamma}_b(\mR^d)}\|\phi\ast \overline{f}\|_{C^\gamma_b(\mR^d)} \leq C(f)\|\rho(t) - \rho(s)\|_{C^{-\gamma}_b(\mR^d)}\|\phi\|_{L^\infty(\mR^d)},
  \end{align*}
  thus from $\rho \in C([0,T],C^{-\gamma}_b(\mR^d))$ we get $\rho\ast f\in C([0,T],L^1(\mR^d))$. Since $\{\rho(t)\ast f: t\in [0,T]\}$ is a continuous image of a compact interval, it is compact in $L^1(\mR^d)$, hence by Kolmogorov--Riesz theorem \cite[Theorem~5]{MR2734454} we get its tightness.
\end{proof}
The following result allows us to quantify the notion of tightness, with the help of generalized moments. \ar{The general idea is quite well-known, see e.g., \cite[Example~8.6.5]{MR2267655}, but the form presented here obtained in \cite[Section~4.1]{2021arXiv210406985C}, allows to use the function $\psi$ below as a test function in the Fokker--Planck equation. This is crucial for proving preservation of tightness in Lemma~\ref{lem:FPL1} and Theorem~\ref{th:FP}.}
\begin{lemma}\label{lem:tight}
A family $\mathcal W$ of probability measures  is tight if and only if there exists a constant $C>0$ and a non-negative subadditive function $\psi\in C^2(\mR^d)$ such that $\|D\psi\|_{\infty}$, $\|D^2\psi\|_{\infty}<\infty$, $\lim\limits_{|x|\to\infty} \psi(x) = \infty$ and
$$\int_{\mR^d} \psi(x)\, m(dx) \leq C,\quad m\in \mathcal W.$$
\end{lemma}
\begin{lemma}\label{lem:moll}
Let $m\in \mP$ and denote $m_\eps = m\ast \eta_\eps$, where $\eta$ is a standard mollifier. Then for every bounded uniformly continuous function $\phi$ with the modulus of continuity $\omega_\phi$ we have
$$\bigg|\int_{\mR^d} \phi(x)\, (m_\eps - m)(dx)\bigg| \leq \omega_\phi(\eps).$$
As a consequence,
$$\lim\limits_{\eps\to 0^+} d_0(m\ast \eta_\eps,m) = 0.$$
\end{lemma}
\begin{proof}
Let $\phi$ be bounded and uniformly continuous. We have
\begin{align*}
    \bigg|\int_{\mR^d} \phi(x)\, (m_\eps - m)(dx)\bigg| &= \bigg|\int_{\mR^d} \phi(x) \int_{\mR^d} \eta_\eps(x-y)\, m(dy)\, dx - \int_{\mR^d} \phi(y)\, m(dy)\bigg|\\
    &=\bigg|\int_{\mR^d} \phi(x) \int_{\mR^d} \eta_\eps(x-y)\, m(dy)\, dx - \int_{\mR^d} \phi(y)\int_{\mR^d} \eta_\eps(x-y)\, dx\, m(dy)\bigg|\\
    &\leq \int_{\mR^d}\int_{\mR^d}|\phi(x) - \phi(y)|\eta_\eps(x-y)\, dx\, m(dy)\\
    &\leq \int_{\mR^d}\int_{B(x,\eps)}\omega_\phi(|x-y|) \eta_\eps(x-y)\, dx\, m(dy) \leq \omega_\phi(\eps).
\end{align*}
Since $\omega_{\phi}(\eps)\leq \eps$ for every $\phi\in \Lip$, we get the convergence in $d_0$.
\end{proof}
The following standard result follows from Riesz' representation theorem and mollification.
\begin{lemma}\label{lem:d0test}
Let $m,m'\in\mP$ and assume that for every $\phi\in C_c^\infty(\mR^d)$ we have $\int_{\mR^d} \phi (m-m') = 0$. Then $m=m'$.
\end{lemma}
\begin{lemma}\label{lem:testapprox}
Let $m,m'\in \mP$ and assume that there exists $C>0$ such that for every $\phi\in C_b^{\infty}(\mR^d)\cap \Lip$,
\begin{align*}
    \bigg|\int_{\mR^d}\phi (m-m')\bigg| \leq C.
\end{align*}
Then $d_0(m,m')\leq C$.
\end{lemma}
\begin{proof}
  Let $\phi \in \Lip$, $\eps>0$, and $\phi_\eps = \phi\ast\eta_\eps\in C_b^{\infty}(\mR^d)\cap \Lip$. By the dominated convergence theorem,
  \begin{align*}
      \bigg|\int_{\mR^d} \phi (m-m')\bigg| =  \bigg|\int_{\mR^d} \lim\limits_{\eps\to 0^+}\phi_\eps (m-m')\bigg| = \lim\limits_{\eps\to 0^+}\bigg|\int_{\mR^d} \phi_\eps (m-m')\bigg| \leq C.
  \end{align*}
\end{proof}

\subsection{Derivative in the space of probability measures}\label{sec:dmUdef}
\begin{definition}\label{def:UC1}
Let $U\colon \mP\to \mR$. We say that $U$ is $C^1$ if there exists a continuous mapping $\dm{U}\colon \mP\times \mR^d \to \mR$ such that $\dm{U}(m,\cdot)$ is bounded for all $m\in\mP$ and the following relation holds for every $m,m' \in \mP$:
$$\lim\limits_{h\to 0^+} \frac{U(m+ h(m'-m)) - U(m)}{h} = \int_{\mR^d} \dm{U}(m,y)\, (m'-m)(dy).$$
\end{definition}
\begin{remark}\ \label{rem:normalization}

\begin{enumerate}[label=(\alph*)]
\item From the above relation we see that $\dm{U}$ is not defined uniquely, as we may add to it any constant. Nonetheless, we \textbf{do not use} the normalization condition appearing commonly in the literature:
\begin{align}\label{eq:normalization}
    \int_{\mR^d} \dm{U}(m,y)\, m(dy) = 0,\quad m\in \mP.
\end{align}
In many of the cases below, $\dm{}$ is either given explicitly, or predetermined. If it is not explicit, then by saying that $\dm{}$ exists and has certain properties, we mean that there exists a version of $\dm{}$ which has these properties. We drop the normalization as it interferes with the monotonicity condition for some natural examples of $F$ and $G$, see Remark~\ref{rem:monot} below.

\item If $\Phi_1$ and $\Phi_2$ satisfy the definition of $\dm{U}$ for all $m,m'$ and $y$, then by taking  \ar{$m' = \delta_{y_1}$  
and $m' = \delta_{y_2}$ and subtracting} we find that
\begin{align*}
    \Phi_1(m,y_1) - \Phi_1(m,y_2) = \Phi_2(m,y_1) - \Phi_2(m,y_2),\quad m\in\mP,\ y_1,y_2\in\mR^d.
\end{align*}
Therefore, the existence and values of $D_y\dm{U}, D^2_{yy}\dm{U}, \mL_y\dm{U}$, and $\mL^\ast_y\dm{U}$ do not depend on choice of $\dm{U}$. Note that this includes all the expressions appearing in the master equation.
\end{enumerate}
\end{remark}
\begin{lemma}\label{lem:fund} If $U\in C^1$ then for every $m,m'\in\mP$,
\begin{equation}\label{eq:fund}
    U(m') - U(m) = \int_0^1\int_{\mR^d} \dm{U}(\lambda m+(1-\lambda)m',y)\, (m'-m)(dy)\, d\lambda.
\end{equation}
\end{lemma}
\begin{proof}
For $\lambda\in [0,1]$ let $\Phi(\lambda) = U(m+\lambda(m'-m))$. If $\lambda\in(0,1)$, then by the definition of $\dm{U},$
\begin{align*}\Phi'(\lambda) &= \lim\limits_{h\to 0} \frac{U(m+\lambda(m'-m) + h(m'-m)) - U(m+\lambda(m'-m))}{h}=\int_{\mR^d} \dm{U}(m+\lambda(m'-m),y)\, (m'-m)(dy).
\end{align*}
The statement follows from the fundamental theorem of calculus and the fact that $U(m') - U(m) = \Phi(1) - \Phi(0).$
\end{proof}
The following result has been stated in the $d_1$ setting in \cite[p. 31]{MR3967062}. The proof for $d_0$ is identical.
\begin{lemma}\label{lem:dmfacts}
Assume that $\dm{U}$ and $D_y\dm{U}$ are bounded. Then $U$ is Lipschitz with respect to $m$ and 
\begin{align*}
|U(m_1) - U(m_2)|\leq \sup\limits_{m\in \mP}\bigg\|\dm{U}(m,\cdot)\bigg\|_{C^1_b(\mR^d)}d_0(m_1,m_2). 
\end{align*}
\end{lemma}
In the next lemma we discuss interchanging space and measure derivatives.
\begin{lemma}\label{lem:Schwarz}
 Let $U\colon \mR^d\times \mP \to \mR$. Assume that $D_xU$, $\dm{U}$ and $D_x\dm{U}$ exist and are continuous and bounded in all variables.
Then $\dm{}D_xU$ exists and we can set $$\dm{}D_xU(x,m,y)= D_x\dm{U}(x,m,y),\quad x,y\in \mR^d,\, m\in \mP.$$
\end{lemma}
\begin{proof}
 Let $x,x'\in \mR^d$ and $m,m'\in \mP$. We follow the approach of Rudin for two real variables \cite[Theorem~9.41]{MR0385023}. By Lemma~\ref{lem:fund} and the mean value theorem
\begin{align*}
    &U(x',m') - U(x',m) - U(x,m') + U(x,m)\\ = &\int_0^1\int_{\mR^d} \bigg[\dm{U}(x',m + \lambda(m'-m),y) - \dm{U}(x,m + \lambda(m'-m),y)\bigg]\, (m'-m)(dy)\, d\lambda\\
    = &\,(x'-x)\cdot\int_0^1\int_{\mR^d} D_x\dm{U}(\xi,m + \lambda(m'-m),y)\, (m'-m)(dy)\, d\lambda,
\end{align*}
where $\xi$ belongs to the line segment connecting $x$ and $x'$. By subtracting
$$(x'-x)\cdot\int_0^1\int_{\mR^d} D_x\dm{U}(x,m + \lambda(m'-m),y)\, (m'-m)(dy)\, d\lambda$$
from both sides and by using the continuity and boundedness of $D_x\dm{U}$ and the dominated convergence theorem, we find that the expression
\begin{align*}
    \bigg|U(x',m') - U(x',m) - U(x,m') + U(x,m) - (x'-x)\cdot\int_0^1\int_{\mR^d} D_x\dm{U}(x,m + \lambda(m'-m),y)\, (m'-m)(dy)\, d\lambda\bigg|
\end{align*}
is $o(|x-x'|)$ when $x'\to x$.
Therefore
\begin{align*}
    D_xU(x,m') - D_xU(x,m) = \int_0^1\int_{\mR^d} D_x\dm{U}(x,m + \lambda(m'-m),y)\, (m'-m)(dy)\, d\lambda.
\end{align*}
Thus, by taking $m' = m+h(\widetilde{m} - m)$, using the boundedness and continuity in $m$ of $D_x\dm{U}$, and the dominated convergence theorem, we find that 
\begin{align*}
    \lim\limits_{h\to 0^+} \frac{D_xU(x,m+h(\widetilde{m} - m)) - D_xU(x,m)}{h} = \int_0^1 D_x\dm{U}(x,m,y)\, (\widetilde{m} - m)(dy).
\end{align*}
This implies that $\dm{} D_x U$ exists and can be taken to be $D_x \dm{U}$.
\end{proof}

\subsection{Operator $\mL$ and its heat kernel}\label{sec:hk} Let $\mL$ be the L\'evy operator given in \eqref{eq:operator}. It is well-known (see, e.g., Khoshnevisan and Schilling \cite[Theorem~6.8]{MR3587832}) that $\mL$ can be represented as a Fourier multiplier $\Psi$, i.e., for all $u\in C_c^{2}(\mR^d)$ and $x\in \mR^d$,
\begin{align*}
    \mL u(x) = -\mathcal{F}^{-1} (\Psi \mathcal{F}u)(x),
\end{align*}
where $\mathcal{F}$ is the Fourier transform. The multiplier $\Psi$ is also called the characteristic exponent or the symbol. By the L\'evy--Khintchine representation theorem we have
\begin{align*}
    \Psi(\xi) = -i B\xi + \xi A\xi + \int_{\mR^d} (1 - e^{i\xi z} + i\xi z \textbf{1}_{B(0,1)}(z))\, \nu(dz), \quad \xi\in \mR^d.
\end{align*}
The adjoint operator has symbol \arr
$\overline{\Psi}(\xi) = i B\xi + \xi A\xi + \int_{\mR^d} (1 - e^{i\xi z} + i\xi z \textbf{1}_{B(0,1)}(z))\, \nu(-dz),$
therefore if we let $\widetilde{\nu}( E) = \nu(- E)$ for Borel sets $ E$,  we get that for $u\in C^2_c(\mR^d)$,
\begin{equation*}
    \mL^\ast u (x) = -B\cdot Du(x) + \Div(A Du(x))+ \int_{\mR^d}(u(x+z) - u(x) - Du(x)\cdot z \textbf{1}_{B(0,1)}(z))\, \widetilde{\nu}(dz),
\end{equation*} 
and for $u,v\in C_c^2(\mR^d)$,
\begin{align*}
    \int_{\mR^d} u \mL v = \int_{\mR^d} \mL^* u v.
\end{align*}
By considering a smooth cut-off function we can extend the above formula to $u\in C^2_b(\mR^d)$ and $v\in C^2_c(\mR^d)$ (note that $\mL v\in L^p(\mR^d)$ for any $p\geq 1$). Furthermore, by using the cut-off once more, we may show that it also holds if $u\in C^2_b(\mR^d)$, $v\in C^2_0(\mR^d)$ and $v,\mL v\in L^1(\mR^d)$.

For functions in $C^2_0(\mR^d)$, operator $\mL$ agrees with the generator of the semigroup corresponding to the transition probabilities of L\'evy process $X_t$ with L\'evy triplet $(B,A,\nu)$, see \cite[Theorem~31.5]{MR3185174}. The transition probability is given by the formula
\begin{align*}
    K(0) = \delta_0,\qquad K(t,x)\, dx = (\mathcal{F}^{-1}e^{-t\Psi})(x)\, dx,\quad t>0,\ x\in \mR^d.
\end{align*}
We will sometimes use the notation $K_t(x) = K(t,x)$. Recall that $K(t,\cdot)\, dx$ is a probability measure for every $t\geq 0$. We \textit{assume} that $K(t,\cdot)\in C^{\infty}_0(\mR^d)$ for every $t>0$\footnote{This is implied by assumption \ref{eq:K} below. Note that for general L\'evy operators $K(t,\cdot)$ may be a singular measure, e.g., when the L\'evy triplet equals $(0,0,\delta_{x_0})$ or $(B,0,0)$.}, which implies that $K(t,\cdot)$ is in the domain of the generator, so we can infer that $K$ solves the heat equation for $\mL$ pointwise:
\begin{align*}
    \partial_t K(t,x) = \mL K(t,x),\quad t>0,\ x\in\mR^d.
\end{align*}
In addition we have $\|K_t\ast f - f\|_{\infty} \to 0$ as $t\to 0^+$, for every $f\in C_0(\mR^d)$ --- this is the strong continuity of the semigroup corresponding to $X_t$. From now on $K$ will be called the heat kernel of $\mL$. Note that the heat kernel $K_t^\ast$ of $\mL^*$ satisfies $K_t^\ast(x) = K_t(-x)$. Recall that if $\mL = \Delta$, then
\begin{align*}
    K_t(x) = \frac{1}{(4\pi t)^{d/2}} e^{-\frac{|x|^2}{4t}},\quad t>0,\ x\in \mR^d.
\end{align*}
For most operators $\mL$ the heat kernel does not have an explicit formula.

\subsection{Assumptions}\label{sec:assume} 
\subsubsection{The Hamiltonian} The set of assumptions on the Hamiltonian $H$ is fairly standard, see \cite{MR4214773,MR4309434,Lions}.
\begin{description}
    \item[(H1)\label{eq:H}]$H\colon \mR^d\times\mR\times \mR^d\to \mR$ is smooth and for every $l\in \mathbb{N}^{2d+1}$ with $|l|\leq 4$, $\sup\limits_{x\in \mR^d}|D^lH(x,\cdot,\cdot)|$ is locally bounded.
\item[(H2)\label{eq:H2}] For every $R>0$ there exists $C_R>0$ such that for $x,y\in \mR^d$, $|u|\leq R$, and $p\in \mR^d$,
$$|H(x,u,p) - H(y,u,p)|\leq C_R(1+|p|)|x-y|.$$
\item[(H3)\label{eq:H3}] There exists $\gamma\in \mR$ such that for all $x\in \mR^d$, $u\leq v$, and $p\in\mR^d$,
$$H(x,v,p) - H(x,u,p) \geq \gamma (v-u).$$
\end{description}
\medskip
\subsubsection{The operator} The assumption on the diffusion operator $\mL$ is given in terms of its heat kernel.
\begin{description}[labelwidth=29pt]
\item[(K)\label{eq:K}]
         There is $\mathcal K > 0$ and $\alow\in (1,2]$, such that the heat kernels $K$ and $K^*$ of $\mL$ and $\mL^*$ respectively are smooth densities  of probability measures, and for $\tilde K = K,K^*$ and $\beta \in \mathbb{N}^d$ we have
        $$\|D^{\beta}\tilde K(t,\cdot)\|_{L^1(\mR^d)} \leq \mathcal Kt^{-\frac {|\beta|}{\alow} }.$$
\end{description}
\begin{remark}\ \label{rem:K}

\begin{enumerate}[label=(\alph*)]
\item Assumption \ref{eq:K} implies that $K(t,\cdot)\in W^{n,1}(\mR^d)$ for all $t>0$ and $n\in \mathbb{N}$. From the Sobolev embeddings it follows that $D^\beta K(t,\cdot) \in L^p(\mR^d)$ for all $t>0$, $\beta\in \mathbb{N}^d$, and $1<p\leq \infty$, and as a consequence $K(t,\cdot)\in C_0^\infty(\mR^d)$ for all $t>0$.
\item We note that $\|\mL u\|_{L^1(\mR^d)} \leq C(\mL)(\|u\|_{L^1(\mR^d)} +\|D^2u\|_{L^1(\mR^d)})$ and $\|\mL u\|_{\infty} \leq C(\mL)\|u\|_{C^2_b(\mR^d)}$ hold for any L\'evy operator $\mL$. The condition \cite[L2(i)]{MR4309434} yields a more precise bound on $\mL u$ (see Lemma~2.1 therein), but we do not adopt this assumption, as the only gain from it are slightly better bounds for time regularity, which are not needed in our work.
\item We will often use integration by parts in integrals involving $K$. Since $K(t)$ and $D_yK(t)$ are integrable and smooth with all derivatives bounded, we get that $\int_{\partial B_r} K(t)$ vanishes at infinity as a function of $r$. Therefore, for any $\phi \in C^1_b(\mR^d)$,
\begin{align*}
    \int_{\mR^d} D_yK(t,x-y) \phi(y)\, dy = \int_{\mR^d} K(t,x-y) D_y\phi(y)\, dy. 
\end{align*}
\item We can interpolate the bounds on derivatives in \ref{eq:K} to get similar bounds for H\"older quotients. In order to see that, take $f\in C^1(\mR^d)$ with both $f$ and $Df$ in $L^p(\mR^d)$ and note that for any $\sigma\in(0,1)$ and $h\in \mR^d\setminus\{0\}$ we have
\begin{align*}
    &\int_{\mR^d} \frac{|f(x+h) - f(x)|^p}{|h|^{p\sigma}}\, dx \leq |h|^{p(1-\sigma)}\int_0^1\int_{\mR^d}|Df(x+\lambda h)|^p\, dx\, d\lambda = |h|^{p(1-\sigma)}\|Df\|_{L^p(\mR^d)}^p,\\
    &\int_{\mR^d} \frac{|f(x+h) - f(x)|^p}{|h|^{p\sigma}}\, dx \leq 2^p|h|^{-p\sigma}\|f\|_{L^p(\mR^d)}^p,
\end{align*}
therefore
\begin{align*}
    &\sup\limits_{|h|>0}\int_{\mR^d} \frac{|f(x+h) - f(x)|^p}{|h|^{p\sigma}}\, dx \leq \sup\limits_{|h|>0}\big(|h|^{p(1-\sigma)}\|Df\|_{L^p(\mR^d)}^p\wedge  2^p|h|^{-p\sigma}\|f\|_{L^p(\mR^d)}^p\big) = 2^{p(1-\sigma)}\|f\|_{L^p(\mR^d)}^{p(1-\sigma)}\|Df\|_{L^p(\mR^d)}^{p\sigma}.
\end{align*}
\end{enumerate}
\end{remark}
\medskip
\subsubsection{Functions $F$ and $G$}
In the following assumptions we allow $k \in \{0,1,\ldots\}$ and $\sigma \in [0,1)$, but for the sake of the main results we only need $k=1$ and a fixed $\sigma \in (0,\alow-1)$.
\begin{description}[labelwidth=50pt]
\item[(F1)\label{eq:F}] $F\colon \mR^d\times \mP\to \mR$ satisfies
$$\sup\limits_{m\in\mP}\|F(\cdot,m)\|_{C^2_b(\mR^d)} < \infty,$$
$$\sup\limits_{x\in\mR^d,\, m\neq m'} \frac{|F(x,m)-F(x,m')|}{d_0(m,m')} <\infty.$$
\item[(F2(k,$\sigma$))\label{eq:F2}] There exists $C>0$ such that for all $m,m'\in \mP$,
\begin{align*}
\bigg\|\dm{F}(\cdot,m,\cdot) \bigg\|_{C^{k+1+\sigma}_b(\mR^d,C^{k+1+\sigma}_b(\mR^d))} &\leq C,\\
    \bigg\|\dm{F}(\cdot,m,\cdot) - \dm{F}(\cdot,m',\cdot)\bigg\|_{C^{k+1+\sigma}_b(\mR^d,C^{k+1+\sigma}_b(\mR^d))} &\leq Cd_0(m,m').
\end{align*}
\item[(G1($\sigma$))\label{eq:G}] $G\colon \mR^d\times \mP\to \mR$ satisfies
$$\sup\limits_{m\in\mP}\|G(\cdot,m)\|_{C^{3+\sigma}_b(\mR^d)} < \infty,$$
$$\sup\limits_{x\in\mR^d,\, m\neq m'} \frac{|G(x,m)-G(x,m')|}{d_0(m,m')} <\infty.$$
\item[(G2(k,$\sigma$))\label{eq:G2}]There exists $C>0$ such that for all $m,m'\in \mP$,
\begin{align*}
\bigg\|\dm{G}(\cdot,m,\cdot) \bigg\|_{C^{k+2+\sigma}_b(\mR^d,C^{k+1+\sigma}_b(\mR^d))} &\leq C,\\
    \bigg\|\dm{G}(\cdot,m,\cdot) - \dm{G}(\cdot,m',\cdot)\bigg\|_{C^{k+2+\sigma}_b(\mR^d,C^{k+1+\sigma}_b(\mR^d))} &\leq Cd_0(m,m').
\end{align*}
\end{description}
\medskip
\subsubsection{Monotonicity conditions} 
\begin{description}
\item[(M1)\label{eq:M1}] The Lasry--Lions monotonicity condition \cite{MR2295621} holds for $F$ and $G$, that is, for all $m,m'\in\mP$,
$$\int_{\mR^d} (F(x,m') -F(x,m))(m'-m)(dx) \geq 0,$$
$$\int_{\mR^d} (G(x,m') -G(x,m))(m'-m)(dx) \geq 0.$$
\item[(M2)\label{eq:M12}]
\ref{eq:F2} and \ref{eq:G2} hold and for every $\rho\in C^{-k-1-\sigma}_b(\mR^d)$ and $m\in \mP$ we have
\begin{align*}
    &\bigg\langle \big\langle \dm{F(\cdot,m,\cdot)}, \rho\big\rangle_{y},\, \rho\bigg\rangle_x \geq 0,\\
    &\bigg\langle \big\langle \dm{G(\cdot,m,\cdot)}, \rho\big\rangle_{y},\, \rho\bigg\rangle_x \geq 0,
\end{align*}
where $\langle\cdot,\cdot\rangle_x,\langle\cdot,\cdot\rangle_y$ are the pairings between $C^{k+1+\sigma}_b(\mR^d)$ and $C^{-k-1-\sigma}_b(\mR^d)$ in $x$ and $y$ respectively. 
\end{description}

\begin{remark}\label{rem:monot}
The assumption \ref{eq:M12} is stronger than \ref{eq:M1}, in fact, strictly stronger, if the normalization condition \eqref{eq:normalization} is assumed. In order to show that \ref{eq:M12} implies \ref{eq:M1}, it suffices to take $\rho = m'-m$ and use the fundamental theorem of calculus of Lemma~\ref{lem:fund}. \arr Though it seems plausible, it is unclear to us if there is a universal way to pick/normalize $\dm{F}$ so that \ref{eq:M1} implies \ref{eq:M12}. The results of Sections~\ref{sec:further} and \ref{sec:master} could be modified to only require $\rho$ with mean zero in the sense that $\langle \rho, 1\rangle = 0$, but we were unable to prove that \ref{eq:M1} implies \ref{eq:M12} for such $\rho$  -- even if we assume they are 
nice (e.g. measures with smooth densities and integral 0).
 
\end{remark}
\begin{example}\label{ex:monot} Let $F(x,m) = \phi\ast m(x)$ for some nontrivial odd function $\phi\in C_c^\infty(\mR^d)$. Then it is easy to verify that 
\begin{align*}
    \int_{\mR^d} (F(x,m_1) - F(x,m_2))\, (m_1 - m_2)(dx) = 0,\quad m_1,m_2\in \mP,
\end{align*}
so \ref{eq:M1} is satisfied. Under the normalization condition \eqref{eq:normalization} we have (compare that with \eqref{eq:dmF})\footnote{Accordingly, the derivative in \cite[Example 1.1]{MR4191529} should read $\dm{U}(m)(y) = h(y) - \int h(y)\, m(dy)$.}
\begin{align*}
    \dm{F}(x,m,y) = \phi(x-y) - \phi\ast m(x).
\end{align*}
Let $m=\delta_0$ and $\rho = \delta_{x_0}$ for some $x_0\neq 0$. Then 
\begin{align*}
    \bigg\langle \big\langle \dm{F(\cdot,m,\cdot)}, \rho\big\rangle_{y},\, \rho\bigg\rangle_x = \int_{\mR^d}\int_{\mR^d} \dm{F}(x,m,y)\, \delta_{x_0}(dy)\, \delta_{x_0}(dx) = \dm{F}(x_0,m,x_0) = \phi(0) - \phi\ast m(x_0) = \phi(0) - \phi(x_0).
\end{align*}
Since $\phi$ is odd we have $\phi(0)= 0$, and since it is non-trivial, the above expression does not have a constant sign for all $x_0$. Hence \ref{eq:M12} is not satisfied.
\end{example}
\begin{description}
\item[(M3)\label{eq:M2}] $H(x,u,p) = H(x,p)$ and there exists $c_1\geq 1$ such that for all $x\in \mR^d$
$$\frac 1 {c_1} I_d \leq D^2_{pp}H(x,\cdot) \leq c_1 I_d.$$
\item[(M4)\label{eq:M3}] $H(x,u,p) = H_1(x,p) + H_2(x,u)$ and there exist $c_1,c_2>0$ such that for all $x\in \mR^d$,
$$\frac 1 {c_1} I_d \leq D^2_{pp}H_1(x,\cdot) \leq c_1 I_d,$$
and
$$0\leq D_uH_2(x,\cdot) \leq c_2.$$
\end{description}
\begin{remark}
Assumption \ref{eq:M2} is strictly stronger than \ref{eq:M3}. The latter yields a weaker variant of the Lasry--Lions monotonicity lemma (Lemma~\ref{lem:LL}) which suffices for obtaining uniqueness for the MFG system, but seems ineffective in terms of stability results. That is why at some point in the paper we need to adopt \ref{eq:M2}, which implies the more standard version of the monotonicity lemma (see Remark~\ref{rem:LL}).
\end{remark}
\subsection{Examples}\label{sec:examples}
\subsubsection{The operator $\mL$}
Recall that $\mL$ is determined by its L\'evy triplet $(A,B,\nu)$. Here are some concrete assumptions on operators or the L\'evy triplets, under which \ref{eq:K} holds, taken mostly from \cite[Section~4]{MR4309434}. 
\begin{itemize}
    \item If $\mL$ satisfies \ref{eq:K} and $L$ is any L\'evy operator, then $\mL + L$ satisfies \ref{eq:K}. \footnote{Hence it suffices to specify the assumptions only on $A$ or $\nu$, or even a part of $\nu$. The remaining terms in the L\'evy triplet can be arbitrary, in particular they can vanish or be degenerate.}
    \item The diffusion matrix $A$ satisfies $\langle Ax,x\rangle \geq |x|^2$ for $x\in \mR^d$, e.g., $\mL = \Delta$. Then \ref{eq:K} is satisfied with $\alow=2$.\footnote{This was not discussed in \cite{MR4309434}, but can easily be checked by hand for the Gaussian heat kernel $\mathcal{G}_t$. For nondegenerate $A$, the heat kernel corresponding to $\Div (ADu)$ is equal to $K_t(x) = |A|^{-1}\mathcal{G}_t(Ax)$, so the estimate follows from the Gaussian case.}
    \item The L\'evy measure $\nu$ has an absolutely continuous part with density $\nu_{ac}$ such that $\nu_{ac}(z) \approx |z|^{-d-\alow}$ for all $|z|\leq 1$ and $\alow\in (1,2)$, see \cite[Theorem~4.3]{MR4309434}\footnote{The proof of this result is based on the heat kernel estimates in \cite[Theorem~5.2]{MR4308627}.}.
    \item $\mL  = -(-\partial^2_{x_1})^{\alpha_1/2}-(-\partial^2_{x_2})^{\alpha_2/2} -\ldots -(-\partial^2_{x_d})^{\alpha_d/2}$, where $\alpha_i\in (1,2)$. Then $\alow = \min_i \alpha_i$.
    \item The Riesz--Feller operator on $\mR$: $\nu(z) = |z|^{-1-\alow}\textbf{1}_{(0,\infty)}(z)$, see \cite[Lemma~2.1 (G7) and Proposition~2.3]{MR3360395}.
    \item The CGMY model operator, see \cite[Example~4.4]{MR4309434}.
\end{itemize}
\subsubsection{The coupling terms $F$ and $G$}
We will only discuss $F$, as the assumptions for $G$ are identical up to the order of differentiation. We now give two (slightly) different examples. Define
\begin{align}\label{eq:Fexample}
  &\begin{cases}
    F_1(x,m) = \phi_1\ast m(x),\quad x\in \mR^d,\ m\in \mP, \\[0.2cm]
    \qquad\text{for  $\phi_1\in C_c^\infty(\mR^d)$ positive semi-definite (see Definition~\ref{def:posdef} below).}
    \end{cases}\\
\intertext{\arr Examples similar to $F_1$ are studied in \cite[Chapter~3, Example~5]{MR3752669} in connection with \ref{eq:M1}. Following \cite{MR3967062}, we let}
    &\label{eq:Fexample2}\begin{cases}
        F_2(x,m) = \int_{\mR^d} \Phi(z,(\phi_2\ast m)(z))\, \phi_2(x-z)\, dz,\quad x\in \mR^d,\ m\in\mP, \\[0.2cm]
    \qquad\text{for  $\phi_2\in C_c^\infty(\mR^d)$ non-negative and even, and } \\ \qquad\quad\text{$\Phi\colon \mR^d\times \mR\to \mR$ smooth,  $\Phi(z,s)$ and $\partial_s\Phi(z,s)$ are bounded in $z$, locally uniformly in $s$, and $\partial_s \Phi(z,s)\geq 0$.}
    \end{cases}
    \end{align} 
    
We now check that our assumptions are satisfied. Starting with $F_1$ in \eqref{eq:Fexample} we  note that
\begin{align}\label{eq:dmF}
    \cm\dm{F_1}(x,m,y) = \phi_1(x-y),\quad x,y\in \mR^d,\ m\in \mP,
\end{align}
satisfies Definition \ref{def:UC1}. It is then obvious that \ref{eq:F} and \ref{eq:F2} are satisfied. The monotonicity properties of $F_1$ and $\dm{F_1}$  are closely related to the positive definiteness of $\phi_1$. 
\begin{definition}\label{def:posdef}
We say that a function $\phi\colon \mR^d \to \mR$ is positive semi-definite, if it is even and for any $n\in\mathbb{N}$, 
\begin{align*}
    \sum\limits_{i,j=1}^n \phi(x_i - x_j) \alpha_i \alpha_j \geq 0,\quad x_1,\ldots,x_n\in \mR^d,\ \alpha_1,\ldots,\alpha_n\in \mathbb{R},
\end{align*}
\end{definition}
\begin{lemma}\label{lem:Fexample}
The function $F_1$ defined by \eqref{eq:Fexample} satisfies \ref{eq:M1} and \ref{eq:M12}.
\end{lemma}

\begin{proof}
Since $\phi_1 \in C^\infty_c(\mR^d)$,  by the well known characterization of positive-definiteness \cite[p.~93]{MR0107124} for every $f\in L^1(\mR^d)$,
\begin{align*}
    \int_{\mR^d}\int_{\mR^d}  \phi_1(x-y) f(x)f(y)\, dy\, dx \geq 0.
\end{align*}
Let $m_1,m_2\in \mP$ and let $m_\eps = (m_1 - m_2)\ast \eta_\eps$ for $\eps>0$. Then $m_\eps \in L^1(\mR^d)$ and by smoothness of $\cm\phi_1$ we have
\begin{align*}
    \int_{\mR^d} (F(x,m_1) - F(x,m_2))\, (m_1(dx) - m_2(dx)) &= \int_{\mR^d}\int_{\mR^d} \cm\phi_1(x-y)\,(m_1(dy) - m_2(dy)) (m_1(dx) - m_2(dx))\\
    &=\lim\limits_{\eps\to 0^+} \int_{\mR^d}\int_{\mR^d} \cm\phi_1(x-y)\, m_\eps(dx) m_\eps(dy)\geq 0.
\end{align*}
This proves that $F$ satisfies \ref{eq:M1}. In order to get \ref{eq:M12} we first note that by Lemmas~\ref{lem:cnegapprox} and \ref{lem:dmf}, and the fact that $\phi$ is smooth, we get that for any $\gamma \geq 0$ and $\rho \in C^{-\gamma}_b(\mR^d)$,
\begin{align*}
    \langle\langle \cm\phi_1(x-y),\rho\rangle_y,\rho\rangle_x = \lim\limits_{\eps\to 0^+}  \langle\langle \cm\phi_1(x-y),\rho\ast\eta_\eps\rangle_y,\rho\ast\eta_\eps\rangle_x.
\end{align*}
Also by Lemma~\ref{lem:cnegapprox}, we have $\rho\ast\eta_\eps\in C^{-0}_b(\mR^d)$, so there exists $\mu_\eps\in \mathcal{M}(\mR^d)$ such that $\langle f,\rho\ast\eta_\eps\rangle = \int f\mu_\eps$ for every $f\in C_0(\mR^d)$. In particular, since $\cm\phi_1(x-\cdot)\in C_0(\mR^d)$ for every $x\in \mR^d$, we have
\begin{align*}
    \langle \cm\phi_1(x-y),\rho\ast\eta_\eps\rangle_y = \int_{\mR^d} \cm\phi_1(x-y)\, \mu_\eps(dy) =: \Phi(x)\in C_0(\mR^d).
\end{align*}
This implies that
\begin{align*}
    \langle\langle \cm\phi_1(x-y),\rho\ast\eta_\eps\rangle_y,\rho\ast\eta_\eps\rangle_x = \langle \Phi,\rho\ast\eta_\eps\rangle = \int_{\mR^d} \Phi(x)\, \mu_\eps(dx) = \int_{\mR^d}\int_{\mR^d} \cm\phi_1(x-y)\, \mu_\eps(dy) \mu_\eps(dx) \geq 0,
\end{align*}
which proves \ref{eq:M12}.
\end{proof}

    We now discuss $F_2$ defined in \eqref{eq:Fexample2}. \cb
    It was shown in \cite{MR3967062} that it satisfies \ref{eq:M1}.
     \begin{lemma}
        The function $F_2$ defined in \eqref{eq:Fexample2} has a version of $\dm{F_2}$ which satisfies \ref{eq:M12}.
    \end{lemma}
    \begin{proof}
    We first note that for $m,m'\in \mP$
    \begin{align*}
        &\lim\limits_{h\to 0^+} \frac{F_2(x,m + h(m'-m)) - F_2(x,m)}{h} \\
        &= \int_{\mR^d} \partial_{s}\Phi(z,(\phi_2\ast m)(z)) \lim\limits_{h\to 0^+} \frac{\phi_2\cb\ast(m + h(m'-m))(z) - \phi_2\cb\ast m (z)}{h}\phi_2\cb(x-z)\, dz\\
        &=\int_{\mR^d} \partial_{s}\Phi(z,(\phi_2\cb\ast m)(z)) \phi_2\cb(x-z)\int_{\mR^d} \phi_2\cb(z-y)\, (m'-m)(dy)\, dz\\
        &=\int_{\mR^d} \int_{\mR^d}\partial_{s}\Phi(z,(\phi_2\cb\ast m)(z)) \phi_2\cb(x-z) \phi_2\cb(z-y)\, dz\, (m'-m)(dy),
    \end{align*}
    so we can take 
    \begin{align*}
        \dm{F_2}(x,m,y) = \int_{\mR^d}\partial_{s}\Phi(z,(\phi_2\cb\ast m)(z)) \phi_2\cb(x-z) \phi_2\cb(z-y)\, dz.
    \end{align*}
    By the proof of Lemma~\ref{lem:Fexample} it suffices to verify \ref{eq:M12} only for $\mu\in \mathcal{M}(\mR^d)$. By Fubini's theorem we have
    \begin{align*}
    \int_{\mR^d}\int_{\mR^d} \dm{F_2}(x,m,y)\, \mu(dy)\, \mu(dx) &= \int_{\mR^d} \int_{\mR^d}\int_{\mR^d}\partial_{s}\Phi(z,(\phi_2\cb\ast m)(z)) \phi_2\cb(x-z) \phi_2\cb(z-y)\, dz\, \mu(dy)\, \mu(dx)\\
    &=\int_{\mR^d}\partial_{s}\Phi(z,(\phi_2\cb\ast m)(z)) (\phi_2\cb\ast \mu(z))^2\, dz \geq 0.
    \end{align*}
    \end{proof}
    \arr\subsubsection{The Hamiltonian $H$} One prominent example satisfying all our assumptions is the quadratic Hamiltonian $H(x,u,p) = H(p) = |p|^2$. We can also consider \begin{align*}
        H(x,p) = f_1(x) + f_2(x) g_1(p) + g_2(p),
    \end{align*}
    where $f_1\in C^4_b(\mR^d)$, $g_2$ is smooth and satisfies \ref{eq:M2}, $f_2\in C^4_b(\mR^d)$ is sufficiently small, $g_1$ is smooth with at most linear growth. Such $H$ satisfies assumptions \ref{eq:H}--\ref{eq:H3} and \ref{eq:M2} as well.

    If we let
    \begin{align*}
        H(x,u,p) = H_1(x,p) + H_2(x,u),
    \end{align*}
    with $H_1$ satisfying the assumptions above, we can consider e.g. $H_2(x,u) = f(x)g(u)$, where $f\in C^4_b(\mR^d)$ is nonnegative (e.g. $f\equiv 1$) and $g\in C^4(\mR^d)$ with $g'\geq\gamma$ for some $\gamma\in \mathbb{R}$.
    Then \ref{eq:H}--\ref{eq:H3} are satisfied. Some examples are $g(u) = \beta u$ for 
    $\beta\in\mR$, or when $u\geq0$, $g(u) = (1 + u^2)^{
    q/2}$ 
    for 
    $q\in\mR$ 
    \cb. In order to have \ref{eq:M3} needed for uniqueness of solutions to the MFG system, it suffices to have $0\leq g'\leq c_2$, which is the case e.g. for $g(u) = \beta u$ with $\beta>0$ and for non-negative $u$, $g(u) = (1 + u^2)^{q/2}$ with $q\in(0,1)$.

\subsection{Modifications needed in   existing proofs/assumptions}\label{sec:err}

\subsubsection{Approximations in H\"older spaces}\label{sec:approx} Let $\gamma\geq 0$ and $K\in\{\mR^d,\mathbb{T}^d,\Omega\}$, where $\Omega\subseteq \mR^d$ is a smooth domain.  We first note  that
$C^\infty_b(K)\cap C^{\pm \gamma}_b(K)$ is in general not dense in $C^{\pm \gamma}_b(K)$.  
There are some exceptions for which the density holds, e.g., $C_b^{\gamma}(\mathbb{T}^d)$ with $\gamma\in \{0,1,2,\ldots\}$, but in  most  
cases it does not hold -- here  are some  counterexamples.
\begin{itemize}
\item For $C^{-0}_b(K)$ we consider $\delta_0$: 
 any $\rho \in C_b(K)$ can be identified with an element in $C^{-0}_b(K)$ via  $\langle \rho,\phi\rangle = \int \rho(x) \phi(x)\, dx$,  and  then by taking test functions $\phi_\eps\in C_b(K)$ such that $0\leq \phi_\eps\leq 1$, $\phi_\eps(0)=1$, and $\int \phi_\eps(x)\, dx < \eps$, we get $|\langle \rho,\phi_\eps\rangle| \leq \eps\|\rho\|_{\infty}$. Therefore,
\begin{align*}
    \|\rho - \delta_0\|_{C^{-0}_b(K)}\geq \frac{|\langle \delta_0 - \rho,\phi_\eps\rangle|}{\|\phi_\eps\|_{\infty}} =\langle \delta_0 - \rho,\phi_\eps\rangle \geq 1 - \delta,
\end{align*}
where $\delta$ can be arbitrarily small.
\item \ar{A counterexample for $C^\gamma_b(K)$ with $\gamma \in (0,1)$ is the function $\mathbb{R}\ni x\mapsto (x\vee 0)^{\gamma}$, see also \cite{29869}.}
\item For $C^{-\gamma}_b((-1,1))$ with $\gamma\in (0,1)$ consider the functional $\rho$ defined as 
\begin{align*}
    \langle \rho,\phi\rangle = \lim\limits_{x\to 0^+} \frac{\phi(x)}{x^{\gamma}}
\end{align*}
for $\phi$ in the subspace of functions in $C^{\gamma}_b((-1,1))$ for which the above limit exists and is finite. Then extend $\rho$ to the whole $C^{\gamma}_b((-1,1))$ by the Hahn--Banach theorem. If $\phi_\eps(x) = x^{\gamma}\textbf{1}_{[0,\eps)}(x) + \eps^\gamma\textbf{1}_{(\eps,1)}(x)$ for $\eps\in (0,1)$ and $x\in (-1,1)$, then $\|\phi_\eps\|_{C^{\gamma}_b((-1,1))} \in [1,2]$ and $\langle \rho,\phi_\eps\rangle = 1$. Therefore, for any $\widetilde{\rho}\in C_b((-1,1))$,
\begin{align*}
    \|\widetilde{\rho} - \rho\|_{C^{-\gamma}_b((-1,1))}\geq \frac{|\langle \widetilde{\rho} - \rho,\phi_\eps\rangle|}{\|\phi_\eps\|_{C^{\gamma}_b(\mR^d)}} \geq \frac 12 \bigg|1 - \int \phi_\eps(x)\widetilde{\rho}(x)\, dx\bigg|,
\end{align*}
which is greater than $\frac 14$ for $\eps$ sufficiently small.
\item The above counterexamples can also be used to show lack of density for $C^{-n}$, $C^{\gamma}$, and $C^{-\gamma}$ spaces for $n=1,2,\ldots$ and non-integer $\gamma>0$.
\end{itemize}
The lack of density has the following consequences: 
\begin{enumerate}[label=(\alph*)]
    \item (Time continuity in Schauder estimates) The solutions of $(\partial_t - \Delta)u = 0$ with the initial condition in $u_0\in C^{2+\sigma}_b(K)$ ($\sigma\in (0,1)$) are in general not of class $C([0,T],C^{2+\sigma}_b(K))$\footnote{The parabolic regularity $C^{1+\sigma/2,2+\sigma}$ given by Schauder estimates does not imply that the $C^{2+\sigma}_b$ norms are continuous up to 0.}\footnote{But they are in $C_b((0,T],C^{2+\sigma}_b(K))$ and $C([0,T],C^{2+\sigma-\eps}_b(K))$ for arbitrarily small $\eps>0$.}. If that was the case, then the instantaneous smoothing effect\footnote{For example, on $\mR^d$ the solution is given by the convolution with the smooth heat kernel.} would imply that any $u_0\in C^{2+\sigma}_b(K)$ can be approximated by smooth functions $u(t,\cdot)$ in $C^{2+\sigma}_b(K)$, which contradicts the lack of density. 
     Due to this, some changes are needed in formulations of e.g. \cite{2020arXiv200110406C,MR3967062,MR4420941,2022arXiv221106514Z}.  This also means that the H\"older regularity in time in \cite[Lemma~3.2.2]{MR3967062}\footnote{The last inequality on page 59 is not true.}  does not hold,
    similar for \cite[Corollary~5.21]{2021arXiv211107020J}.
    \item (Approximation of data in negative H\"older spaces) In the linear system result (see Theorem~\ref{th:linsyst} below) the coefficients in negative order H\"older spaces are approximated by smooth functions. This works fine as long as the coefficient has a slightly better regularity than the norm it is approximated in, see Lemma~\ref{lem:cnegapprox},  but the approximation of initial conditions of class $C^{-k-1-\sigma}_b(K)$ by smooth functions in the $C^{-k-1-\sigma}_b(K)$ norm 
     in  \cite[Lemma~3.3.1]{MR3967062} (see also \cite{2020arXiv200110406C,2022arXiv220315583D,2021arXiv211107020J,MR4420941,2022arXiv221106514Z}) is not possible.  We avoid this problem by considering initial conditions with $C^{-k-1}_b$ regularity, which is sufficient for all the cases where Theorem~\ref{th:linsyst} is used. Another solution would be to drop norm convergence and use weak-$\ast$ and the Banach--Alaoglu theorem. However, making this argument work seems cumbersome because of the lack of time equicontinuity of the approximate solutions in $C^{-k-1-\sigma}_b(K)$.
    \item (Pathological functionals in negative H\"older spaces) Working with the negative order H\"older spaces on non-compact sets poses yet another problem with density: the convolution of a functional with a $C_c^\infty$ function need not even be a measure. This is the case for functionals supported at infinity like the Banach limit, see Remark~\ref{rem:Banach}. Therefore, the approximation arguments in \cite[Theorem~5.28]{2021arXiv211107020J} and \cite[Proposition~5.5]{2020arXiv200110406C} need to be changed. 
    We solve the issue of pathological functionals by requiring the data to have representations in terms of integrals, see Definition~\ref{def:MR}. This requires some work, because whenever we apply Theorem~\ref{th:linsyst}, we need to ensure that the term $c$ has this property.
    \item (Compactness arguments for the linearized system) The discussion of data approximation and obtaining a compact Leray--Schauder map in \cite[Proposition~5.5]{2020arXiv200110406C}  needs 
     modification, since the whole space requires different compactness arguments than the torus.  For $K$ of infinite Lebesgue measure,  all of $C^{\gamma}_b(K)$ 
    cannot be embedded into 
    $C^{-\gamma}_b(K)$ because 
    $\langle f,\phi\rangle=\int_K f\phi$  does not converge for all  $f,\phi\in C^{\gamma}_b(K)$,  
    and compactness argumets using the Arzel\`a--Ascoli theorem are no longer appropriate.  Instead,  we use the $L^1$ setting and the Kolmogorov--Riesz theorem, see the discussion following \eqref{eq:represent}. 
    The $C^\gamma_b$ norm works on the torus only because on sets of finite measure the supremum norm dominates the $L^1$ norm. 
    The above discussion is also relevant for \cite[Theorem~5.28]{2021arXiv211107020J}.  We also note that  the proof of \cite[Lemma~B.1]{2021arXiv211107020J} states that the parabolic space $C^{2+\alpha,1+\alpha/2}([0,T]\times (0,\infty))$ is compactly embedded in $C^{2,1}([0,T]\times (0,\infty))$, which is not true.
\end{enumerate}
\subsubsection{Insufficient regularity assumptions}
\begin{enumerate}[label=(\alph*)]
    \item The proof of the existence of $\dm{U}$ in \cite[Proposition~3.4.3]{MR3967062} requires $n\geq 1$ instead of $n\geq 0$, because the estimate $\|m - \widehat{m}\|_{-n-\alpha} \leq d_1(m,\widehat{m})$ at the end of the proof does not hold with $n=0$.
    \item The main ``bottleneck'' in the study of the well-posedness of the master equation is the continuity in $m$ (or in $t$) of $D^2_y\dm{U}$. In \cite{MR3967062} the proof is skipped, whereas in \cite{MR4420941} this property is used without mention in the density argument in the proof of Theorem~2.5. In fact a stronger Lipschitz continuity result is proved in \cite{MR3967062} in connection with the convergence problem, after the proof of well-posedness for the master equation (see Proposition~3.6.1). It requires $n\geq 2$ in their setting, which enforces quite strong assumptions. We adapt and modify the idea of Ricciardi \cite{MR4420941} 
    to allow blow-up of norms at time $t_0$, and prove
    Lipschitz continuity in $m$ with less assumptions -- corresponding roughly to $n\geq 1$ in \cite{MR3967062}.

    One could expect that the excessive assumptions used in \cite[Proposition~3.6.1]{MR3967062} appear because Lipschitz continuity is 
    stronger 
    than mere continuity, but 
    the main obstacle here seems to be the structure of the generic linear system result. The expression $D^2_y\dm{U}(\cdot,\cdot,m_0^1,\cdot) - D^2_y\dm{U}(\cdot,\cdot,m_0^2,\cdot)$ satisfies the linear system \eqref{eq:linsyst} with $c\in B([t_0,T], C^{-2}_b(\mR^d))$ (see $c_1$ in \eqref{eq:coefficients}, where $\rho^2$ satisfies the initial condition $\rho^2(t_0) = \partial^2_{ij} \delta_y$). In the framework of \cite[Lemma~3.3.1]{MR3967062} they therefore need to take $n=2$, whether to prove just continuity or Lipschitz continuity. Hence the proof of \cite[Theorem 2.4.2]{MR3967062} 
    really requires $n\geq 2$.

     For \cite{MR4420941} (see also \cite{2022arXiv220315583D,2022arXiv221106514Z})  this has further consequences,  as the linear system result in \cite[Proposition~5.8]{MR4420941} is insufficient to prove the main results 
     as the estimates depend on $\|c\|_{L^1([t_0,T]\times\Omega)}$. However for the $c$ mentioned above, we can only expect to have $\|c(t)\|_{L^1} \leq C(t-t_0)^{-1}$, which is not integrable in time. We fix the approach of Ricciardi \cite{MR4420941} by using more integrability in time at the expense of less regularity in space:
in Theorem~\ref{th:linsyst} we allow $c\in L^1([t_0,T],C^{-k-\sigma}_b(\mR^d))$. Then by using Lemma~\ref{lem:l1cneg} we ensure that every $c$ we encounter (including above) is in $L^1([t_0,T],C^{-1-\sigma}_b(\mR^d))$, so that we only need $k=n=1$ for well-posedness.\footnote{In fact Lemma~\ref{lem:l1cneg} gives even better regularity than $C^{-1}$, but for various technical reasons we were unable to work with $n=0$.}


\end{enumerate}
\subsubsection{Monotonicity and normalization}
The cross-multiplication estimate in the proof of the linear system result  \cite[Lemma~3.3.1]{MR3967062} (see also \cite{2022arXiv220315583D,MR4420941,2022arXiv221106514Z})  tacitly uses the property (as above, $K\in \{\Omega, \mathbb{T}^d,\mR^d\})$ 
\begin{align}\label{eq:monotonicity1}
   \bigg\langle \big\langle \dm{F(\cdot,m,\cdot)}, \rho\big\rangle_{y},\, \rho\bigg\rangle_x \geq 0,\quad \rho\in C^{-k-1-\sigma}_b(K),
\end{align}
with the claim that it follows from the standard monotonicity condition 
\begin{align}\label{eq:monotonicity2}
    \int_{\mR^d} (F(x,m_1) - F(x,m_2))\, (m_1 - m_2)(dx) \geq 0,\quad m_1,m_2\in \mathcal{P}(K).
\end{align}
As we argue in Example~\ref{ex:monot}, this is not the case, therefore \eqref{eq:monotonicity1} became a separate assumption in our paper. It turns out quite tedious to verify; in  Section~\ref{sec:examples} we show that both conditions are satisfied for standard examples in the literature, but to this end we need to drop the usual normalization condition \eqref{eq:normalization}, see Remark~\ref{rem:normalization}. 

\subsubsection{Stability in $t_0$ for the MFG system}
In order to prove continuity in $t_0$ for $U$ and its derivatives in Proposition~\ref{prop:Jdef} we consider a linearized system with the terms corresponding to solutions of the MFG system with the same initial measure on different time intervals $[t_1,T]$ and $[t_2,T]$. This requires an appropriate stability result with respect to $t_1$ and $t_2$, which we give in Lemma~\ref{lem:timestab}. This is not a great issue and it does not impose any additional assumptions, but it seems to be overlooked in the literature.


\section{Well-posedness and regularity for single equations}\label{sec:single}
In this section we establish well-posedness and regularity for several linear equations. We also recall results of \cite{MR4309434} on Fokker--Planck and Hamilton--Jacobi equations and prove new and refined variants of these results: better (almost optimal) regularising effect, (much) weaker assumptions on the data, and mixed local-nonlocal operators. In \cite{MR4309434} purely nonlocal diffusions are treated, but the results which are required for our present purposes are valid also for local and mixed local-nonlocal diffusions -- see Theorem \ref{th:FPOEERJ}, Remark \ref{rem:FP}, and Theorem \ref{th:HJ} with proof. 
\begin{remark}\label{rem:t0}
The results of this section are given for equations on the time interval $(0,T)$, but since $T$ is arbitrary, they are true also on the intervals of the form $(t_0,T)$ with any $t_0\in (0,T)$. The constants in the estimates may depend on $T$, but when $T$ is fixed, the constants for equations on $(t_0,T)$ do not depend on $t_0$.
\end{remark}
Almost all the results of this section are obtained with the help of Duhamel's formula and Banach's fixed point theorem. Using these tools we obtain mild solutions and then we argue that they in fact solve the equation in the desired sense. We say that $u$ is a mild solution to the problem \begin{align*}
    \begin{cases}
    \partial_t u - \mL u +\mathcal{H}(t,x,u,Du) = 0,\quad &(t,x)\in (0,T]\times \mR^d,\\
    u(0) = u_0,\quad &x\in \mR^d,
    \end{cases}
\end{align*}
if it satisfies
\begin{align*}
    u(t,x) = K_t \ast u_0(x) + \int_0^t K(t-s,\cdot)\ast \mathcal{H}(s,\cdot,u(s,\cdot),Du(s,\cdot))\, dy\, ds.
\end{align*}
We will differentiate the above formula to get bounds for derivatives of $u$. Formally,
\begin{align*}
    D_xu(t,x) = K_t \ast D_xu_0(x) + \int_0^t (D_xK)(t-s,\cdot)\ast \mathcal{H}(s,\cdot,u(s,\cdot),Du(s,\cdot))\, dy\, ds.
\end{align*}
By \ref{eq:K} we have $\|D_x K(t-s)\|_{L^1(\mR^d)}\leq \mathcal{K}(t-s)^{-\frac 1{\alow}}$. This motivates the need for the following Gr\"{o}nwall-type inequality.
\begin{lemma}[{\cite[Corollary~2]{MR2290034}}]\label{lem:Gronwall}
Let $a,b\geq 0$, $\gamma>0$ and assume that $f\colon [0,T]\to \mR$ satisfies
\begin{align*}
    f(t) \leq a + b\int_0^t (t-s)^{\gamma-1} f(s)\, ds,\quad t\in [0,T].
\end{align*}
Then $f(t) \leq c(T,b)\cdot a$ for all $t\in [0,T]$, with $c$ locally bounded.
\end{lemma}
\subsection{Schauder theory 
for linear equations}
Various forms of regularity estimates for parabolic equations with nonlocal L\'evy-type operators have been known for quite a long time, but we include the proof for completeness and in order to demonstrate that \ref{eq:K} is tailor-made for this result. 
\begin{lemma}\label{lem:cd}
Assume that \ref{eq:K} holds, let $\sigma\in (0,\alow-1)$, $b\in (C([0,T],C^{k-1}_b(\mR^d)))^d$, $z_0\in C^{k+\sigma}_b(\mR^d)$ for some natural number $k\geq 2$, and $f\in C([0,T],C^{k-1}_b(\mR^d))$. Then,
\begin{enumerate}[label=(\alph*)]
\item the problem
\begin{equation}\label{eq:cd}
    \begin{cases}
        \partial_t z - \mL z - b(t,x) Dz = f(t,x)\quad &{\rm in}\ (0,T)\times \mR^d,\\
        z(0) = z_0 &{\rm in}\ \mR^d.
    \end{cases}
\end{equation}
has a unique classical solution $z\in  B([0,T], C^{k+\sigma}_b(\mR^d))\cap C([0,T],C^{k+\sigma-\eps}_b(\mR^d))$ (for arbitrarily small $\eps>0$), which satisfies
\begin{align}\label{eq:c3bound} \|\partial_tz\|_{\infty} + \sup\limits_{t\in [0,T]}\|z(t,\cdot)\|_{C^{k+\sigma}_b(\mR^d)} \leq C( \|z_0\|_{C^{k+\sigma}_b(\mR^d)} + \sup_{t\in [0,T]}\|f(t,\cdot)\|_{C^{k-1}_b(\mR^d)}),\end{align}
where $C$ is independent of $\|z_0\|_{C^k_b(\mR^d)}$ and $\sup_{t\in (0,T)}\|f(t,\cdot)\|_{C^{k-1}_b(\mR^d)}$.
\item For every $\eps>0$ there exists a modulus of continuity $\omega$ independent of $f$ and $z_0$ such that
\begin{align*}
    \|z(s,\cdot) - z(t,\cdot)\|_{C^{k+\sigma-\eps}_b(\mR^d)}&\leq ( \|z_0\|_{C^{k+\sigma}_b(\mR^d)} + \sup_{t\in [0,T]}\|f(t,\cdot)\|_{C^{k-1}_b(\mR^d)})\omega(t-s),\quad s,t\in [0,T],\\
    \|\partial_t z(s,\cdot) - \partial_t z(t,\cdot)\|_{\infty}&\leq ( \|z_0\|_{C^{k+\sigma}_b(\mR^d)} + \sup_{t\in [0,T]}\|f(t,\cdot)\|_{C^{k-1}_b(\mR^d)})\omega(t-s),\quad s,t\in [0,T].
\end{align*}
\item There exists $C>0$ which depends only on $T$, $\mathcal{K}$ from \ref{eq:K} and $\|b\|_{\infty}$, such that
\begin{equation}\label{eq:gradest}\sup\limits_{t\in [0,T]}\|z(t,\cdot)\|_{C^1_b(\mR^d)} \leq C(\|z_0\|_{C^1_b(\mR^d)} + \|f\|_{\infty}).\end{equation}
\end{enumerate}
\end{lemma}

\begin{remark}\label{rem:minuseps}
It should be possible to prove this result with $\sigma = \alow -1$, but it would require more work and more precise assumptions on the heat kernel. One needs to have stronger scaling properties in order to isolate and treat singularities separately from the integrable part of the integrands, see Priola \cite[Theorem~3.4]{MR2945756} for stable diffusions, or the book of Gilbarg and Trudinger \cite[Lemma 4.4]{MR0473443} for local elliptic equations.
\end{remark}
\begin{proof}
(a)\quad The proof is based on Duhamel's formula and uses Banach fixed point arguments similar to those in \cite[Sections~5 and 6]{MR4309434}. We let $\sigma \in (0,\alow -1)$, $\eps\in (0,\sigma)$ and we define
\begin{align*}
    X = B([0,T],C^{k+\sigma}_b(\mR^d))\cap C([0,T],C^{k+\sigma-\eps}_b(\mR^d))
\end{align*}
and for $\widetilde{z}\in X$ we say that $S(\widetilde{z}) = z$, if for every $x\in \mR^d$ and  $t\in [0,T]$,
\begin{align}\label{eq:duhc0}
    z(t,x) = K_t\ast z_0(x) + \int_0^t \int_{\mR^d} K(t-s,x-y)b(s,y)D\widetilde{z}(s,y)\, dy\, ds + \int_0^t\int_{\mR^d} K(t-s,x-y) f(s,y)\, dy\, ds.
\end{align}
We will now estimate spatial $C^{k+\sigma}$ norms of $z$. First, using integration by parts, the triangle inequality, and \ref{eq:K} we get
\begin{align*}
    |z(t,x)| &\leq \|z_0\|_{\infty} + \int_0^t \int_{\mR^d} |D_y\big(K(t-s,x-y)b(s,y)\big)||\widetilde{z}(s,y)|\, dy\, ds + T\sup\limits_{t\in[0,T]} \|f(t,\cdot)\|_{\infty}\\
    &\leq \|z_0\|_{\infty} + \int_0^t\|\widetilde{z}(s,\cdot)\|_{\infty} \int_{\mR^d} |D_y\big(K(t-s,x-y)b(s,y)\big)|\, dy\, ds + T\sup\limits_{t\in[0,T]} \|f(t,\cdot)\|_{\infty}\\
    &\leq \|z_0\|_{\infty}+T\sup\limits_{t\in[0,T]} \|f(t,\cdot)\|_{\infty} + C\sup\limits_{t\in[0,T]} \|b(t,\cdot)\|_{C^1_b(\mR^d)}\sup\limits_{t\in[0,T]}\|\widetilde{z}(t,\cdot)\|_{\infty},
\end{align*}
where $C=C(K,T)$ and $C(T)\to 0$ as $T\to 0^+$. It follows that 
\begin{align*}
    \sup\limits_{t\in[0,T]}\|z(t,\cdot)\|_{\infty} \leq \|z_0\|_{\infty}+T\sup\limits_{t\in[0,T]} \|f(t,\cdot)\|_{\infty} + C\sup\limits_{t\in[0,T]} \|b(t,\cdot)\|_{C^1_b(\mR^d)}\sup\limits_{t\in[0,T]}\|\widetilde{z}(t,\cdot)\|_{\infty}.
\end{align*}
 By differentiating \eqref{eq:duhc0} with respect to $x$ we get
 \begin{align}\label{eq:duhc1}Dz(t,x) = K_t\ast Dz_0(x) + \int_0^t \int_{\mR^d} D_xK(t-s,x-y)b(s,y)D\widetilde{z}(s,y)\, dy\, ds + \int_0^t\int_{\mR^d} D_xK(t-s,x-y) f(s,y)\, dy\, ds.\end{align}
 The interchange of the derivative and the double integrals above is justified as follows: we use the fundamental theorem of calculus on the difference quotients, the Fubini--Tonelli theorem, and we change the variables, getting expressions of the type:
 \begin{align*}
     \int_0^t\int_{\mR^d}\int_0^1 D_xK(t-s,x+\lambda h-y) f(s,y)\, d\lambda \, dy\, ds = \int_0^1\int_0^t\int_{\mR^d} D_xK(t-s,x-y) f(s,y+\lambda h) \, dy\, ds \, d\lambda,
 \end{align*}
 where $h\in \mR^d$. Then, $|D_x K| \|f\|_{\infty}$ is a majorizing function (it is integrable by \ref{eq:K}) and so \eqref{eq:duhc1} follows from the dominated convergence theorem.
 
 By using \ref{eq:K} in \eqref{eq:duhc1} we find that \begin{align*}
     \sup\limits_{t\in [0,T]} \|Dz(t,\cdot)\|_{\infty} \leq \|z_0\|_{C^1_b(\mR^d)} + C(T)\sup\limits_{t\in [0,T]} \|f(t,\cdot)\|_{\infty} + C(T)\sup\limits_{t\in [0,T]} \|b(t,\cdot)\|_{\infty}\sup\limits_{t\in [0,T]}\|\widetilde{z}(t,\cdot)\|_{C^1_b(\mR^d)},
 \end{align*}
 where $C(T)\to 0$ as $T\to 0^+$. The procedure for the second order derivatives of $z$ is as follows: we use integration by parts in both integrals in \eqref{eq:duhc1} and then we differentiate once more. If we repeat that for the next derivatives, we get that for $1\leq l\leq k$,
 \begin{align} D^lz(t,x)\label{eq:duhcl} = K_t\ast D^lz_0(x) &+ \int_0^t \int_{\mR^d} D_xK(t-s,x-y)D^{l-1}\big(b(s,y)D\widetilde{z}(s,y)\big)\, dy\, ds\\\nonumber &+ \int_0^t\int_{\mR^d} D_xK(t-s,x-y) D^{l-1}f(s,y)\, dy\, ds.
 \end{align}
 It follows that 
 \begin{align*}
      \sup\limits_{t\in [0,T]} \|D^lz(t,\cdot)\|_{\infty} \leq \|z_0\|_{C^l_b(\mR^d)} + C(T)\sup\limits_{t\in [0,T]} \|f(t,\cdot)\|_{C^{l-1}_b(\mR^d)} + C(T)\sup\limits_{t\in [0,T]} \|b(t,\cdot)\|_{C^{l-1}_b(\mR^d)}\sup\limits_{t\in [0,T]}\|\widetilde{z}(t,\cdot)\|_{C^l_b(\mR^d)},
 \end{align*}
 where $C(T)\to 0$ as $T\to 0^+$.
 We now proceed to the H\"older norms. For any $h\in \mR^d\setminus\{0\}$ we have
 \begin{align}
     &\frac{|D^kz(t,x+h) - D^kz(t,x)|}{|h|^{\sigma}}\nonumber\\ &\leq \|z_0\|_{C^{k+\sigma}_b(\mR^d)}  + \int_0^t \int_{\mR^d} \frac{|D_xK(t-s,x+h-y) -D_xK(t-s,x-y)|}{|h|^{\sigma}}|D^{k-1}\big(b(s,y)D\widetilde{z}(s,y)\big)|\, dy\, ds\nonumber\\ &+ \int_0^t\int_{\mR^d} \frac{|D_xK(t-s,x+h-y) -D_xK(t-s,x-y)|}{|h|^{\sigma}} |D^{k-1}f(s,y)|\, dy\, ds.\label{eq:interp}
 \end{align}
 Therefore, by Remark~\ref{rem:K} (c) and \ref{eq:K} we get
 \begin{align*}
     &\sup\limits_{t\in[0,T]} \|D^kz(t,\cdot)\|_{C^\sigma_b(\mR^d)} \leq \|z_0\|_{C^{k+\sigma}_b(\mR^d)}\\  &+ C\big(\sup\limits_{t\in [0,T]} \|f(t,\cdot)\|_{C^{k-1}_b(\mR^d)} +\sup\limits_{t\in [0,T]} \|b(t,\cdot)\|_{C^{k-1}_b(\mR^d)}\sup\limits_{t\in [0,T]}\|\widetilde{z}(t,\cdot)\|_{C^k_b(\mR^d)}\big)\int_0^T \|D^2_x K(s)\|_{L^1(\mR^d)}^{\sigma}\|D_xK(s)\|_{L^1(\mR^d)}^{1-\sigma}\, ds\\
     &\leq \|z_0\|_{C^{k+\sigma}_b(\mR^d)}+ C\big(\sup\limits_{t\in [0,T]} \|f(t,\cdot)\|_{C^{k-1}_b(\mR^d)} +\sup\limits_{t\in [0,T]} \|b(t,\cdot)\|_{C^{k-1}_b(\mR^d)}\sup\limits_{t\in [0,T]}\|\widetilde{z}(t,\cdot)\|_{C^k_b(\mR^d)}\big)\int_0^T s^{-\frac{1+\sigma}{\alow}}\, ds\\
     &\leq \|z_0\|_{C^{k+\sigma}_b(\mR^d)}+ C(T)\big(\sup\limits_{t\in [0,T]} \|f(t,\cdot)\|_{C^{k-1}_b(\mR^d)} +\sup\limits_{t\in [0,T]} \|b(t,\cdot)\|_{C^{k-1}_b(\mR^d)}\sup\limits_{t\in [0,T]}\|\widetilde{z}(t,\cdot)\|_{C^k_b(\mR^d)}\big),
 \end{align*}
 where $C(T)\to 0$ as $T\to 0^+$, because $1+\sigma < \alow$. Note that $C(T)$ blows up as $\sigma\to (\alow-1)^-$, but stays bounded as long as $\sigma$ is separated from $\alow - 1$. 
 
By summing up the above estimates we get
\begin{align*}
\sup\limits_{t\in [0,T]} \|z(t,\cdot)\|_{C^{k+\sigma}_b(\mR^d)} \leq \|z_0\|_{C^{k+\sigma}_b(\mR^d)}+ C(T)\big(\sup\limits_{t\in [0,T]} \|f(t,\cdot)\|_{C^{k-1}_b(\mR^d)} +\sup\limits_{t\in [0,T]} \|b(t,\cdot)\|_{C^{k-1}_b(\mR^d)}\sup\limits_{t\in [0,T]}\|\widetilde{z}(t,\cdot)\|_{C^k_b(\mR^d)}\big),
\end{align*}
where $C(T)\to 0$ as $T\to 0^+$. The proof that $\|z(t,\cdot)\|_{C^{k+\sigma-\eps}_b(\mR^d)}$ is continuous in $t$ is postponed to the appendix, see Lemma~\ref{lem:contsmooth}. We get that $z = S(\widetilde{z}) \in X$. Furthermore, since $S$ is an affine map, we get that it is a contraction if $T$ is small enough (note that $z_0$ does not affect how small $T$ needs to be). Therefore, by Banach's fixed point theorem we get that $S$ has a unique fixed point $z$. By \cite[Lemma~5.12 (a)]{MR4309434} $z$ is a classical solution of \eqref{eq:cd}. We can extend the solution further in time by considering the equation with $z(T)$ as the initial condition. The fact that we can do it indefinitely follows from \eqref{eq:c3bound}, which we now prove. We use the Duhamel formula \eqref{eq:duhc0} in a similar manner as above. For $t\in (0,T)$, $1\leq l\leq k$, and $\sigma\in (0,\alow-1)$ we have (by \ref{eq:K} and Remark~\ref{rem:K} (c))
\begin{align}
    &\|z(t,\cdot)\|_{\infty} \leq \|z_0\|_{\infty} + C\sup\limits_{t\in[0,T]}\|f(t,\cdot)\|_{\infty} + C\sup\limits_{t\in[0,T]}\|b(t,\cdot)\|_{C^{1}_b(\mR^d)}\int_0^t (t-s)^{-\frac 1\alow} \|z(s,\cdot)\|_{\infty}\, ds,\nonumber\\
    &\|D^l z(t,\cdot)\|_{\infty} \leq \|D^lz_0\|_{\infty} + C\sup\limits_{t\in[0,T]}\|f(t,\cdot)\|_{C^{l-1}_b(\mR^d)} + C\sup\limits_{t\in[0,T]}\|b(t,\cdot)\|_{C^{l-1}_b(\mR^d)}\int_0^t (t-s)^{-\frac 1\alow} \|z(s,\cdot)\|_{C^l_b(\mR^d)}\, ds,\label{eq:c1c1}\\
    &\|D^k z(t,\cdot)\|_{C^{\sigma}_b(\mR^d)}\nonumber \\ &\leq \|D^kz_0\|_{C^{\sigma}_b(\mR^d)} + \widetilde{C}\sup\limits_{t\in[0,T]}\|f(t,\cdot)\|_{C^{k-1}_b(\mR^d)} + \widetilde{C}\sup\limits_{t\in[0,T]}\|b(t,\cdot)\|_{C^{k-1}_b(\mR^d)}\int_0^t (t-s)^{-\frac {1+\sigma}\alow} \|z(s,\cdot)\|_{C^k_b(\mR^d)}\, ds.\nonumber
\end{align}
By summing up these estimates and using the Gr\"onwall inequality of Lemma~\ref{lem:Gronwall}, we obtain \eqref{eq:c3bound}.\medskip

(b)\quad The part concerning spatial derivatives is a direct consequence of Lemma~\ref{lem:contsmooth}. Once we have it, continuity for time derivatives follows by using the equation. \medskip

(c)\quad This part follows by taking $l=1$ in \eqref{eq:c1c1} (note that for $l=1$ the estimate also works with $\|Dz(s,\cdot)\|_{\infty}$ in place of $\|z(s,\cdot)\|_{C^1_b(\mR^d)}$) and using Lemma~\ref{lem:Gronwall}.
\end{proof}

\subsection{Integrable solutions of linear equations with divergence form drifts and sources}
Later we will consider equations and systems with data in the negative order H\"older spaces, which we approximate by $L^1$ functions. The following result gives well-posedness and a basis for compactness arguments for linear equations in the $L^1$ setting.
\begin{lemma}\label{lem:FPL1}
Assume that \ref{eq:K} holds and let $V_1\in (C_b([0,T]\times \mR^d))^d$, $V_2\in (L^\infty([0,T],L^1(\mR^d)))^d$, and $\rho_0\in  L^1(\mR^d)$. Then
\begin{enumerate}[label=(\alph*)]
\item There exists a unique $\rho\in C([0,T],L^1(\mR^d))$ satisfying
\begin{align}\label{eq:Duhamel}
        \rho(t,x) = K_t\ast \rho_0(x) + \int_0^t \int_{\mR^d} D_y K(t-s,x-y)(V_1(s,y)\rho(s,y) + V_2(s,y))\, dy\, ds.
\end{align}
Furthermore, $\rho$ is a distributional solution of the equation
\begin{align}\label{eq:distl1}
    \begin{cases}
        \partial_t \rho - \mL \rho - \Div(V_1\rho) - \Div(V_2) = 0,\quad &{\rm in}\ (0,T)\times \mR^d,\\
        \rho(0) = \rho_0, &{\rm in}\ \mR^d,
    \end{cases}
\end{align}
when tested with functions $\phi\in C([0,T],C^2_b(\mR^d))$ such that $\partial_t \phi\in C_b([0,T]\times \mR^d)$, see Lemma~\ref{lem:mildweak}.
\item The following estimate holds true:
\begin{align}\label{eq:L1estimates}
    \sup\limits_{t\in [0,T]} \|\rho(t,\cdot)\|_{L^1(\mR^d)} \leq C(T,\|V_1\|_{\infty},\alow)\bigg(\|\rho_0\|_{L^1(\mR^d)} + \frac{\alow}{\alow - 1}T^{1-1/\alow} \esssup\limits_{t\in[0,T]}\|V_2(t,\cdot)\|_{L^1(\mR^d)}\bigg)
\end{align}
\item There exists $c>0$ depending on $\sup\limits_{t\in[0,T]}\|\rho(t,\cdot)\|_{L^1(\mR^d)}, \|V_1\|_{\infty},$ and  $\esssup\limits_{t\in[0,T]}\|V_2(t,\cdot)\|_{L^1(\mR^d)},$ such that
\begin{align}\label{eq:RK}
    \sup\limits_{t\in [0,T]}\|\rho(t,\cdot + z) - \rho(t,\cdot)\|_{L^1(\mR^d)} \leq \|\rho_0(\cdot + z) - \rho_0\|_{L^1(\mR^d)} + c|z|^{\alow -1}.
\end{align}
\item There exists $C>0$ depending on $\sup\limits_{t\in[0,T]}\|\rho(t,\cdot)\|_{L^1(\mR^d)}, \|V_1\|_{\infty}, \esssup\limits_{t\in[0,T]}\|V_2(t,\cdot)\|_{L^1(\mR^d)},$ and $\rho_0$, and a modulus of continuity $\omega$ independent of these quantities, such that for any $z\in \mR^d$,
\begin{align}\label{eq:L1uc}
    \|\rho(t) - \rho(s)\|_{L^1(\mR^d)} \leq C\omega(|t-s|),\quad t,s\in [0,T].
\end{align}
Furthermore,
\begin{align}\label{eq:cm2uc}
   \|\rho(t) - \rho(s)\|_{C^{-2}_b(\mR^d)} \leq \big(\sup\limits_{t\in[0,T]}\|\rho(t,\cdot)\|_{L^1(\mR^d)}(1+ \|V_1\|_{\infty}) + \esssup\limits_{t\in[0,T]}\|V_2(t,\cdot)\|_{L^1(\mR^d)}\big)|t-s|,\quad t,s\in [0,T].
\end{align}
\item Let $\psi$ be a tightness measuring function for $\{\rho_0,\nu\textbf{1}_{B(0,1)^c}\}$. Then, for all $t\in [0,T]$,
\begin{align*}
    &\int_{\mR^d} \psi(x)\, |\rho(t,x)|\, dx \leq \int_{\mR^d} \psi(x) \, |\rho_0(x)|\, dx + \|D\psi\|_{\infty}\esssup_{t\in [0,T]} \|V_2(t,\cdot)\|_{L^1(\mR^d)}
    \\
    &\qquad +\|D\psi\|_{\infty}\sup\limits_{t\in [0,T]}\|\rho(t,\cdot)\|_{L^1(\mR^d)}\Big(2 + cT\big(\|A\|_{2}+ \|B\|_{\infty}+\|b\|_{\infty} + \int_{|z|< 1}|z|^2\, \nu(dz)\big)+ T\int_{|z|\geq 1}\psi(z)\, \nu(dz)\Big).
\end{align*}
\end{enumerate}
\end{lemma}
\begin{proof}
   (a)\quad We will establish the existence with the use of Banach's fixed point theorem in the space $X = L^\infty([0,T],L^1(\mR^d))$, continuity in time will be shown in d).
   For $\widetilde{\rho}$ we let $S(\widetilde{\rho}) = \rho$ if for every\footnote{By Fubini's theorem $\rho$ is measurable in $(t,x)$ and in $x$ for every fixed $t$.} $t\in[0,T]$ and $x\in \mR^d$,
   \begin{align}
       \rho(t,x)
       &=K_t\ast \rho_0(x) + \int_0^t\int_{\mR^d} D_yK(t-s,x-y)
       (V_1(s,y)\widetilde{\rho}(s,y) + V_2(s,y))\, dy\, ds.\label{eq:mild2}
   \end{align}
    By \eqref{eq:mild2}, Tonelli's theorem, and \ref{eq:K} we get that for every $t\in [0,T]$,
   \begin{align}
       \int_{\mR^d} |\rho(t,x)|\, dx &\leq \|\rho_0\|_{L^1(\mR^d)} + \int_0^t\int_{\mR^d}\int_{\mR^d}|D_y K(t-s,x-y)||V_1(s,y)\widetilde{\rho}(s,y) + V_2(s,y)|\, dx\, dy\, ds\label{eq:rhol1}\\
       &\leq \|\rho_0\|_{L^1(\mR^d)} + \mathcal{K} \int_0^t (t-s)^{-\frac 1{\alow}}\int_{\mR^d}|V_1(s,y)\widetilde{\rho}(s,y) + V_2(s,y)|\, dy\, ds\nonumber\\
       &\leq \|\rho_0\|_{L^1(\mR^d)} + \mathcal{K} T^{1-\frac 1{\alow}}\frac{\alow}{\alow - 1}\big(\|V_1\|_{\infty}\sup\limits_{t\in [0,T]}\|\widetilde{\rho}(t,\cdot)\|_{L^1(\mR^d)} + \esssup\limits_{t\in [0,T]}\|V_2(t,\cdot)\|_{L^1(\mR^d)}\big).\nonumber
   \end{align}
   Therefore we get that $\rho\in L^\infty([0,T],L^1(\mR^d))$ (and $B([0,T],L^1(\mR^d))$). Furthermore, since the map $S$ is affine, a similar calculation leads to the fact that for $T$ small enough $S$ is a contraction on $L^\infty([0,T],L^1(\mR^d))$.
    More explicitly, if
    \begin{align*}
        T\leq \bigg(\frac {\alow - 1}{2\mathcal{K}\alow \|V_1\|_{\infty}}\bigg)^{\frac \alow{\alow - 1}},
    \end{align*}
    then $S$ is a contraction on $X$. By Banach's fixed point theorem we get that $S$ has a unique fixed point in $X$. Furthermore, since the above choice of $T$ does not depend on $\rho_0$, we may extend the solution indefinitely, therefore we get the long time existence as well. By Lemma~\ref{lem:mildweak} we find that $\rho$ is a distributional solution of \eqref{eq:distl1}. \smallskip
    
    \noindent (b)\quad In order to get \eqref{eq:L1estimates} 
    we integrate \eqref{eq:Duhamel} with respect to $x$ and we estimate using triangle inequality, Tonelli's theorem, and \ref{eq:K}:
    \begin{align*}
        \int_{\mR^d}|\rho(t,x)|\, dx &\leq \|\rho_0\|_{L^1(\mR^d)} +\int_0^t \int_{\mR^d}\int_{\mR^d} |D_y K(t-s,x-y)| |V_1(s,y)\rho(s,y) + V_2(s,y)| \, dx\, dy\, ds\\
        &\leq \|\rho_0\|_{L^1(\mR^d)} +\|V_1\|_{\infty}\int_0^t(t-s)^{-\frac 1{\alow}}\int_{\mR^d}|\rho(s,y)|\, dy\, ds + \esssup\limits_{t\in [0,T]}\|V_2(t,\cdot)\|_{L^1(\mR^d)}\int_0^t (t-s)^{-\frac 1{\alow}}\, ds\\
        &= \|\rho_0\|_{L^1(\mR^d)} + \esssup\limits_{t\in [0,T]}\|V_2(t,\cdot)\|_{L^1(\mR^d)}T^{1-\frac 1{\alow}}\frac {\alow}{\alow - 1} +\|V_1\|_{\infty}\int_0^t(t-s)^{-\frac 1{\alow}}\int_{\mR^d}|\rho(s,y)|\, dy\, ds.
    \end{align*}
    By using Lemma~\ref{lem:Gronwall}, we get \eqref{eq:L1estimates}.\smallskip
    
    \noindent (c)\quad Fix $z \in \mR^d$ and $t\in [0,T]$. By Duhamel's formula \eqref{eq:Duhamel},
    \begin{align*}
        \|\rho(t,\cdot + z) - \rho(t,\cdot)\|_{L^1(\mR^d)} &\leq \|\rho_0\ast K_t(\cdot + z) - \rho_0\ast K_t\|_{L^1(\mR^d)}\\ &+ \int_{\mR^d}\int_0^t\int_{\mR^d}|D_y K(t-s,x+z-y) - D_yK(t-s,x-y)| |(V_1\rho + V_2)(s,y)|\, dy\, ds\, dx.
    \end{align*}
    By Tonelli's theorem we can estimate the first term as follows:
    \begin{align}\label{eq:RK1}
        \|\rho_0\ast K_t(\cdot + z) - \rho_0\ast K_t\|_{L^1(\mR^d)} \leq \|\rho_0(\cdot + z) - \rho_0\|_{L^1(\mR^d)}.
    \end{align}
    In the second term we split the time integral into $\int_0^{t-\eps} + \int_{t-\eps}^t$. For the first part we use the fundamental theorem of calculus, Tonelli's theorem and \ref{eq:K}:
    \begin{align*}
        &\int_{0}^{t-\eps}\int_{\mR^d}\int_{\mR^d}|D_y K(t-s,x+z-y) - D_yK(t-s,x-y)| |(V_1\rho + V_2)(s,y)|\, dy\, dx\, ds\\
        \leq &|z|\int_0^1\int_{0}^{t-\eps}\int_{\mR^d}\int_{\mR^d}|D^2_{yy} K(t-s,x+\lambda z-y)| |(V_1\rho + V_2)(s,y)|\, dx\, dy\, ds\, d\lambda\\
        \leq &|z|\int_{0}^{t-\eps}\int_{\mR^d}\|D^2_{yy} K(t-s,\cdot)\|_{L^1(\mR^d)} |(V_1\rho + V_2)(s,y)|\, dy\, ds\\
        \leq &|z|\mathcal{K}(\|V_1\|_{\infty}\sup\limits_{t\in[0,T]}\|\rho(t,\cdot)\|_{L^1(\mR^d)} + \esssup\limits_{t\in[0,T]}\|V_2(t,\cdot)\|_{L^1(\mR^d)})\int_0^{t-\eps}(t-s)^{-\frac{2}{\alow}}\, ds\\
        = &|z|\eps^{1-\frac 2{\alow}}\cdot C(\alow,T)\mathcal{K}(\|V_1\|_{\infty}\sup\limits_{t\in[0,T]}\|\rho(t,\cdot)\|_{L^1(\mR^d)} + \esssup\limits_{t\in[0,T]}\|V_2(t,\cdot)\|_{L^1(\mR^d)}).
    \end{align*}
    The second part is estimated as follows:
    \begin{align*}
        &\int_{t-\eps}^t\int_{\mR^d}\int_{\mR^d}|D_y K(t-s,x+z-y) - D_yK(t-s,x-y)| |(V_1\rho + V_2)(s,y)|\, dy\, dx\, ds\\
        \leq &2\int_{t-\eps}^t\int_{\mR^d} \|D_y K(t-s,\cdot)\|_{L^1(\mR^d)} |(V_1\rho + V_2)(s,y)|\, dy\, ds\\
        \leq &2\mathcal{K}(\|V_1\|_{\infty}\sup\limits_{t\in[0,T]}\|\rho(t,\cdot)\|_{L^1(\mR^d)} + \esssup\limits_{t\in[0,T]}\|V_2(t,\cdot)\|_{L^1(\mR^d)}) \int_{t-\eps}^t (t-s)^{-\frac 1{\alow}}\, ds\\
        \leq &\eps^{1- \frac 1{\alow}}\cdot 2C(\alow,T)\mathcal{K}(\|V_1\|_{\infty}\sup\limits_{t\in[0,T]}\|\rho(t,\cdot)\|_{L^1(\mR^d)} + \esssup\limits_{t\in[0,T]}\|V_2(t,\cdot)\|_{L^1(\mR^d)}).
    \end{align*}
    By summing the two estimates, we get
    \begin{align}
    \int_{\mR^d}\int_0^t\int_{\mR^d}|D_y K(t-s,x+z-y) - D_yK(t-s,x-y)| |(V_1\rho + V_2)(s,y)|\, dy\, ds\, dx &\leq \widetilde{C} (|z|\eps^{1 - \frac 2{\alow}} + \eps^{1-\frac 1{\alow}})\nonumber\\
    &\leq C |z|^{\alow -1},\label{eq:RK2}
    \end{align}
    where the last inequality was obtained by minimization in $\eps$. Now, \eqref{eq:RK} follows by adding \eqref{eq:RK1} to \eqref{eq:RK2}.\smallskip

    \noindent (d)\quad Let $\varphi\in C_b^{2}(\mR^d)$ and let $0\leq s<t\leq T$. Using the definition of distributional solution and the fact that $\rho$ satisfying \eqref{eq:Duhamel} is in $B([0,T],L^1(\mR^d))$, it is easy to verify that
    \begin{align*}
        \bigg|\int_{\mR^d} (\rho(t,x) - \rho(s,x))\varphi(x)\, dx\bigg| \leq C\|\varphi\|_{C^2_b(\mR^d)}|t-s|\big(\sup\limits_{t\in[0,T]}\|\rho(t,\cdot)\|_{L^1(\mR^d)}(1+ \|V_1\|_{\infty}) + \esssup\limits_{t\in[0,T]}\|V_2(t,\cdot)\|_{L^1(\mR^d)}\big),
    \end{align*}
    where $C=C(\mL^*)$.
    Note that this gives the second part of the statement. We denote \begin{align*}
        C_1 = C\sup\limits_{t\in[0,T]}\|\rho(t,\cdot)\|_{L^1(\mR^d)}(1+ \|V_1\|_{\infty}) + \esssup\limits_{t\in[0,T]}\|V_2(t,\cdot)\|_{L^1(\mR^d)}.
    \end{align*}
    Now let $\phi\in L^{\infty}(\mR^d)$ with $\|\phi\|_{\infty} = 1$ and let $\eta_\eps$ be a standard mollifier. We have $\phi_{\eps} :=\phi\ast\eta_\eps\in C_b^2(\mR^d)$ and $\|\phi_{\eps}\|_{C^2_b(\mR^d)} \leq C \eps^{-2}$. Therefore,
    \begin{align*}
        \bigg|\int_{\mR^d} (\rho(t,x) - \rho(s,x))\phi_{\eps}(x)\, dx\bigg| \leq C_1|t-s|\eps^{-2}.
    \end{align*}
    Furthermore, if we let $\rho(t,x) - \rho(s,x) = g(x)$, then using Fubini's theorem and a change of variables we get
    \begin{align*}
        \bigg|\int_{\mR^d} g(x)\phi_{\eps}(x)\, dx- \int_{\mR^d} g(x)\phi(x)\, dx\bigg| =  \bigg|\int_{\mR^d} (g_{\eps}(x)- g(x))\phi(x)\, dx\bigg| \leq \|\phi\|_{\infty}\sup\limits_{|z|\leq \eps}\|g(\cdot+z) - g(\cdot)\|_{L^1(\mR^d)}.
    \end{align*}
    Thus, by using (c) and the fact that $\|\phi\|_{\infty} = 1$, we get
    \begin{align*}
    \bigg|\int_{\mR^d} (\rho(t,x) - \rho(s,x))\phi_{\eps}(x)\, dx- \int_{\mR^d} (\rho(t,x) - \rho(s,x))\phi(x)\, dx\bigg| \leq 2\big(\sup\limits_{|z|<\eps}\|\rho_0(\cdot + z) - \rho_0\|_{L^1(\mR^d)} + c\eps^{\alow -1}\big).
    \end{align*}
    It follows that 
    \begin{align*}
        \bigg|\int_{\mR^d} (\rho(t,x) - \rho(s,x))\phi(x)\, dx\bigg| \leq C_1|t-s|\eps^{-2} + 2\big(\sup\limits_{|z|<\eps}\|\rho_0(\cdot + z) - \rho_0\|_{L^1(\mR^d)} + c\eps^{\alow - 1}\big).
    \end{align*}
    By minimizing in $\eps$ (or taking, e.g., $\eps = |t-s|^{\frac 14}$) and using the fact that $L^{\infty}$ is the dual of $L^1$, we get \eqref{eq:L1uc}.\smallskip
    
    \noindent (e)\quad The proof consists in testing the equation with $\psi$, but since it is unbounded, an appropriate cut-off argument is needed. We refer to \cite[Proposition~6.6 (c)]{MR4309434} for the details.\smallskip

\end{proof}
\subsection{Measure valued solutions of Fokker--Planck equations}\label{sec:FP}
\begin{theorem}[{\cite[Section~6]{MR4309434}}]\label{th:FPOEERJ} Assume that \ref{eq:K} holds, $b\in (C_b((0,T),C^2_b(\mR^d)))^d$, and that $m_0\in C^2_b(\mR^d)$ is the density of a probability measure. Then the following assertions hold.
\begin{enumerate}[label=(\alph*)]
\item The equation
\begin{equation}\label{eq:FP}
    \begin{cases}
    \partial_t m - \mL^* m -\Div(bm) = 0\quad {\rm in }\ (0,T)\times \mR^d,\\
    m(0) = m_0\quad {\rm in}\ \mR^d, \end{cases}
\end{equation}
has a unique classical solution $m$. The function $m(t,\cdot)$ is the density of a probability measure for every $t\in[0,T]$ and
$$\|m\|_{\infty} + \|\partial_t m\|_{\infty} + \|Dm\|_{\infty} + \|D^2m\|_{\infty} \leq C,$$
with $C$ depending on $d,\sup_{t\in [0,T]}\|b(t,\cdot)\|_{C^2_b(\mR^d)}$, $\|m_0\|_{C^2_b(\mR^d)}$ and the constant in \ref{eq:K}. 
\item There exists $c_0>0$ independent of $m_0$ such that 
\begin{equation}\label{eq:FPHolder0}
    d_0(m(t),m(s)) \leq c_0 (1+\|b\|_{\infty})|t-s|^{\frac 12}.
\end{equation}
\end{enumerate}
\end{theorem}
\begin{remark}\ \label{rem:FP}

\begin{enumerate}[label=(\alph*)]
\item We want to consider general probability measures as the initial conditions in the Fokker--Planck equations. To this end we rework the above result in Theorem~\ref{th:FP} below. We note that $b\in (C_b((0,T),C^1_b(\mR^d)))^d$ is sufficient for getting mild and distributional solutions, thus the difference between the assumptions in Theorem~\ref{th:FPOEERJ} and Theorem~\ref{th:FP}. 
\item The H\"older regularity in time is stronger in \cite{MR4309434} than in \eqref{eq:FPHolder0}, because of additional assumption \cite[(L2)(i)]{MR4309434}. By inspecting the proof of \cite[Proposition~6.6 (b)]{MR4309434} it is easy to see that without this assumption we can always get $\frac 12$ regularity, cf. Remark~\ref{rem:K} (a).
\item The well-posedness and spatial regularity in Theorem~\ref{th:FPOEERJ} are proved in \cite{MR4309434} using Duhamel's principle, much like Lemma~\ref{lem:cd} above. In \cite{MR4309434} the authors assume \ref{eq:K} allowing $\alow\in(1,2)$, but in the results that we use here, $\alow=2$ can be added without any change in the arguments.
\end{enumerate}
\end{remark}
\begin{definition}
We say that $m$ is a distributional solution of \eqref{eq:FP} if for every $\phi\in C_b^{1,2}([0,T]\times \mR^d)$ and $0\leq t\leq T$,
\begin{equation}\label{eq:distFP}
    \int_{\mR^d} \phi(t,x)\, m(t,dx) - \int_{\mR^d} \phi(0,x)\, m_0(dx) = \int_0^t\int_{\mR^d} \big(\partial_t \phi(s,x) + \mL \phi(s,x) - D_x\phi(s,x)b(s,x)\big)\, m(s,dx)\, ds.
\end{equation}
\end{definition}
\begin{theorem}\label{th:FP}
 Assume that \ref{eq:K} holds, $b\in (C_b((0,T),C^1_b(\mR^d)))^d$ and $m_0\in \mP$. Then the following statements are true.
 \begin{enumerate}[label=(\alph*)]
\item The problem \eqref{eq:distFP} has a distributional solution $m\in C([0,T],\mP)$.
\item There exists a constant $c_0>0$ independent of $m_0$ such that
\begin{equation}\label{eq:FPHolder}
    d_0(m(s),m(t)) \leq c_0 (1+\|b\|_{\infty})|s-t|^{\frac 12},\quad s,t\in[0,T].
\end{equation}
\item The family $\{m(t):t\in[0,T]\}$ is tight and if we let $\psi$ be the tightness measuring function from Lemma~\ref{lem:tight} for $m_0$ and $\nu\textbf{1}_{B_1^c}$, then there exists $c>0$ such that for $t\in [0,T]$,
\begin{equation}\label{eq:FPtight}\begin{split}
    &\int_{\mR^d} \psi(x)\, m(t,dx)\\ \leq &\int_{\mR^d} \psi(x) \, m_0(dx) + \|D\psi\|_{\infty}\bigg(2 + cT\big(\|A\|_{2}+ \|B\|_{\infty}+\|b\|_{\infty} + \int_{|z|< 1}|z|^2\, \nu(dz)\big)\bigg) + T\int_{|z|\geq 1}\psi(z)\, \nu(dz).\end{split}
\end{equation}
\item The distributional solution is unique. 
\item
If the sequence $(b_n) \subset (C_b((0,T),C^1_b(\mR^d)))^d$ converges to $b$ uniformly (i.e., $\|b_n-b\|_{\infty}\to 0$), then the corresponding solutions $m_n$ converge in $C([0,T],\mP)$ to $m$ -- the solution corresponding to $b$.
\end{enumerate}
\end{theorem}
\begin{proof}
(a) The idea is to approximate the solutions for general $m_0\in\mP$ by solutions for regular initial measures obtained in \cite{MR4309434}. Let $\eta$ be a standard mollifier and let $m_0^{\eps} = m_0\ast\eta_\eps$. We have $m_0^\eps \in  C^2_b(\mR^d)$, therefore by the proof of \cite[Proposition~6.8 (a)]{MR4309434} (note that only $b,Db\in C_b(\mR^d)$ is used there) there exists $m^\eps$ which satisfies the Duhamel formula:
\begin{equation}\label{eq:Duh}
    m^\eps(t,x) = \int_0^t\int_{\mR^d} D_xK^*(t-s,x-y) b(s,y)m^\eps(s,y)\, dy\, ds +\int_{\mR^d} K^*(t,x-y)\, m_0^\eps (dy),\quad t\in [0,T],\, x\in \mR^d.
\end{equation}
Furthermore, $m^\eps(t)\in C_b(\mR^d)$ for $t\in [0,T]$ and $m^\eps \in C([0,T],L^1(\mR^d))$. By Lemma~\ref{lem:mildweak} $m^\eps$ is also a distributional solution, therefore by the proof of \cite[Proposition~6.6 (a)]{MR4309434} we get that $\int m^\eps(t,dx) = 1$ for each $t\in [0,T]$. Furthermore, if we test the equation with $\phi$ satisfying 
\begin{equation}\label{eq:Holmgren}
    \begin{cases}
        \partial_t \phi + \mL \phi - b D_x\phi = 0,\quad {\rm in}\ [0,t)\times \mR^d,\\
        \phi(t) = \xi,
    \end{cases}
\end{equation}
with nonnegative $\xi\in C^{2+\sigma}_b(\mR^d)$ (see Lemma~\ref{lem:cd}), then by the \ar{comparison principle\footnote{This is standard, see e.g. \cite[Lemma~6.2]{MR4309434} for a similar result. The proof there can easily be adapted to the current setting.}} we have $\phi(0)\geq  0$, so that
\begin{align*}
    \int_{\mR^d} \xi(x) m^\eps(t,x)\, dx = \int_{\mR^d} \phi(0,x) m^\eps_0(x)\, dx \geq 0.
\end{align*}
Therefore, $m^\eps\geq 0$. Overall, we get $m^\eps\in C([0,T],\mP)$.

We will show that any sequence $(m^{\eps_n})$ with $\eps_n\to 0^+$ is Cauchy in $C([0,T],\mP)$.
Let $\eps,\eps'>0$, $t\in (0,T]$, $\phi\in \Lip\cap C^{2+\sigma}_b(\mR^d)$, and let $\phi$ satisfy \eqref{eq:Holmgren}. By testing $m^\eps - m^{\eps'}$ with $\phi$ we get
\begin{align*}
    \int_{\mR^d} \xi(x) (m^\eps(t,x) - m^{\eps'}(t,x))\, dx = \int_{\mR^d} \phi(0,x) (m^\eps_0(x) - m^{\eps'}_0(x))\, dx &\leq \|\phi(0)\|_{C^1_b(\mR^d)}d_0(m^\eps_0,m^{\eps'}_0)\\ &\leq C\|\xi\|_{C^1_b(\mR^d)}d_0(m^\eps_0,m^{\eps'}_0),
\end{align*}
with $C$ independent of  $t,\eps,\eps'$ and $\xi$, see Lemma~\ref{lem:cd} (c). Since by Lemma~\ref{lem:moll} any sequence $(m_0^{\eps_n})$ is Cauchy, Lemma~\ref{lem:testapprox} gives that any $(m^{\eps_n})$ is a Cauchy sequence in $C([0,T],\mP)$. Thus, by completeness, $m^\eps$ converge in $C([0,T],\mP)$ to some $m$. By passing with $\eps$ to $0$ in \eqref{eq:distFP} for $m^\eps$ we conclude that $m$ is a distributional solution.

(b) The estimate \eqref{eq:FPHolder} is obtained easily from the analogue for $m^\eps$ in Theorem~\ref{th:FPOEERJ} and the triangle inequality:
$$d_0(m(s),m(t))\leq d_0(m(s),m^\eps(s)) + d_0 (m^\eps(s),m^\eps(t)) + d_0(m^\eps(t),m(t)) \leq C|s-t|^{\frac 12} + c(\eps),$$ where $c(\eps)\to 0$ as $\eps\to 0^+$.

(c) The proof of this part is exactly the same as in \cite[Proposition~6.6 (c)]{MR4309434}, therefore we skip it (note that only the distributional formulation is used there).

(d) To show the uniqueness, we use a Holmgren-type duality argument. If $m$ and $m'$ are distributional solutions of \eqref{eq:FP}, then their difference $\wtm = m-m'$ is a distributional solution of \eqref{eq:FP} with $m_0 = 0$. Let $t\in[0,T]$ and assume that $\xi\in C_c^\infty(\mR^d)$. Then, by Lemma~\ref{lem:cd} there exists a classical solution to \eqref{eq:Holmgren}.
By using the above $\phi$ in the distributional formulation \eqref{eq:distFP} we find that $\int_{\mR^d} \xi \wtm (t)=0$. Since $t\in [0,T]$ and $\xi\in C_c^\infty(\mR^d)$ are arbitrary, it follows by Lemma~\ref{lem:d0test} that $m=m'$.

(e) Let $\wtm_n = m - m_n$. Then by subtracting the equations we find that $\wtm_n$ satisfies the following problem in the distributional sense:
\begin{equation*}
    \begin{cases}
    \partial_t \wtm_n - \mL^* \wtm_n -\Div(b\wtm_n) = \Div((b_n - b)m_n),\quad {\rm in }\ (0,T)\times \mR^d,\\
    \wtm_n(0) = 0. \end{cases}
\end{equation*}
We take the test function $\phi$ satisfying \eqref{eq:Holmgren} with \ar{$\xi\in C^{2+\sigma}_b(\mR^d)$} such that $\|\xi\|_{C^1_b(\mR^d)} \leq 1$. Then we get
$$\bigg|\int_{\mR^d} \xi(x)\, \wtm_n(t,dx)\bigg| = \bigg|\int_0^t \int_{\mR^d} D\phi(s,x)(b(s,x)-b_n(s,x))\, m_n(s,dx)\, ds\bigg| \leq T\|D\phi\|_{\infty} \|b_n - b\|_{\infty}.$$
By Lemma~\ref{lem:cd} (c) $\|D\phi\|_{\infty}\leq C\|\xi\|_{C^1_b(\mR^d)} \leq C$ with $C$ independent of $t$ and $\xi$ which, together with Lemma~\ref{lem:testapprox}, proves that $m_n$ converges to $m$ in $C([0,T],\mP)$.
\end{proof}
\subsection{Smooth solutions of viscous Hamilton--Jacobi equations}\label{sec:HJ}
The following result shows well-posedness and regularity for viscous Hamilton--Jacobi equations such as the one appearing in the MFG system \eqref{eq:MFG}. The existence and uniqueness part comes directly from \cite{MR4309434}, but the regularity estimates are slightly different. The result here is closer to being optimal, as it uses in full the regularity given by the heat kernel, i.e., we gain $\alow-\eps$ derivatives compared to the right-hand side.
\begin{theorem}\label{th:HJ}
Assume that \ref{eq:H}, \ref{eq:H2}, \ref{eq:H3}, and \ref{eq:K} hold. Let $\sigma\in (0,\alow-1)$, $f\in C([0,T],C^2_b(\mR^d))$, $u_0 \in C^{3+\sigma}_b(\mR^d)$. Then the Hamilton--Jacobi equation
$$\begin{cases}
\partial_t u - \mL u + H(x,u,Du) = f(t,x)\quad &{\rm in}\ (0,T)\times \mR^d,\\
u(0) = u_0\quad &{\rm in}\ \mR^d,
\end{cases}$$
has a unique classical solution $u\in B([0,T],C^{3+\sigma}_b(\mR^d))\cap C([0,T],C^{3+\sigma-\eps}_b(\mR^d))$ (for arbitrarily small $\eps>0$). Furthermore, there exists $c$, depending only on $d,T$, $\sup_{t\in[0,T]}\|f(t,\cdot)\|_{C^2_b(\mR^d)},$ $\|u_0\|_{C^{3+\sigma}_b(\mR^d)}$, and the quantities in the assumptions \ref{eq:H}, \ref{eq:H2}, \ref{eq:H3}, and \ref{eq:K}, such that
\begin{equation}\label{eq:ubounds}\|\partial_t u\|_{\infty} + \sup\limits_{t\in[0,T]} \|u(t,\cdot)\|_{C^{3+\sigma}_b(\mR^d)} \leq c.\end{equation}
Furthermore, for every $\eps>0$ there exists a modulus of continuity independent of $f$ and $u_0$ such that
\begin{align}\label{eq:HJmodcont}
    \|\partial_t u(t,\cdot) - \partial_t u(s,\cdot)\|_{\infty} + \|u(t,\cdot) - u(s,\cdot)\|_{C^{3+\sigma-\eps}_b(\mR^d)} \leq C\sup\limits_{t\in[0,T]} \|u(t,\cdot)\|_{C^{3+\sigma}_b(\mR^d)}\omega(t-s),\quad s,t\in [0,T].
\end{align}
\end{theorem}
\begin{proof}[Sketch of the proof]
  The arguments are almost identical to the ones in \cite{MR4309434}, but we present them in order to see that including the case $\alow=2$ in \ref{eq:K} (cf. \cite[(L2)]{MR4309434}) does not change anything in the proof, i.e., the operator can also be local or mixed local-nonlocal. Furthermore, we want $C^{3+\sigma}$ regularity for the solution $u$ instead of $C^3$ as in \cite{MR4309434}, which makes the estimates slightly different. Let 
  \begin{align*}X = \{u\in B([0,T],C^{3+\sigma}_b(\mR^d)):  \sup\limits_{t\in[0,T]}\|u(t,\cdot)\|_{C^{3+\sigma}_b(\mR^d)} \leq R_1, \quad u\in C([0,T],C^{3+\sigma-\eps}_b(\mR^d))\},
  \end{align*}
  where $R_1$ is an explicit constant depending on the constant $\mathcal{K}$ in \ref{eq:K}, $\sup\limits_{t\in[0,T]} \|f(t,\cdot)\|_{C^{2}_b(\mR^d)}$, and $\|u_0\|_{C^{3+\sigma}_b(\mR^d)}$. Cf.  Proposition~5.8 in \cite{MR4309434} (the $f$-dependence of $R_1$ was overlooked there). We only show how to bound the derivatives for the fixed point map to get a self-map, and we show its contractivity. All the other arguments are the same. Let $\widetilde{u}\in X$ and let $u = S(\widetilde{u})$ be defined as
  \begin{align*}
      u(t,x) = K_t\ast u_0(x) + \int_0^t\int_{\mR^d} K(t-s,x-y)\big(H(y,\widetilde{u}(s,y),D\widetilde{u}(s,y)) - f(s,y)\big)\, dy\, ds.
  \end{align*}
  Since $\widetilde{u}\in X$, we can use the dominated convergence theorem (cf. the arguments following \eqref{eq:duhc1}) to find that for $k=1,2,3$ we have
  \begin{align}\label{eq:Dk}
      D^k_xu(t,x) = K_t\ast D^k_x u_0(x) + \int_0^t\int_{\mR^d} D_y K(t-s,x-y)D^{k-1}_y\big(H(y,\widetilde{u}(s,y),D\widetilde{u}(s,y)) - f(s,y)\big)\, dy\, ds.
  \end{align}
  Using the regularity of $H$ given by \ref{eq:H}\footnote{Note that the local bounds given there are sufficient, because on $[0,T]$ we control $\|u\|_{C^1_b(\mR^d)}$ with $u_0$ and $f$.}  and \ref{eq:K} we therefore get that
  \begin{align*}
      \|D^ku(t,\cdot)\|_{\infty} &\leq  \|D^k_x u_0\|_{\infty} + \big(C(H) \sup\limits_{t\in[0,T]}\|\widetilde{u}(t,\cdot)\|_{C^{k}_b(\mR^d)} + \sup\limits_{t\in[0,T]}\|f(t,\cdot)\|_{C^{k-1}_b(\mR^d)}\big) \int_0^T \int_{\mR^d}|D_yK(t-s,x-y)|\, dy\, ds\\
       &\leq  \|D^k_x u_0\|_{\infty} + C(K,T)\big(C(H) \sup\limits_{t\in[0,T]}\|\widetilde{u}(t,\cdot)\|_{C^{k}_b(\mR^d)} + \sup\limits_{t\in[0,T]}\|f(t,\cdot)\|_{C^{k-1}_b(\mR^d)}\big),
  \end{align*}
  with $C(K,T)=C(K)T^{1-\frac{1}\alow}\to 0$ as $T\to 0^+$.
Similarly (note that by \ref{eq:H} the second derivatives of $H$ are Lipschitz) we get that for any $\widetilde{u_1},\widetilde{u_2}\in X$,
\begin{align*}
    \|S(\widetilde{u_1})(t) -S(\widetilde{u_2})(t)\|_{C^3_b(\mR^d)} \leq C(K,T,H,R_1)\sup\limits_{t\in[0,T]}\|\widetilde{u_1}(t,\cdot) - \widetilde{u_2}(t,\cdot)\|_{C^3_b(\mR^d)}.
\end{align*}
Here, $C(K,T,H,R_1)$ is also small for small $T$. Furthermore, since $\sigma\in (0,\alow-1)$, by using \eqref{eq:Dk} and Remark~\ref{rem:K} (cf. the argument following \eqref{eq:interp}) we obtain
\begin{align*}
      &\frac{|D^3u(t,x+h) - D^3 u(t,x)|}{|h|^{\sigma}}\\ &\leq  \| u_0\|_{C^{3+\sigma}_b(\mR^d)} + \big(C(H) \sup\limits_{t\in[0,T]}\|\widetilde{u}(t,\cdot)\|_{C^{3}_b(\mR^d)} + \sup\limits_{t\in[0,T]}\|f(t,\cdot)\|_{C^{2}_b(\mR^d)}\big)\\
      &\qquad\qquad\qquad\qquad\qquad \times \int_0^T \int_{\mR^d}\frac{|D_yK(t-s,x+h-y) - D_yK(t-s,x-y)|}{|h|^{\sigma}}\, dy\, ds\\
        &\leq \| u_0\|_{C^{3+\sigma}_b(\mR^d)} + \widetilde{C}(K,T)\big(C(H) \sup\limits_{t\in[0,T]}\|\widetilde{u}(t,\cdot)\|_{C^{3}_b(\mR^d)} + \sup\limits_{t\in[0,T]}\|f(t,\cdot)\|_{C^{2}_b(\mR^d)}\big),
  \end{align*}
  for any $h\in\mR^d\setminus \{0\}$. Here, as before $\widetilde{C}(K,T)=C(K)T^{1-\frac{1+\sigma}{\alow}}\to 0$ as $T\to 0^+$. Analogously, we get
  \begin{align*}
    \|S(\widetilde{u_1})(t) -S(\widetilde{u_2})(t)\|_{C^{3+\sigma}_b(\mR^d)} \leq \widetilde{C}(K,T,H,R_1)\sup\limits_{t\in[0,T]}\|\widetilde{u_1}(t,\cdot) - \widetilde{u_2}(t,\cdot)\|_{C^3_b(\mR^d)},
\end{align*}
with $\widetilde{C}(K,T,H,R_1)$ small for $T$ small. This proves that $S$ is a contraction if $T$ is small enough.
Furthermore, by Lemma~\ref{lem:contsmooth} we get $u\in C([0,T],C^{3+\sigma-\eps}_b(\mR^d))$. Thus, by taking $T=T_0$ sufficiently small we get that $S$ is a self-map and a contraction. Hence, Banach's fixed point theorem gives the short-time existence and uniqueness of a mild solution $u$ on $[0,T_0]$.

Then we prove that $u$ is a classical solution which is done exactly as in \cite[Section~5]{MR4309434}. Going from the short-time existence to existence on any time interval $[0,T]$ is a bit delicate, because we only assume local boundedness for $H$ and its derivatives in $u$ and $p$, so the constants $c(H)$ above actually depend also on $\sup_{t\in[0,T]}\|u(t,\cdot)\|_{C^1_b(\mR^d)}$. It suffices to control this quantity in order to get a lower bound for the length of the short-time interval. This can be done using a result similar to \cite[Theorem~5.3 (b), (c)]{MR4309434} and proceeding as in \cite[Section~5.2]{MR4309434}. The procedure is rather standard so we skip the details, but the fact that an analogue of \cite[Theorem~5.3]{MR4309434} holds in our case is discussed in Remark~\ref{rem:Lipschitz} below.

Once we obtain the solution, \eqref{eq:ubounds} follows from the definition of $X$. Then, to get \eqref{eq:HJmodcont} we use Lemma~\ref{lem:contsmooth} and the original mild formulation of the problem.
\end{proof}
\begin{remark}\label{rem:Lipschitz}
The global $C^1_b$ bound mentioned in the last lines of the proof above is based on the result of Imbert \cite[Lemma~2]{MR2121115}. Although \cite{MR2121115} assumes the diffusion to be strictly nonlocal, the proof of \cite[Lemma~2]{MR2121115} works for any (linear and translation invariant) operator of the form \eqref{eq:operator}. Indeed, since our solutions are classical, the extension of the comparison principle \cite[Theorem~2]{MR2121115} used in this proof is straightforward.
Then, in order to extend  \cite[Lemma~2]{MR2121115} to mixed local-nonlocal operators it suffices to use an appropriate viscosity formulation, see e.g. Jakobsen and Karlsen \cite{MR2129093}. Alternatively, since our solutions are very regular, one can also use the Bernstein method, see e.g. Achdou et al. \cite[p. 31]{MR4214773} \ar{or \cite[Theorem~3.3]{2024arXiv240303884J} for an argument tailored to our case.}
\end{remark}
\section{Well-posedness of the mean field game system} \label{sec:mfg}

In this section we prove well-posedness results for the MFG system \eqref{eq:MFG}. These results refine and extend the results in \cite{MR4309434}. For the purpose of the master equation we need to allow probability measures as solutions in the Fokker-Planck equation. We will therefore work in the setting of classical-distributional solutions of \eqref{eq:MFG}. This will allow us to assume less regular data and coefficients than in \cite{MR4309434}. We also show a stronger regularizing effect on $u$ than \cite{MR4309434}, without adding any assumptions. Finally we mention that the results of this section seem to be the first ones to consider MFG systems with mixed local-nonlocal diffusion operators.

\begin{definition}\label{def:MFGsol} 
We say that the pair $(u,m)$ is a solution to the MFG system \eqref{eq:MFG} if $u\in C^{1,2}_b((t_0,T)\times \mR^d)\cap C([t_0,T]\times\mR^d)$ satisfies the Hamilton--Jacobi equation pointwise and $m\in C([t_0,T],\mP)$ satisfies the Fokker--Planck equation in the distributional sense.
\end{definition}
Before we proceed to the well-posedness for the MFG system, we give a variant of the Lasry--Lions monotonicity lemma for Hamiltonians $H$ which can depend on $u$. For $a\in \mR$ let $a^+ = a\vee 0$ and $a^- = -(a\wedge 0)$ and for $m\in \mathcal M(\mR^d)$ let $m= m^+ - m^-$ be its unique Jordan decomposition.
\begin{lemma}\label{lem:LL}
Assume that \ref{eq:M1} and \ref{eq:M3} hold. Let $(u_1,m_1)$ and $(u_2,m_2)$ be solutions to the MFG system \eqref{eq:MFG} on $(t_0,T)$ with initial measures $m_0^1, m_0^2$. Then there exists $C>0$ such that the following inequality holds:
\begin{align*}
    &\int_{t_0}^T \int_{\mR^d} (H(x,u_2,Du_1) - H(x,u_2,Du_2) - D_pH(x,u_2,Du_2)(Du_1 - Du_2))\, m_2(t)\, dt\\
    &+\int_{t_0}^T \int_{\mR^d} (H(x,u_1,Du_2) - H(x,u_1,Du_1) - D_pH(x,u_1,Du_1)(Du_2 -Du_1))\, m_1(t)\, dt\\
    \leq\, &C\bigg(\int_{\mR^d} (u_1(t_0,x)- u_2(t_0,x))^+(m_0^1(dx) - m_0^2(dx))^+ + \int_{\mR^d} (u_1(t_0,x)- u_2(t_0,x))^-(m_0^1(dx) - m_0^2(dx))^-\bigg).
\end{align*}
 Here and in the proof $u_i$ and $Du_i$ are evaluated at $(t,x)$ unless specified otherwise.
\end{lemma}
\begin{proof}
For brevity, let $u=u_1-u_2$, $m=m_1-m_2$, and $H_{(i)}=H_i(x,u_i(t,x),Du_i(t,x))$ for $i=1,2$. The first part of the proof is standard \cite{MR3967062,Lions} (a detailed computation is given e.g. in \cite[Appendix~A]{MR4309434}) and consists in using the equations satisfied by $u_i$ and $m_i$ and the monotonicity assumption for $F$ in \ref{eq:M1} to prove that for $t\in (t_0,T)$,
\begin{align}\label{eq:monot1}
    -\frac{d}{dt} \int_{\mR^d}u(t,x)\, m(t,dx) \geq \int_{\mR^d} (H_{(1)} - H_{(2)} - D_pH_{(2)}\cdot Du)\, m_2 +\int_{\mR^d} (H_{(2)} - H_{(1)} - D_pH_{(1)}\cdot(-Du))\, m_1.
\end{align}
We remark here that the approach in \cite[Lemma 3.1.2]{MR3967062} of approximating $m_0^i$ by measures with smooth densities, leads to a circular reasoning because it requires the stability of the MFG system. Instead, one should compute the above derivative using difference quotients and the distributional formulation for the Fokker--Planck equation.

By integrating \eqref{eq:monot1} with respect to $t$ and by using the monotonicity of $G$ in \ref{eq:M1} we further get that for $s\in [t_0,T]$,
\begin{equation}\label{eq:monot2}
\begin{split}
 \int_{\mR^d}u(s,x)\, m(s,dx) \geq &\int_{s}^T\int_{\mR^d} (H_{(1)} - H_{(2)} - D_pH_{(2)}\cdot Du)\, m_2(t)\, dt \\ &+\int_{s}^T\int_{\mR^d} (H_{(2)} - H_{(1)} - D_pH_{(1)}\cdot (-Du))\, m_1(t)\, dt.
 \end{split}
\end{equation}
We add the following term to both sides of \eqref{eq:monot2}:
\begin{align*}
    I(s):=-\int_{s}^T \int_{\mR^d} (H_{(1)} - H(x,u_2,Du_1))\, m_2(t)\, dt - \int_{s}^T \int_{\mR^d} (H_{(2)} - H(x,u_1,Du_2))\, m_1(t)\, dt.
\end{align*}
The right-hand side of \eqref{eq:monot2} then becomes 
\begin{align*}
    J(s):=&\int_{s}^T \int_{\mR^d} (H(x,u_2,Du_1) - H(x,u_2,Du_2) - D_pH(x,u_2,Du_2)(Du_1 - Du_2))\, m_2(t)\, dt\\
    &+\int_{s}^T \int_{\mR^d} (H(x,u_1,Du_2) - H(x,u_1,Du_1) - D_pH(x,u_1,Du_1)(Du_2 -Du_1))\, m_1(t)\, dt,
\end{align*}
so \eqref{eq:monot2} can be rewritten as
\begin{equation}\label{eq:monot3} \int_{\mR^d}u(s,x)\, m(s,dx) + I(s) \geq J(s), \quad s\in [t_0,T].
\end{equation}
By \ref{eq:M3} we have $H(x,u,p) = H_1(x,p) + H_2(x,u)$, thus
\begin{align*}
    I(s) = &-\int_s^T\int_{\mR^d} (H_2(x,u_1) - H_2(x,u_2))\, m_2(t,dx)\, dt
    - \int_s^T \int_{\mR^d} (H_2(x,u_2) - H_2(x,u_1))\, m_1(t,dx)\, dt\\
    =&\int_s^T\int_{\mR^d} (H_2(x,u_1) - H_2(x,u_2))\, m(t,dx)\, dt.
\end{align*}
Note that 
\begin{align*}
    I(s) = \int_s^T\int_{\mR^d} (H_2(x,u_1) - H_2(x,u_2))\, m(t,dx)\, dt &\leq \int_s^T\int_{\mR^d} (H_2(x,u_1) - H_2(x,u_2))^+\, m(t,dx)^+\, dt\\ &+ \int_s^T\int_{\mR^d} (H_2(x,u_1) - H_2(x,u_2))^-\, m(t,dx)^-\, dt,
\end{align*}
By the mean value theorem and \ref{eq:M3} we have 
\begin{align*}(H_2(x,u_1) - H_2(x,u_2))^{\pm} = (D_uH_2(x,\xi)(u_1-u_2))^{\pm} = D_uH_2(x,\xi)((u_1-u_2))^{\pm} \leq c_2(u_1-u_2)^{\pm} = c_2 u^{\pm}.
\end{align*}
Therefore,
\begin{align*}
    I(s) \leq c_2\bigg(\int_s^T\int_{\mR^d} u(t,x)^+\, m(t,dx)^+\, dt + \int_s^T\int_{\mR^d} u(t,x)^-\, m(t,dx)^-\, dt\bigg).
\end{align*}
Furthermore,
\begin{equation*}
\int_{\mR^d}u(s,x)\, m(s,dx) \leq \int_{\mR^d}u(s,x)^+\, m(s,dx)^+ + \int_{\mR^d}u(s,x)^-\, m(s,dx)^-,
\end{equation*}
therefore \eqref{eq:monot3} yields
\begin{equation*}
    \int_{\mR^d}u^+m^+ + \int_{\mR^d} u^-m^- + c_2\int_s^T\bigg(\int_{\mR^d}u^+m^+ + \int_{\mR^d} u^-m^-\bigg) \geq J(s),\quad s\in[t_0,T].
\end{equation*}
By Gr\"{o}nwall's inequality there exists $C>0$ such that 
\begin{equation*}
    C\bigg(\int_{\mR^d}u(t_0,x)^+\, m(t_0,x)^+ + \int_{\mR^d} u(t_0,x)^-\, m(t_0,dx)^-\bigg)  \geq J(t_0),
\end{equation*}
which ends the proof.
\end{proof}
\begin{theorem}\label{th:MFGwp}
Let $t_0\in[0,T)$, $m_0\in \mP$, $\sigma \in (0,\alow-1)$, $\eps\in (0,\sigma)$, and assume that \ref{eq:H},\ref{eq:H2},\ref{eq:H3},\ref{eq:K},\ref{eq:F}, and \ref{eq:G} hold. Then,
\begin{enumerate}[label=(\alph*)]
    \item the system \eqref{eq:MFG} has a solution $(u,m)$ such that $u\in B([t_0,T],C^{3+\sigma}_b(\mR^d))\cap C([t_0,T],C^{3+\sigma-\eps}_b(\mR^d))$ and $m\in C^{\frac 12}([t_0,T],\mP)$. Furthermore, the following estimates hold:
    \begin{align*}
        &\|\partial_t u\|_{L^{\infty}(\mR^d)} +\sup\limits_{t\in [t_0,T]} \|u(t,\cdot)\|_{C^{3+\sigma}_b(\mR^d)} \leq C(d,T,F,G,H,\mL,\sigma),\\
        &d_0(m(t),m(s))\leq C(d,T,F,G,H,\mL)|t-s|^{\frac 12},\quad t,s\in[t_0,T].
    \end{align*}
    In particular, the constants do not depend on $t_0$ and $m_0$.
    \item If in addition \ref{eq:M1} and \ref{eq:M3} are true, then the solution is unique.
    \end{enumerate}
\end{theorem}
\begin{proof}
(a) The proof uses a Schauder--Tychonoff fixed point argument very similar to the one in \cite[Theorem 3.2]{MR4309434} (see also \cite{MR4214773,Lions}), so we will be brief. Let $\psi$ be the tightness measuring function of Lemma~\ref{lem:tight} for $m_0$ and $\nu\textbf{1}_{B(0,1)^c}$. We define
$$X = \bigg\{\mu\in C([t_0,T],\mP): \mu(t_0) = m_0,\, \sup\limits_{t\in [t_0,T]} \int_{\mR^d} \psi(x)\, \mu(t,dx) \leq C_1,\, \sup\limits_{s\neq t} \frac{d_0(\mu(s),\mu(t))}{|s-t|^{1/2}} \leq C_2\bigg\},$$
where $C_1,C_2>0$ will be specified later. By standard arguments including the Prokhorov and Arzel\`{a}--Ascoli theorems we get that $X$ is convex and compact in $C([t_0,T],\mP)$.

For $\mu\in X$ we let $u\in B([t_0,T],C^{3+\sigma}_b(\mR^d))\cap C([t_0,T],C^{3+\sigma-\eps}_b(\mR^d))$ be the unique classical solution to the Hamilton--Jacobi equation of \eqref{eq:MFG} with $\mu$ in place of $m$. The solution exists by virtue of Theorem~\ref{th:HJ}, which we can use because of \ref{eq:H},\ref{eq:H2},\ref{eq:H3},\ref{eq:K}, \ref{eq:F}, \ref{eq:G}, and the fact that $\mu\in C([t_0,T],\mP)$. We also get that for $b = D_pH(x,u,Du)$ we have $\|b\|_{\infty} \leq \widetilde{C}$, with $\widetilde{C}$ independent of $\mu$. We then let $m$ be the unique distributional solution of the Fokker--Planck equation of \eqref{eq:MFG} with the function $u$ obtained above, see Theorem~\ref{th:FP}. Next we define the mapping $S\colon X \to C([0,T],\mP)$ as $T(\mu) = m$. By taking $C_1$ and $C_2$ to be the right-hand sides of estimates \eqref{eq:FPHolder} and \eqref{eq:FPtight} respectively, with $\|b\|_{\infty}$ replaced by $\widetilde{C}$ (as was done in \cite{MR4309434}), we ensure that $S\colon X\to X$.

In order to get the continuity of $S$, let $\mu_n\to \mu$ in $C([t_0,T],\mP)$. Using \ref{eq:F} and \ref{eq:G}, as in \cite{MR4309434} we get that the corresponding solutions of the Hamilton--Jacobi equation $u_n$, with $Du_n, D^2u_n,\partial_tu_n,$ and $\mL u_n$ converge uniformly to $u,Du,D^2u,\partial_t u$, and $\mL u$ respectively, and $u$ solves the H--J equation with $\mu$ in place of $m$. Now, let $m_n$ and $m$ be the solutions of the Fokker--Planck equation with $u_n$ and $u$ respectively. Then the convergence $m_n\to m$ in $C([0,T],\mP)$ follows by taking $b_n(t,x) = D_pH(x,u_n(t,x),Du_n(t,x))$ and $b(t,x) = D_pH(x,u(t,x),Du(t,x))$ in Theorem~\ref{th:FP} (e).

By the Schauder--Tychonoff fixed point theorem $S$ has a fixed point in $X$, from which we conclude that the MFG system \eqref{eq:MFG} has a solution. The regularity properties of $m$ and $u$ follow from respectively, the construction of the solution (the definition of $X$) and Theorem~\ref{th:HJ}, see also Remark~\ref{rem:t0}.

(b) Assume that $(u_1,m_1)$ and $(u_2,m_2)$ solve the MFG system with the same data. We use Lemma~\ref{lem:LL}: by the strict convexity in $p$ for $\widetilde{H}$ in \ref{eq:M3} we get that
\begin{equation*}
    \int_{t_0}^T\int_{\mR^d} |Du_1(t,x) - Du_2(t,x)|^2 (m_1(t,dx) + m_2(t,dx))=0,
\end{equation*}
therefore for almost every $t$ we have $Du_1(t) = Du_2(t)$ both $m_1(t)$- and $m_2(t)$-almost everywhere. Since $D_pH$ depends only on $x$ and $Du$, we find that $m_1$ and $m_2$ solve exactly the same Fokker--Planck equation, therefore by the uniqueness in Theorem~\ref{th:FP} we obtain $m_1 = m_2$. From this we get that $u_1$ and $u_2$ solve the same Hamilton--Jacobi equation, therefore by Theorem~\ref{th:HJ} they are equal as well, which proves the uniqueness for the whole system.
\end{proof}
\section{Further regularity for the MFG system}\label{sec:further}
The proof of the well-posedness for the master equation requires further regularity and stability results for the solutions of the MFG system. In particular we need to establish that the dependence of $u$ on the initial measure $m_0$ is not only continuous, but also differentiable. This section is devoted to proving these properties. The structure of the results and some of the proofs follow \cite[Chapter 3]{MR3967062}.
The results obtained below will imply differentiability properties for the following function, which turns out to be a solution to the master equation.
\begin{definition}\label{def:Udef}Let $(u,m)$ solve \eqref{eq:MFG} with $t_0\in[0,T]$ and $m_0\in \mP$ given. We define
\begin{equation}\label{eq:Udef}
U(t_0,x,m_0) = u(t_0,x),\quad x\in\mR^d.
\end{equation}
\end{definition}
\begin{remark}
We have $\mL_x U(t_0,x,m_0) = \mL_x u(t_0,x)$ and $\partial^\alpha_xU(t_0,x,m_0) = \partial^\alpha_x u(t_0,x)$ for any multi-index $\alpha$ for which the right-hand side makes sense.
\end{remark}
\subsection{Lipschitz continuity in the measure variable}
\begin{theorem}\label{th:Lip}
Let $\sigma\in(0,\alow-1)$ and assume \ref{eq:H}, \ref{eq:H2}, \ref{eq:H3}, \ref{eq:K}, \ref{eq:F}, \hyperref[eq:F2]{(F2(1,$\sigma$))}, \ref{eq:G}, \hyperref[eq:G2]{(G2(1,$\sigma$))}, \ref{eq:M1}, and \ref{eq:M2}. Let $(u_1,m_1)$ and $(u_2,m_2)$ be the solutions to the MFG system \eqref{eq:MFG} on $(t_0,T)$ with initial measures $m_0^1$ and $m_0^2$ respectively. Then there exists $C>0$ independent of $m_0^1,m_0^2$ such that
\begin{equation*}
    \sup\limits_{t\in[t_0,T]} \bigg(d_0(m_1(t),m_2(t)) + \|u_1(t,\cdot) - u_2(t,\cdot)\|_{C^{3+\sigma}_b(\mR^d)}\bigg) \leq Cd_0(m_0^1,m_0^2).
\end{equation*}
As a consequence,
\begin{equation*}
\|U(t,\cdot,m) - U(t,\cdot,m')\|_{C^{3+\sigma}_b(\mR^d)} \leq Cd_0(m,m').
\end{equation*}
\end{theorem}
\begin{remark}\label{rem:LL}
Assumptions \ref{eq:M1}, \ref{eq:M2} yield the standard Lasry--Lions monotonicity estimate (see \cite[Appendix~A]{MR4309434} and \cite[Lemma 3.1.2]{MR3967062}):
\begin{align*}
    &\int_{t_0}^T \int_{\mR^d} (H(x,Du_1) - H(x,Du_2) - D_pH(x,Du_2)(Du_1 - Du_2))\, m_2(t)\, dt\\
    &+\int_{t_0}^T \int_{\mR^d} (H(x,Du_2) - H(x,Du_1) - D_pH(x,Du_1)(Du_2 -Du_1))\, m_1(t)\, dt\\
    \leq\, &C\int_{\mR^d} (u_1(t_0,x)- u_2(t_0,x))(m_0^1(dx) - m_0^2(dx)).
\end{align*}
The monotonicity condition in Lemma~\ref{lem:LL} which allows $u$-depending Hamiltonians $H$,  despite being sufficient for uniqueness arguments for the MFG system \eqref{eq:MFG}, does not seem strong enough to give  
stability results like Theorem~\ref{th:Lip} that are needed for the analysis of the master equation \eqref{eq:master}. The reason is that 
an estimate in terms of $d_0((m_0^1- m_0^2)^+,0)$ and $d_0((m_0^1- m_0^2)^-,0)$ given by Lemma~\ref{lem:LL}, is in general not enough to get the estimate in terms of $d_0(m_0^1- m_0^2,0)=d_0(m_0^1,m_0^2)$ in Theorem~\ref{th:Lip}.
Therefore we drop the $u$-dependence in $H$ and assume \ref{eq:M2} in place of \ref{eq:M3}.  Hence in what follows we have $H = H(x,Du)$. 
\end{remark}
\begin{proof}[Proof of Theorem~\ref{th:Lip}] 
1)\quad  We first establish a structure estimate. By using the Lasry--Lions condition of Remark~\ref{rem:LL}, \ref{eq:M2}, and the definition of $d_0$, we find that
\begin{equation}\label{eq:structure}\begin{split}
    \int_{t_0}^T\int_{\mR^d} |Du_1(t,x) - Du_2(t,x)|^2 (m_1(t,dx) + m_2(t,dx))\, dt
    \leq\, &C\int_{\mR^d} (u_1(t_0,x)- u_2(t_0,x))(m_0^1(dx) - m_0^2(dx))\\
    \leq\, &C\|u_1(t_0,\cdot) - u_2(t_0,\cdot)\|_{C^1_b(\mR^d)} d_0(m_0^1,m_0^2).
\end{split}\end{equation}
\smallskip

\noindent 2)\quad We will now estimate $d_0(m_1,m_2)$. By assumption, $m_1$ and $m_2$ satisfy the Fokker--Planck equation in the distributional sense \eqref{eq:distFP}, thus for every $\phi\in C^{1,2}_b([t_0,T]\times \mR^d)$ and $t\in [t_0,T]$,
\begin{align*}
    \int_{\mR^d}\phi(t,x)(m_1(t,dx) - m_2(t,dx)) &= \int_{\mR^d}\phi(t_0,x)(m_0^1(dx) - m_0^2(dx))\\ &+ \int_{t_0}^t \int_{\mR^d} \big(\partial_t \phi(s,x)+ \mL \phi(s,x) - D_x\phi(s,x)D_pH(x,Du_1)\big) (m_1(s,dx)-m_2(s,dx))\, ds\\
    &+\int_{t_0}^t\int_{\mR^d}D_x\phi(s,x)(D_pH(x,Du_2) - D_pH(x,Du_1))\, m_2(s,dx)\, ds.
\end{align*}
Fix an arbitrary $\phi_0\in \Lip{}\cap C^4_b(\mR^d)$. Let $\phi$ be the solution of the linear equation \eqref{eq:cd} on $(t_0,t)$ with $\phi(t) = \phi_0$, $b(t,x) = D_pH(x,Du_1)$, and $f=0$, which exists by Lemma~\ref{lem:cd} and the fact that $Du_1$ has two continuous and bounded spatial derivatives (see Theorem~\ref{th:MFGwp}). Then the middle integral above vanishes, and by \ref{eq:H},
\begin{align}
     &\bigg|\int_{\mR^d}\phi_0(x)(m_1(t,dx) - m_2(t,dx))\bigg|\nonumber\\ \leq\, &\bigg|\int_{\mR^d}\phi(t_0,x)(m_0^1(dx) - m_0^2(dx))\bigg|+\bigg|\int_{t_0}^t\int_{\mR^d}D_x\phi(s,x)(D_pH(x,Du_2) - D_pH(x,Du_1))\, m_2(s,dx)\, ds\bigg|\nonumber\\
    \leq\, &C\bigg(d_0(m_0^1,m_0^2) + \int_{t_0}^t\int_{\mR^d}|Du_1 - Du_2|\, m_2(s,dx)\, ds\bigg).\label{eq:d0midstep}
\end{align}
Note that by \eqref{eq:gradest} there is $C>0$ independent of $\phi_0$ such that $\sup_{s\in [t_0,t]} \|\phi(s,\cdot)\|_{C^1_b(\mR^d)} \leq C\|\phi_0\|_{C^1_b(\mR^d)} \leq 2C$.
By Jensen's inequality and \eqref{eq:structure} we can estimate the last term as follows:
\begin{align*}
    \int_{t_0}^t\int_{\mR^d}|Du_1 - Du_2|\, m_2(s,dx)\, ds &= \bigg(\int_{t_0}^t\int_{\mR^d}|Du_1 - Du_2|\, m_2(s,dx)\, ds\bigg)^{2\cdot \frac 12} \\
    &\leq (T-t_0)\bigg(\int_{t_0}^t\int_{\mR^d}|Du_1 - Du_2|^2(m_1(s,dx) + m_2(s,dx))\, ds\bigg)^{\frac 12}\\
    &\leq C\|u_1(t_0,\cdot) - u_2(t_0,\cdot)\|_{C^1_b(\mR^d)}^{\frac 12} d_0(m_0^1,m_0^2)^{\frac 12}\\
    &\leq C\sup\limits_{t\in [t_0,T]}\|u_1(t,\cdot) - u_2(t,\cdot)\|_{C^{3+\sigma}_b(\mR^d)}^{\frac 12} d_0(m_0^1,m_0^2)^{\frac 12}.
\end{align*}
By this and \eqref{eq:d0midstep} it follows that
\begin{equation*}
    \bigg|\int_{\mR^d}\phi_0(x)(m_1(t,dx) - m_2(t,dx))\bigg| \leq C\bigg(d_0(m_0^1,m_0^2) + \sup\limits_{t\in [t_0,T]}\|u_1(t,\cdot) - u_2(t,\cdot)\|_{C^{3+\sigma}_b(\mR^d)}^{\frac 12} d_0(m_0^1,m_0^2)^{\frac 12}\bigg),
\end{equation*}
so by taking the supremum over $\phi_0\in \Lip{}\cap C^4_b(\mR^d)$ (see Lemma~\ref{lem:testapprox}) and $t$, we obtain
\begin{equation}\label{eq:d0estimate}
    \sup\limits_{t\in[t_0,T]}d_0(m_1(t),m_2(t))\leq C\bigg(d_0(m_0^1,m_0^2) + \sup\limits_{t\in [t_0,T]}\|u_1(t,\cdot) - u_2(t,\cdot)\|_{C^{3+\sigma}_b(\mR^d)}^{\frac 12} d_0(m_0^1,m_0^2)^{\frac 12}\bigg).
\end{equation}
\smallskip

\noindent 3)\quad Now consider $w = u_1 - u_2$. By the Hamilton--Jacobi equations for $u_1$ and $u_2$,
\begin{equation*}
\begin{cases}
\partial_t w + \mL w + b(t,x)\cdot Dw = f(t,x)\quad &{\rm in}\ (t_0,T)\times \mR^d,\\
w(T) = w_T\quad &{\rm in}\ \mR^d,
\end{cases}
\end{equation*}
where
\begin{align*}
&b(t,x) = \int_0^1 D_pH(x,\lambda Du_2 + (1-\lambda) Du_1)\, d\lambda,\\
&f(t,x) = F(x,m_1(t)) - F(x,m_2(t)) = \int_0^1\int_{\mR^d} \dm{F}(x,\lambda m_1(t) + (1-\lambda)m_2(t),y)(m_2(t,dy) - m_1(t,dy))\, d\lambda,\ {\rm and}\\
 &w_T(x) = \int_0^1\int_{\mR^d} \dm{G}(x,\lambda m_1(T) + (1-\lambda)m_2(T),y)(m_2(T,dy) - m_1(T,dy))\, d\lambda.
\end{align*}
By the bounds on $\|u\|_{C^3_b(\mR^d)}$ in Theorem~\ref{th:MFGwp} and by \ref{eq:H} there exists $C = C(d,T,F,G,H,\mL)$ such that
\begin{equation*}\label{eq:bbound}
    \sup\limits_{t\in[t_0,T]}\|b(t,x)\|_{C^2_b(\mR^d)} \leq C.
\end{equation*}
Furthermore, by \hyperref[eq:F2]{(F2(1,$\sigma$))} and \hyperref[eq:G2]{(G2(1,$\sigma$))} respectively, we get
\begin{align*}
    &\sup\limits_{t\in[t_0,T]} \|f(t,\cdot)\|_{C^2_b(\mR^d)} \leq  C(F)\sup\limits_{t\in[t_0,T]} d_0(m_1(t),m_2(t)), \ \text{and}\\
    &\|w_T\|_{C^{3+\sigma}_b(\mR^d)} \leq C(G)d_0(m_1(T),m_2(T)) \leq C\sup\limits_{t\in[t_0,T]} d_0(m_1(t),m_2(t)). 
\end{align*}

Therefore, by Lemma~\ref{lem:cd} we obtain
\begin{equation*}\label{eq:ubound}
    \sup\limits_{t\in [t_0,T]}\|u_1(t,\cdot) - u_2(t,\cdot)\|_{C^{3+\sigma}_b(\mR^d)} = \sup\limits_{t\in [t_0,T]}\|w(t,\cdot)\|_{C^{3+\sigma}_b(\mR^d)}\leq C\sup\limits_{t\in[t_0,T]} d_0(m_1(t),m_2(t)),
\end{equation*}
where $C = C(d,T,F,G,H,\mL)$. Thus, by \eqref{eq:d0estimate},
\begin{equation}\label{eq:uimpl}
    \sup\limits_{t\in [t_0,T]}\|u_1(t,\cdot) - u_2(t,\cdot)\|_{C^{3+\sigma}_b(\mR^d)} \leq C\bigg(d_0(m_0^1,m_0^2) + \sup\limits_{t\in [t_0,T]}\|u_1(t,\cdot) - u_2(t,\cdot)\|_{C^{3+\sigma}_b(\mR^d)}^{\frac 12} d_0(m_0^1,m_0^2)^{\frac 12}\bigg).
\end{equation}
By using $ab \leq \frac 12(\frac 1C a^2 + Cb^2)$ and rearranging we get
\begin{equation*}
    \sup\limits_{t\in [t_0,T]}\|u_1(t,\cdot) - u_2(t,\cdot)\|_{C^{3+\sigma}_b(\mR^d)} \leq Cd_0(m_0^1,m_0^2).
\end{equation*}
By inserting this estimate into \eqref{eq:d0estimate} we get 
\begin{equation*}
    \sup\limits_{t\in [t_0,T]}d_0(m_1(t),m_2(t)) \leq Cd_0(m_0^1,m_0^2),
\end{equation*}
which ends the proof.
\end{proof}
\subsection{Estimates on an auxiliary forward-backward linear system}
In order to obtain existence of $\dm{U}$ and its regularity in $m_0$ and $t_0$, we will study several variants of linearized systems coming from the MFG system \eqref{eq:MFG}. To cover all cases we consider a quite general linear forward-backward system in \eqref{eq:linsyst} below.

Let $V\colon (t_0,T)\times\mR^d\to \mR^d$, $b\colon (t_0,T)\times\mR^d\to \mR$, $z_T\colon \mR^d\to \mR$, $m\in C([t_0,T],\mP)$, $\rho_0\in C^{-k-1-\sigma}_b(\mR^d)$, and $c\in (L^1([t_0,T],C^{-k-\sigma}_b(\mR^d)))^d$ for some $k\in \{1,2,\ldots\}$ and $\sigma\in (0,\alow-1)$. Assume also that $\Gamma\colon (t_0,T)\times\mR^d\to \mR^{d\times d}$ and $\Gamma(t,x)$ is a symmetric matrix for every $(t,x)$. Furthermore, we stipulate that there exists $c_{\Gamma}>0$ such that for every $(t,x)$,
\begin{align}\label{eq:Gamma}
    c_{\Gamma}^{-1} Id \leq \Gamma(t,x) \leq c_{\Gamma} Id,
\end{align}
the inequality $A\leq B$ meaning that $B-A$ is non-negative definite, i.e., $((B-A)x,x) \geq 0$.
Consider the following forward-backward linear system: 
\begin{equation}\label{eq:linsyst}
\begin{cases}-\partial_t z - \mL z + V(t,x)\cdot Dz = \langle\dm{F}(x,m(t)),\rho(t)\rangle + b(t,x)\quad &{\rm in}\ (t_0,T)\times \mR^d,\\
\partial_t \rho - \mL^* \rho - \Div(\rho V) - \Div(m\Gamma Dz + c) = 0\quad &{\rm in}\ (t_0,T)\times \mR^d,\\
z(T,x) = \langle \dm{G}(x,m(T)),\rho(T)\rangle + z_T(x),\quad \rho(t_0) = \rho_0.
\end{cases}
\end{equation}
Here and below we use the expression $\langle \phi,\rho\rangle$ to denote the functional $\rho\in C^{-\gamma}_b(\mR^d)$ acting on $\phi \in C^{\gamma}_b(\mR^d)$ for some $\gamma>0$. We say that the pair $(z,\rho)$ is a classical solution of \eqref{eq:linsyst} if $z\in C([t_0,T],C^2_b(\mR^d))$ solves the first equation pointwise and $\rho \in C_b((t_0,T), C^{-k-1-\sigma}_b(\mR^d))$ solves the second equation in very weak sense, that is, $\rho(t_0) = \rho_0$ and for every $t\in (t_0,T)$ and $\phi$ such that $\phi(t_0),\phi(t)\in C^{k+1+\sigma}_b(\mR^d)$ and $\phi$, $\partial_t \phi + \mL \phi - V D\phi\in L^\infty((t_0,t),C^{k+1+\sigma}_b(\mR^d))$, it satisfies
\begin{equation}
\begin{split}\label{eq:distcm2}
    &\langle \phi(t),\rho(t)\rangle - \langle\phi(t_0),\rho_0\rangle\\ = &\int_{t_0}^t\langle(\partial_t \phi + \mL \phi - V D\phi)(s),\rho(s)\rangle\, ds -\int_{t_0}^t\int_{\mR^d} D\phi\Gamma Dz(s,x)\, m(s,dx)\, ds - \int_{t_0}^t \langle D\phi(s),c(s)\rangle\, ds.
\end{split}
\end{equation}

With appropriate data, the first equation of \eqref{eq:linsyst} corresponds to $\langle\dm{U},\rho\rangle$, as we will show in Subsection~\ref{sec:dm}. 

\begin{theorem}\label{th:linsyst}
    Let $V,\Gamma,b,z_T,m,c$, and $\rho_0$ be defined as above, let $k\in \{1,2,\ldots\}$, $\sigma\in (0,\alow-1)$, and assume that \ref{eq:K}, \erj{\ref{eq:M12}}, \ref{eq:F2}, and \ref{eq:G2} hold. Suppose in addition that $\Gamma\in (C([t_0,T],C^1_b(\mR^d)))^{d\times d}$ satisfies \eqref{eq:Gamma}, $V\in (L^\infty([t_0,T],C^{k+1+\sigma}_b(\mR^d)))^d\cap (C([t_0,T],C^{k+1}_b(\mR^d)))^d $,  $c\in (L^1([t_0,T],C^{-k-\sigma+\eps}_b(\mR^d)))^d\cap (C([t_0,T],\mathfrak{C}^{-n}_b(\mR^d)))^d$\footnote{Recall that $\mathfrak{C}^{-n}_b(\mR^d)$ is the set of measure representable functionals in $C^{-n}_b(\mR^d)$.} for some $n\in \{0,1,\ldots\}$ and small $\eps>0$, $b\in L^{\infty}([t_0,T],C^{k+1+\sigma}_b(\mR^d))\cap C([t_0,T],C^{k+1}_b(\mR^d))$, $\rho_0\in \mathfrak{C}^{-k-1}_b(\mR^d)$, and $z_T\in C^{k+2+\sigma}_b(\mR^d)$. Then \eqref{eq:linsyst} has a unique classical solution $(z,\rho)$ such that for every $\delta\in (0,\sigma)$ we have $z\in B([t_0,T],C^{k+2+\sigma}_b(\mR^d))\cap C([t_0,T],C^{k+2+\sigma-\delta}_b(\mR^d))$ and $\rho\in C([t_0,T],C^{-k-1-\sigma}_b(\mR^d))$. Furthermore, if we let 
\begin{align*}
    M = \|z_T\|_{C^{k+2+\sigma}_b(\mR^d)} + \|\rho_0\|_{C^{-k-1-\sigma}_b(\mR^d)} + \|b\|_{k+1+\sigma} + \|c\|_{L^1(C^{-k-\sigma}_b(\mR^d))},
\end{align*}
where $\|b\|_{k+1+\sigma} = \sup_{t\in [0,T]} \|b(t,\cdot)\|_{C^{k+1+\sigma}_b(\mR^d)}$, then the following estimate holds true:
\begin{align}\label{eq:linestimates}
    \sup\limits_{t\in[t_0,T]}\|z(t,\cdot)\|_{C^{k+2+\sigma}_b(\mR^d)} +  \sup\limits_{t\in[t_0,T]}\|\rho(t)\|_{C^{-k-1-\sigma}_b(\mR^d)} \leq CM,
\end{align}
where $C= C(d,T,\alow,V,c_{\Gamma},F,G)$.
\end{theorem}
\begin{remark}\label{rem:cnegapprox}

 It is very likely that the result can be done for $k=0$. The limitation $k\geq 1$ appears because we use the classical solution setting of Lemma~\ref{lem:cd} for the test function $w$ in \eqref{eq:xi} below. 
\end{remark}
\begin{proof}[Proof of Theorem~\ref{th:linsyst}]
  We first prove the result for $\rho_0 \in C^2_b(\mR^d)\cap L^1(\mR^d)$, $m\in C([t_0,T],L^1(\mR^d))$, and $c\in (C([t_0,T],L^1(\mR^d)))^d$, using the Leray--Schauder theorem \cite[Theorem 10.3]{MR0473443}. To this end we establish a compact fixed point map and we obtain appropriate a priori estimates for the system. After solving the system for regular data, we handle the general data by convolving it with $\eta_\eps$ and we show that the solutions for the regularized problem converge to the solution for general data. Note that by Lemmas~\ref{lem:cnegconv} and \ref{lem:molltight}, the convolved data has the regularity listed above (here we use the assumption $c\in (C([t_0,T],\mathfrak{C}^{-n}_b(\mR^d)))^d$). We also recall from Section~\ref{sec:Holder} that $L^1(\mR^d)$ is embedded in $C^{-\gamma}_b(\mR^d)$ for every $\gamma>0$ with the convention $\langle f,\varphi\rangle = \int f\varphi$ for $\varphi\in L^1(\mR^d)$ and we have $\|\cdot\|_{C^{-\gamma}_b(\mR^d)} \leq \|\cdot\|_{L^1(\mR^d)}$. \medskip
  
  \noindent 1) {\em The fixed point map $T$.}\quad Let  \begin{align*}X =  C([t_0,T],C^{-k-1-\sigma}_b(\mR^d)).\end{align*} For every $\widetilde{\rho}\in X$, we want to define $\rho=T(\tilde\rho)\in X$ in the way that the fixed point of $T$ solves the linear system~\eqref{eq:linsyst}. In order to do this, we first solve the equation for $z$ with $\widetilde{\rho}$ in place of $\rho$. 
  Recall that $V\in (C([t_0,T],C^{k+1}_b(\mR^d)))^d$ and $b\in C([t_0,T],C^{k+1}_b(\mR^d))$. Furthermore, by \ref{eq:F2}, \ref{eq:G2} and Lemma~\ref{lem:dmf}, we get that
  \begin{align}\label{eq:d2f}
  \sup\limits_{t\in[t_0,T]}&\bigg\|\bigg\langle\dm{F}(\cdot,m(t)),\widetilde{\rho}(t)\bigg\rangle\bigg\|_{C^{k+1}_b(\mR^d)} \leq \sup\limits_{m\in \mP}\bigg\|\dm{F}(\cdot,m,\cdot)\bigg\|_{C^{k+1}_b(\mR^d,C^{k+1+\sigma}_b(\mR^d))}\sup\limits_{t\in [t_0,T]}\|\widetilde{\rho}(t)\|_{C^{-k-1-\sigma}_b(\mR^d)} \leq  C\|\widetilde{\rho}\|_{X},\\\label{eq:d3g}
  &\bigg\|\bigg\langle \dm{G}(\cdot,m(T)),\widetilde{\rho}(T)\bigg\rangle\bigg\|_{C^{k+2+\sigma}_b(\mR^d)} \leq \sup\limits_{m\in \mP}\bigg\|\dm{G}(\cdot,m,\cdot)\bigg\|_{C^{k+2+\sigma}_b(\mR^d,C^{k+1+\sigma}_b(\mR^d))}\|\widetilde{\rho}(T)\|_{C^{-k-1-\sigma}_b(\mR^d)} \leq C\|\widetilde{\rho}\|_{X}.
  \end{align}
  Finally, since $m\in C([t_0,T],\mP)$, we find that $\langle\dm{F}(\cdot,m(t)),\widetilde{\rho}(t)\rangle\in C([t_0,T],C^{k+1}_b(\mR^d))$. 
  Thus, by Lemma~\ref{lem:cd} we get that there exists a classical solution $z$ to the first equation in \eqref{eq:linsyst} with $\widetilde{\rho}$ in place of $\rho$, and we have \begin{align*}\sup\limits_{t\in[t_0,T]}\|z(t,\cdot)\|_{C^{k+2+\sigma}_b(\mR^d)} \leq C(\|z_T\|_{C^{k+2+\sigma}_b(\mR^d)} + \|\widetilde{\rho}\|_{X}+\|b\|_{k+1}).\end{align*} We now set $V_1 = V$ and $V_2 = m\Gamma Dz + c$, and we note that $V_1\in (C_b([t_0,T]\times\mR^d))^d$ and $V_2\in (C([t_0,T],L^1(\mR^d)))^d$ with
  \begin{align}\label{eq:V2L1}
      \sup\limits_{t\in[t_0,T]} \|V_2(t,\cdot)\|_{L^1(\mR^d)} &\leq \|\Gamma\|_{\infty}\|Dz\|_{\infty}+ \sup\limits_{t\in[t_0,T]}\|c(t,\cdot)\|_{L^1(\mR^d)}\\
      &\leq C(\|z_T\|_{C^{k+2+\sigma}_b(\mR^d)} + \|\widetilde{\rho}\|_{X} + \|b\|_{k+1}) + \sup\limits_{t\in[t_0,T]}\|c(t,\cdot)\|_{L^1(\mR^d)}.
  \end{align}
Therefore, by Lemma~\ref{lem:FPL1} (see also Lemma~\ref{lem:mildweak}), we get the existence of $\rho\in C([t_0,T],L^1(\mR^d))\subset X$ solving in the very weak sense the second equation of \eqref{eq:linsyst} with $z$ obtained above. We choose the only such $\rho$ which satisfies the Duhamel formula \eqref{eq:Duhamel} and define $T(\widetilde{\rho}) = \rho$. \medskip


\noindent 2) {\em Continuity and compactness of $T$.}\quad Since the equations are linear, it is easy to check that $T$ is continuous: consider a sequence $\widetilde{\rho}_n$ converging to $\widetilde{\rho}$ in $X$. By \eqref{eq:d2f}, \eqref{eq:d3g} and \eqref{eq:c3bound} we get that for the corresponding solutions $z_n$ and $z$ we have $\|Dz_n - Dz\|_{\infty} \to 0$ as $n\to\infty$. Then, by \eqref{eq:V2L1} and \eqref{eq:L1estimates} we find that the corresponding $\rho_n$ converge to $\rho$ in $C([t_0,T],L^1(\mR^d))$, hence also in $X$.

In order to get that $T$ maps any bounded set $B\subseteq X$ to a pre-compact subset of $X$, we use the Arzel\`a--Ascoli theorem. By Lemma~\ref{lem:FPL1} (d) we get equicontinuity in $t$ for $T(B)$. To get pre-compactness in $L^1(\mR^d)$ (and so in $C^{-k-1-\sigma}_b(\mR^d)$) for every fixed $t$, we use the Kolmogorov--Riesz theorem \cite[Theorem~5]{MR2734454}: Lemma~\ref{lem:FPL1} (b),(c), and (e) give boundedness, continuity of translations, and tightness of  $\{T(\widetilde{\rho})(t): \widetilde{\rho}\in B\}$ respectively, so all the assumptions of the Kolmogorov--Riesz theorem are satisfied. It yields that $\{T(\widetilde{\rho})(t): \widetilde{\rho}\in B\}$ is pre-compact in $L^1(\mR^d)$ for all $t\in [t_0,T]$. Thus we are in a position to use the Arzel\`{a}--Ascoli theorem, which yields that $T$ is a compact map. \medskip

\noindent 3) {\em The set $\mathcal{FP}_\lambda:=\{\rho\in X: \exists \lambda\in[0,1]\ \rho=\lambda T(\rho)\}$ appearing in the Leray--Schauder theorem is bounded and \eqref{eq:linestimates} holds.} 
\quad First note that $0\in\mathcal{FP}_\lambda$ so the set is nonempty. Next, if 
$\rho_\lambda\in X$ satisfies $\rho_\lambda = \lambda T(\rho_\lambda)$ for some $\lambda\in [0,1]$, then there is $z_\lambda$ such that $(z_\lambda,\rho_\lambda)$ solves the following system:
\begin{equation}\label{eq:sigmasyst}
\begin{cases}-\partial_t z_\lambda - \mL z_\lambda + V(t,x)\cdot Dz_\lambda = \langle\dm{F}(x,m(t)),\rho_\lambda(t)\rangle + b(t,x)\quad &{\rm in}\ (t_0,T)\times \mR^d,\\
\partial_t \rho_\lambda - \mL^* \rho_\lambda - \Div(\rho_\lambda V) - \lambda\Div(m\Gamma Dz_\lambda + c) = 0\quad &{\rm in}\ (t_0,T)\times \mR^d,\\
z_\lambda(T,x) = \langle\dm{G}(x,m(T)),\rho_\lambda(T)\rangle + z_T(x),\quad \rho_\lambda(t_0) = \lambda\rho_0.
\end{cases}
\end{equation}
We will prove that $\mathcal{FP}_\lambda$ is bounded in $X$ by showing that the a priori estimate \eqref{eq:linestimates} holds for any solution of \eqref{eq:sigmasyst}, independently of $\lambda$. We emphasize that the arguments below do not use the smoothness of $\rho_0,m,$ and $c$.

We start by establishing a structure estimate for solutions of \eqref{eq:sigmasyst}. By using $z_\lambda$ as a test function in $\rho_\lambda$-equation and then using the $z_\lambda$-equation, we get
\begin{align*}
    &\langle z_\lambda(T),\rho_\lambda(T)\rangle - \lambda\langle z_\lambda(t_0),\rho_0\rangle \\
    = &\int_{t_0}^T\langle(\partial_t z_\lambda + \mL z_\lambda - V Dz_\lambda)(s),\rho_\lambda(s)\rangle\, ds - \lambda\int_{t_0}^T\int_{\mR^d} Dz_\lambda\Gamma Dz_\lambda(s,x)\, m(s,dx)\, ds - \lambda\int_{t_0}^T \langle Dz_\lambda(s),c(s)\rangle\, ds\\
    = &-\int_{t_0}^T\bigg\langle\big\langle\dm{F}(m(s)),\rho_\lambda(s)\big\rangle_y + b(s),\, \rho_\lambda(s)\bigg\rangle_x\, ds - \lambda\int_{t_0}^T\int_{\mR^d} Dz_\lambda\Gamma Dz_\lambda(s,x)\, m(s,dx)\, ds -\lambda\int_{t_0}^T \langle Dz_\lambda(s),c(s)\rangle\, ds.
\end{align*}
By rearranging and using the terminal condition for $z_\lambda$ we get
\begin{align}\lambda\int_{t_0}^T\int_{\mR^d} Dz_\lambda\Gamma Dz_\lambda(s,x)\, m(s,dx)\, ds &=\lambda\langle z_\lambda(t_0),\rho_0\rangle  - \langle z_T,\rho_\lambda(T)\rangle -\int_{t_0}^T\langle  b(s),\rho_\lambda(s)\rangle\, ds-\lambda\int_{t_0}^T \langle Dz_\lambda(s),c(s)\rangle\, ds\\ \label{eq:M2orM1}
    &-\int_{t_0}^T\bigg\langle\big\langle\dm{F}(m(s)),\rho_\lambda(s)\big\rangle_y,\, \rho_\lambda(s)\bigg\rangle_x\, ds - \bigg\langle\big\langle\dm{G}(m(T)),\rho_\lambda(T)\big\rangle_y,\,\rho_\lambda(T)\bigg\rangle_x.
\end{align}
By \ref{eq:M12} the last two terms are negative, therefore we can estimate the left-hand side as follows:
\begin{align}
        \lambda\int_{t_0}^T\int_{\mR^d} Dz_\lambda\Gamma Dz_\lambda(s,x)\, m(s,dx)\, ds \leq&\,\lambda|\langle z_\lambda(t_0),\rho_0\rangle|  + |\langle z_T,\rho_\lambda(T)\rangle| +\int_{t_0}^T|\langle  b(s),\rho_\lambda(s)\rangle|\, ds+\lambda\int_{t_0}^T |\langle Dz_\lambda(s),c(s)\rangle|\, ds\nonumber\\
        \leq &\, \lambda \sup\limits_{t\in[t_0,T]} \|z_\lambda(t,\cdot)\|_{C^{k+2+\sigma}_b(\mR^d)}(\|\rho_0\|_{C^{-k-1-\sigma}_b(\mR^d)} + \|c\|_{L^1(C^{-k-\sigma}_b(\mR^d))})\label{eq:structurelin}\\ + &\sup\limits_{t\in[t_0,T]} \|\rho_\lambda(t)\|_{C^{-k-1-\sigma}_b(\mR^d)}(\|z_T\|_{C^{k+2+\sigma}_b(\mR^d)} + T\|b\|_{k+1+\sigma}).\nonumber
\end{align}
Now we will estimate $\rho_\lambda$ by duality. Let $\xi\in C^{k+1+\sigma}_b(\mR^d)$, $\tau\in (t_0,T]$, and let $w$ be the unique solution to the problem
\begin{align}\label{eq:xi}
        \partial_t w(t,x) + \mL w(t,x) - V(t,x) Dw(t,x) = 0,\ {\rm in}\ (t_0,\tau)\times \mR^d,\quad w(\tau) = \xi.
\end{align}
Existence and uniqueness of $w$ is given by Lemma~\ref{lem:cd}. We also get that \begin{align*}\sup\limits_{t\in[t_0,T]}\|w(t,\cdot)\|_{C^{k+1+\sigma}_b(\mR^d)} \leq C\|\xi\|_{C^{k+1+\sigma}_b(\mR^d)},\end{align*} with $C$ depending on $\sup_{t\in[t_0,T]}\|V(t,\cdot)\|_{C^k_b(\mR^d)}$, but not on $\xi$. Therefore, by testing the $\rho_\lambda$-equation in \eqref{eq:sigmasyst} with $w$ and using the triangle inequality we get 
\begin{align*}
    |\langle \xi,\rho_\lambda(\tau)\rangle| &\leq \lambda|\langle w(t_0),\rho_0\rangle| + \lambda\bigg|\int_{t_0}^\tau\int_{\mR^d} Dw\Gamma Dz_\lambda(s,x)\, m(s,dx)\, ds\bigg| + \lambda\bigg|\int_{t_0}^\tau \langle Dw(s),c(s)\rangle\, ds\bigg|\\
    &\leq C\lambda\|\xi\|_{C^{k+1+\sigma}_b(\mR^d)}(\|\rho_0\|_{C^{-k-1-\sigma}_b(\mR^d)}+\|c\|_{L^1(C^{-k-\sigma}_b(\mR^d))}) + \lambda\int_{t_0}^\tau\int_{\mR^d} |Dw\Gamma Dz_\lambda(s,x)|\, m(s,dx)\, ds. 
\end{align*}
Since $\Gamma$ is symmetric and non-negative definite, we have 
\begin{align*}
    |Dw\cdot \Gamma Dz_\lambda| = |\Gamma^{\frac 12}Dw\cdot  \Gamma^{\frac 12} Dz_\lambda|\leq |\Gamma^{\frac 12} Dw|^{\frac 12}|\Gamma^{\frac 12} Dz_\lambda|^{\frac 12} \leq (\Gamma Dw\cdot Dw)^{\frac 12}(\Gamma Dz_\lambda\cdot Dz_\lambda)^{\frac 12}.
\end{align*}
Hence, by the Cauchy--Schwarz inequality, the fact that $\Gamma \leq CId$, and the bounds on $Dw$ we get
\begin{align*}
|\langle \xi,\rho_\lambda(\tau)\rangle| &\leq C\lambda\|\xi\|_{C^{k+1+\sigma}_b(\mR^d)}(\|\rho_0\|_{C^{-k-1-\sigma}_b(\mR^d)}+\|c\|_{L^1(C^{-k-\sigma}_b(\mR^d))})\\
&+\lambda\bigg(\int_{t_0}^\tau\int_{\mR^d} Dz_\lambda\Gamma Dz_\lambda(s,x)\, m(s,dx)\, ds\bigg)^{\frac 12}\bigg(\int_{t_0}^\tau\int_{\mR^d} Dw\Gamma Dw(s,x)\, m(s,dx)\, ds\bigg)^{\frac 12}\\
&\leq C\|\xi\|_{C^{k+1+\sigma}_b(\mR^d)}\bigg(\lambda\|\rho_0\|_{C^{-k-1-\sigma}_b(\mR^d)} + \lambda \|c\|_{L^1(C^{-k-\sigma}_b(\mR^d))} + \lambda \bigg(\int_{t_0}^T\int_{\mR^d} Dz_\lambda\Gamma Dz_\lambda(s,x)\, m(s,dx)\, ds\bigg)^{\frac 12}\bigg).
\end{align*}
Now, by estimate \eqref{eq:structurelin} and the subadditivity of the square root we get
\begin{align*}
|\langle \xi,\rho_\lambda(\tau)\rangle| \leq C\|\xi\|_{C^{k+1+\sigma}_b(\mR^d)}\big[&\lambda\|\rho_0\|_{C^{-k-1-\sigma}_b(\mR^d)} + \lambda \|c\|_{L^1(C^{-k-\sigma}_b(\mR^d))}\\ &+ \lambda \sup\limits_{t\in[t_0,T]} \|z_\lambda(t,\cdot)\|_{C^{k+2+\sigma}_b(\mR^d)}^{\frac 12}(\|\rho_0\|_{C^{-k-1-\sigma}_b(\mR^d)}^{\frac 12} + \|c\|_{L^1(C^{-k-\sigma}_b(\mR^d))}^{\frac 12})\\ &+ \lambda^{\frac 12}\sup\limits_{t\in[t_0,T]} \|\rho_\lambda(t)\|_{C^{-k-1-\sigma}_b(\mR^d)}^{\frac 12}(\|z_T\|_{C^{k+2+\sigma}_b(\mR^d)}^{\frac 12} + T\|b\|_{k+1+\sigma}^{\frac 12})\big].
\end{align*}
If we divide by $\|\xi\|_{C^{k+1+\sigma}_b(\mR^d)}$ and take the supremum over $\tau\in [t_0,T]$ and $\|\xi\|_{C^{k+1+\sigma}_b(\mR^d)}\leq 1$, we find that
\begin{align*}
    \sup\limits_{t\in[t_0,T]} \|\rho_\lambda(t)\|_{C^{-k-1-\sigma}_b(\mR^d)}\leq C\big[&\lambda\|\rho_0\|_{C^{-k-1-\sigma}_b(\mR^d)} + \lambda \|c\|_{L^1(C^{-k-\sigma}_b(\mR^d))}\\ &+ \lambda \sup\limits_{t\in[t_0,T]} \|z_\lambda(t,\cdot)\|_{C^{k+2+\sigma}_b(\mR^d)}^{\frac 12}(\|\rho_0\|_{C^{-k-1-\sigma}_b(\mR^d)}^{\frac 12} + \|c\|_{L^1(C^{-k-\sigma}_b(\mR^d))}^{\frac 12})\\ &+ \lambda^{\frac 12}\sup\limits_{t\in[t_0,T]} \|\rho_\lambda(t)\|_{C^{-k-1-\sigma}_b(\mR^d)}^{\frac 12}(\|z_T\|_{C^{k+2+\sigma}_b(\mR^d)}^{\frac 12} + T\|b\|_{k+1+\sigma}^{\frac 12})\big].
\end{align*}
We use $ab \leq \frac 12(c a^2 + \frac 1c b^2)$ in the last two terms, with $c = \frac 1C$ for the $\rho_\lambda$ term and small $c$ to be determined later for the $z_\lambda$ term. After rearranging and using the definition of $M$ we obtain
\begin{align}\label{eq:rhoest}
    \sup\limits_{t\in[t_0,T]} \|\rho_\lambda(t)\|_{C^{-k-1-\sigma}_b(\mR^d)}\leq C\lambda\big[M + c`\sup\limits_{t\in[t_0,T]} \|z_\lambda(t,\cdot)\|_{C^{k+2+\sigma}_b(\mR^d)}\big].
\end{align}
Note that $\sup\limits_{t\in[t_0,T]} \|z_\lambda(t,\cdot)\|_{C^{k+2+\sigma}_b(\mR^d)}$ can be estimated in the exact same way as in step 1), by using $\rho_\lambda$ instead of $\widetilde{\rho}$ in \eqref{eq:d2f} and \eqref{eq:d3g}. This means that we have
\begin{align}\label{eq:zestimate}
    \sup\limits_{t\in[t_0,T]} \|z_\lambda(t,\cdot)\|_{C^{k+2+\sigma}_b(\mR^d)} \leq C(\|z_T\|_{C^{k+2+\sigma}_b(\mR^d)} + \sup\limits_{t\in [t_0,T]} \|\rho_\lambda(t)\|_{C^{-k-1-\sigma}_b(\mR^d)} + \|b\|_{k+1+\sigma}).
\end{align}
From this and \eqref{eq:rhoest}, by taking $c$ sufficiently small and rearranging we obtain
\begin{align}\label{eq:sigmaest}
     \sup\limits_{t\in[t_0,T]} \|\rho_\lambda(t)\|_{C^{-k-1-\sigma}_b(\mR^d)} \leq C\lambda M\leq CM.
\end{align}
Then, \eqref{eq:zestimate} yields an analogous estimate for $z_\lambda$:
\begin{align*}
    \sup\limits_{t\in[t_0,T]} \|z_\lambda(t,\cdot)\|_{C^{k+2+\sigma}_b(\mR^d)} \leq  C(1+\lambda)M\leq 2C M.
\end{align*}
In particular, we get \eqref{eq:linestimates} for $(z_\lambda,\rho_\lambda)$ independently of $\lambda$, and hence also the set $\mathcal{FP_\lambda}$ is bounded in $X$.  \medskip

\noindent 4) {\em Existence for regular data by Leray-Schauder fixed point theorem.}\quad By step 2) $T$ is a continuous and compact map on $X$, and by step 3) the set $\mathcal{FP}_\lambda=\{\rho\in X: \exists \lambda\in[0,1]\ \rho=  \lambda T(\rho)\}$ is bounded in $X$. Therefore, by the Leray--Schauder theorem we get the existence of a fixed point of $T$. \medskip

\noindent 5) {\em Existence for general data by approximation.} \quad We will now resolve the case of general $m$, $c$, and $\rho_0$ satisfying the assumptions of the theorem. Let $m_n= m\ast \eta_{\frac 1n}$ and $c_n = c\ast \eta_{\frac 1n}$, $\rho_0^n = \rho_0 \ast \eta_{\frac 1n}$ be their approximating sequences. By Lemmas~\ref{lem:cnegapprox} and \ref{lem:moll} they converge to $m,c$, and $\rho_0$ in $C([t_0,T],\mP)$, $(L^1([t_0,T],C^{-k-\sigma}_b(\mR^d)))^d$, and $C^{-k-1-\sigma}_b(\mR^d)$ norms respectively as $n\to\infty$. For natural numbers $n\neq n'$ the pair $(z_{n,n'},\rho_{n,n'}):= (z_n - z_{n'},\rho_n - \rho_{n'})$ solves the system \eqref{eq:linsyst} with $m=m_n$ and 
\begin{equation}\label{eq:kkprime}
\begin{split}
    b_{n,n'}(t,x) &= \bigg\langle \dm{F}(x,m_n(t)),\rho_{n'}(t)\bigg\rangle - \bigg\langle \dm{F}(x,m_{n'}(t)),\rho_{n'}(t)\bigg\rangle,\\
    z_T^{n,n'}(x) &= \bigg\langle\dm{G}(x,m_n(T)),\rho_{n'}(T)\bigg\rangle - \bigg\langle \dm{G}(x,m_{n'}(T)),\rho_{n'}(T)\bigg\rangle,\\
    c_{n,n'}(t,x)&= (m_{n'} - m_n)\Gamma(t,x)Dz_{n'}(t,x) + c_n - c_{n'},\\
    \rho_0^{n,n'} &= \rho_0^{n} - \rho_0^{n'}.
\end{split}
\end{equation}
By \ref{eq:F2} and \ref{eq:G2} we get that 
\begin{align*}
\|b_{n,n'}\|_{k+1+\sigma} + \|z^{n,n'}_T\|_{C^{k+2+\sigma}_b(\mR^d)} \leq (C(F)+C(G))\sup\limits_{t\in[t_0,T]} d_0(m_{n}(t),m_{n'}(t))\sup\limits_{t\in[t_0,T]}\|\rho_{n'}(t,\cdot)\|_{C^{-k-1-\sigma}_b(\mR^d)}.
\end{align*}
Furthermore, the first term in $c_{n,n'}$ can be estimated as follows:
\begin{align*}
    \|(m_{n'} - m_n)\Gamma Dz_{n'}\|_{-k-\sigma}&\leq  \|(m_{n'} - m_n)\Gamma Dz_{n'}\|_{-1}\quad \\
    &\leq \sup\limits_{\substack{t\in[t_0,T]\\ \|\phi\|_{C^1_b(\mR^d)}=1}} \bigg|\int_{\mR^d} \Gamma(t,x)Dz_{n'}(t,x)\phi(x)(m_{n'}(dx) - m_{n}(dx))\bigg|\\
    &\leq C\sup\limits_{t\in[t_0,T]}\|\Gamma(t,\cdot)\|_{C^1_b(\mR^d)}\sup\limits_{t\in[t_0,T]}\|z_{n'}(t,\cdot)\|_{C^{2}_b(\mR^d)}\sup\limits_{t\in[t_0,T]}d_0(m_{n'}(t),m_{n}(t)). 
\end{align*}
Therefore, by \eqref{eq:linestimates} we find that
\begin{align*}
    &\sup\limits_{t\in[t_0,T]} \big(\|z_{n,n'}(t,\cdot)\|_{C^{k+2+\sigma}_b(\mR^d)} + \|\rho_{n,n'}(t,\cdot)\|_{C^{-k-1-\sigma}_b(\mR^d)}\big)\\ &\leq C\big[(C(F)+C(G))\sup\limits_{t\in[t_0,T]} d_0(m_{n}(t),m_{n'}(t))\sup\limits_{t\in[t_0,T]}\|\rho_{n'}(t,\cdot)\|_{C^{-k-1-\sigma}_b(\mR^d)} \\ &+ \sup\limits_{t\in[t_0,T]}\|\Gamma(t,\cdot)\|_{C^1_b(\mR^d)}\sup\limits_{t\in[t_0,T]}\|z_{n'}(t,\cdot)\|_{C^{2}_b(\mR^d)}\sup\limits_{t\in[t_0,T]}d_0(m_{n'}(t),m_{n}(t))\\
    &+\|c_{n} - c_{n'}\|_{L^1(C^{-k-\sigma}_b(\mR^d))} + \|\rho_0^{n} - \rho_0^{n'}\|_{C^{-k-1-\sigma}_b(\mR^d)}\big], 
\end{align*}
where $C$ does not depend on $n,n'$, and $\eps$. Furthermore, by \eqref{eq:linestimates} the bounds on $\rho_{n'}$ and $z_{n'}$ are uniform in $n'$, therefore we get that $(z_n,\rho_n)$ is a Cauchy sequence in $C([t_0,T],C^{k+2+\sigma}_b(\mR^d))\times C([t_0,T],C^{-k-1-\sigma}_b(\mR^d))$, and so, it converges to some $(z,\rho)$. It is easy to see that it solves the system \eqref{eq:linsyst} and satisfies estimates \eqref{eq:linestimates}. \medskip

\noindent 6)\quad The uniqueness follows by noting that if pairs $(z_1,\rho_1)$ and $(z_2,\rho_2)$ solve \eqref{eq:linsyst}, then their difference solves \eqref{eq:linsyst} with $b,c,z_T$, and $\rho_0$ equal to $0$. Thus, by estimate \eqref{eq:linestimates} and the definition of $M$, we get that they are equal.
\end{proof}
We can get improved space regularity for $\rho$ in the linear system if we allow the norms to blow up at the initial time.
\begin{lemma}\label{lem:l1cneg}
Assume that \ref{eq:K} holds, $V_1\in (C([0,T], C^n_b(\mR^d)))^d$, $V_2\in (C([0,T],(\mathcal{M}(\mR^d),d_0))\cap B([0,T],\mathcal{M}(\mR^d)))^d$, and $\rho_0\in C^{-n}_b(\mR^d)$ for some $n\in \{2,3,\ldots\}$. Then the problem
\begin{align}
    \begin{cases}\label{eq:FPcneg}
    \partial_t \rho - \mL \rho -\Div(\rho V_1) - \Div(V_2) = 0,\quad {\rm on}\ (0,T)\times \mR^d,\\
    \rho(0) = \rho_0.
    \end{cases}
\end{align}
has a distributional solution $\rho$ (in the sense of Lemma~\ref{lem:milddistrcneg}) such that $\rho \in C((0,T],C^{\gamma-n}_b(\mR^d))\cap B([0,T],C^{-n}_b(\mR^d))$ for every $\gamma\in (0,\alow)$ and
\begin{align}\label{eq:cnegexplode}
    \sup\limits_{t\in (0,T]}\|t^{\frac{\gamma}{\alow}}\rho(t)\|_{C^{\gamma-n}_b(\mR^d)} \leq C(V_1)(\sup\limits_{t\in [0,T]}\|V_2(t)\|_{TV} + \|\rho_0\|_{C^{-n}_b(\mR^d)}).
\end{align}
If we assume that $\rho_0\in \mathfrak{C}^{-n}_b(\mR^d)$, i.e., it is measure representable (see Definition~\ref{def:MR}), then $\rho(t)\in \mathfrak{C}^{-n}_b(\mR^d)$ for $t\in [0,T]$.
\end{lemma}
\begin{proof}   We will use Duhamel's formula and Banach's fixed point theorem. We only discuss the case $\rho_0\in \mathfrak{C}^{-n}_b(\mR^d)$, which is slightly more subtle, because aside from the estimates it also requires to verify if being measure representable is preserved under the fixed point map. Fix $\gamma\in (0,\alow)$, let $M=2C_1(\|\rho_0\|_{C^{-n}_b(\mR^d)}+\sup\limits_{t\in [0,T]}\|V_2(t)\|_{TV})$, where $C_1$ is the constant in \eqref{eq:aconstant} below, and define\footnote{For $\rho_0\in C^{-n}_b(\mR^d)$ simply replace $C_b((0,T],\mathfrak{C}^{-n}_b(\mR^d))$ by  $C_b((0,T],C^{-n}_b(\mR^d))$ in the definition of $X$ and follow the proof.}
  \begin{align*}
      X = \big\{\rho \in C_b((0,T],\mathfrak{C}^{-n}_b(\mR^d))\cap C((0,T],C^{-n+\gamma}_b(\mR^d)): \sup\limits_{t\in(0,T]}\|t^{\frac{\gamma}{\alow}}\rho(t)\|_{C^{-n+\gamma}_b(\mR^d)} \leq M\big\},
  \end{align*}
  with metric induced by the norm $\|\rho\|_X = \sup\limits_{t\in [
  (0,T]} \|\rho(t)\|_{C^{-n}_b(\mR^d)} + \sup\limits_{t\in (0,T]} \|t^{\frac{\gamma}{\alow}}\rho(t)\|_{C^{-n+\gamma}_b(\mR^d)}$.
  For $\widetilde{\rho}\in X$ we say that $S(\widetilde{\rho}) = \rho$ if $\rho(0) = \rho_0$ and for every $\xi\in C^{n-\gamma}_b(\mR^d)$ and $t\in (0,T]$ we have
  \begin{align}\label{eq:rhomild}
      \langle\rho(t),\xi\rangle &= \langle \rho_0,K_t^\ast \ast \xi\rangle + \int_0^t \bigg\langle \widetilde{\rho}(s),V_1(s,\cdot) (D_y K_{t-s}^\ast \ast \xi)(\cdot)\bigg\rangle\, ds\\
      &+ \int_0^t \int_{\mR^d} \int_{\mR^d}D_yK(t-s,x-y)\, V_2(s,dy) \xi(x)\, dx\, ds.\nonumber
  \end{align}
  We will show that $S$ maps $X$ to $X$. Using the triangle inequality we get
  \begin{align*}
      |\langle\rho(t),\xi\rangle| &\leq |\langle \rho_0,K_t^\ast\ast\xi\rangle| + \int_0^t \bigg|\bigg\langle \widetilde{\rho}(s),V_1(s,\cdot) D_y K_{t-s}^\ast \ast \xi(\cdot)\bigg\rangle\bigg|\, ds\\
      &+ \int_0^t \int_{\mR^d} \int_{\mR^d}|D_yK(t-s,x-y)||\xi(x)|\, |V_2(s,dy)| \, dx\, ds =: I_1(t) + I_2(t) + I_3(t).
  \end{align*}
  We estimate $I_1$ as follows:
  \begin{align*}
      I_1(t) \leq \|\rho_0\|_{C^{-n}_b(\mR^d)} \|K_t^\ast\ast \xi\|_{C^n_b(\mR^d)}.
  \end{align*}
  Let $\delta>0$ be small and assume that $\gamma\in (1,\alow)$ (for $\gamma=1$ and $\gamma\in (0,1)$ the argument for $I_1$ needs to be modified in a rather apparent way). By the interpolation inequality for H\"older norms (see, e.g., \cite[Corollary~1.7]{MR3753604} and \cite[Corollary~6.2]{MR920166}) we get
  \begin{align}\label{eq:interpolation}
      \|K_t^\ast\ast \xi\|_{C^n_b(\mR^d)} \leq \|K_t^\ast\ast \xi\|_{C^{n+\delta}_b(\mR^d)} \leq \|K_t^\ast\ast \xi\|_{C^{n+1-\gamma}_b(\mR^d)}^{2-\gamma-\delta}\|K_t^\ast\ast \xi\|_{C^{n+2-\gamma}_b(\mR^d)}^{\gamma + \delta -1}.
  \end{align}
  We estimate the terms using \ref{eq:K}:
  \begin{align*}
      \|K_t^\ast\ast \xi\|_{C^{n+1-\gamma}_b(\mR^d)} &\leq c\|\xi\|_{C^{n-\gamma}_b(\mR^d)}(\|K_t^\ast\|_{L^1(\mR^d)} + \|D_y K_t^\ast\|_{L^1(\mR^d)}) \leq Ct^{-\frac{1}{\alow}}\|\xi\|_{C^{n-\gamma}_b(\mR^d)},\\
      \|K_t^\ast\ast \xi\|_{C^{n+2-\gamma}_b(\mR^d)} &\leq c\|\xi\|_{C^{n-\gamma}_b(\mR^d)}(\|K_t^\ast\|_{L^1(\mR^d)} + \|D_y K_t^\ast\|_{L^1(\mR^d)} + \|D^2_{y} K_t^\ast\|_{L^1(\mR^d)}) \leq Ct^{-\frac{2}{\alow}}\|\xi\|_{C^{n-\gamma}_b(\mR^d)}.
  \end{align*}
  Using these bounds in \eqref{eq:interpolation} we find that
  \begin{align*}
      I_1(t)\leq C\|\rho_0\|_{C^{-n}_b(\mR^d)}\|\xi\|_{C^{n-\gamma}_b(\mR^d)}t^{-(2-\gamma -\delta +2\gamma + 2\delta - 2)/\alow} = C\|\rho_0\|_{C^{-n}_b(\mR^d)}\|\xi\|_{C^{n-\gamma}_b(\mR^d)}t^{-\frac{\gamma+\delta}{\alow}}. 
  \end{align*}
  Since $\delta$ is arbitrarily small we get that there exists $C_1\geq 1$ independent of $\rho_0,V_1,V_2,T$, such that
  \begin{align}\label{eq:aconstant}
      I_1(t) \leq C_1\|\rho_0\|_{C^{-n}_b(\mR^d)}\|\xi\|_{C^{n-\gamma}_b(\mR^d)}t^{-\frac{\gamma}{\alow}}.
  \end{align}
  We now estimate the second term using the definition of $X$:
   \begin{align*}
       I_2(t) &= \int_0^t \bigg|\bigg\langle \widetilde{\rho}(s),V_1(s,\cdot) D_y K_{t-s}^\ast \ast \xi(\cdot)\bigg\rangle\bigg|\, ds\\ &\leq c\int_0^t \|\widetilde{\rho}(s)\|_{C^{-n+\gamma}_b(\mR^d)} \|V_1\|_{C^{n-\gamma}_b(\mR^d)}\|D_yK^\ast_{t-s}\ast\xi\|_{C^{n-\gamma}_b(\mR^d)}\, ds \nonumber \\
       &\hspace{-11pt}\mathop{\leq}^{\ref{eq:K}, \widetilde{\rho} \in X} 2c(V_1)C_1(\|\rho_0\|_{C^{-n}_b(\mR^d)} + \sup\limits_{t\in [0,T]}\|V_2(t)\|_{TV})\|\xi\|_{C^{n-\gamma}_b(\mR^d)}\int_0^t s^{-\frac{\gamma}{\alow}}(t-s)^{-\frac{1}{\alow}}\, ds\nonumber \\ &= 2c(V_1)C_1(\|\rho_0\|_{C^{-n}_b(\mR^d)} + \sup\limits_{t\in [0,T]}\|V_2(t)\|_{TV})\|\xi\|_{C^{n-\gamma}_b(\mR^d)}t^{1-\frac {\gamma}{\alow} - \frac{1}{\alow}}\int_0^1v^{-\frac{\gamma}{\alow}}(1-v)^{-\frac 1{\alow}}\, dv\\ &\leq C_2(V_1)(\|\rho_0\|_{C^{-n}_b(\mR^d)} + \sup\limits_{t\in [0,T]}\|V_2(t)\|_{TV})T^{1-\frac 1{\alow}}\|\xi\|_{C^{n-\gamma}_b(\mR^d)}t^{-\frac {\gamma}{\alow}}.\nonumber
 \end{align*}
 The last term is straightforward, we use \ref{eq:K} to get
 \begin{align*}
     I_3(t) \leq \sup\limits_{t\in[0,T]}\|V_2(t)\|_{TV} \|\xi\|_{\infty}
     \mathcal K\int_0^t (t-s)^{-\frac 1{\alow}}\, ds \leq C_3 T^{1-\frac 1{\alow}}\sup\limits_{t\in [0,T]}\|V_2(t)\|_{TV}\|\xi\|_{C^{n-\gamma}_b(\mR^d)}.
 \end{align*}
 By summing up the estimates we get
 \begin{align*}
      |\langle\rho(t),\xi\rangle| &\leq (C_1 +C_2(V_1)T^{1-\frac 1{\alow}} + C_3T^{1-\frac 1{\alow} +\frac{\gamma}{\alow}})(\|\rho_0\|_{C^{-n}_b(\mR^d)} + \sup\limits_{t\in [0,T]}\|V_2(t)\|_{TV})\|\xi\|_{C^{n-\gamma}_b(\mR^d)}t^{-\frac{\gamma}{\alow}}.
 \end{align*}
Therefore, if we take $T$ sufficiently small, we get $\sup\limits_{t\in(0,T]}\|t^{\frac{\gamma}{\alow}}\rho(t)\|_{C^{-n+\gamma}_b(\mR^d)} \leq M$. A similar (but easier) argument shows that $\sup\limits_{t\in[0,T]}\|\rho(t)\|_{C^{-n}_b(\mR^d)} <\infty$. In order to get that $\rho$ is continuous in time in respective norms we first make the following observation: if we repeat the above estimates with $\xi\in C^{n-\alow+\delta}_b(\mR^d)$, for $\delta>0$ small enough (in $I_2$ we have to use interpolation), then we get that for every $0<\eps<T$ we have $\rho(\eps)\in C^{-n+\alow-\delta}_b(\mR^d)$. By using Lemma~\ref{lem:rhotimeDuh} on $(\eps,t)$ (see also \eqref{eq:weakDuhamelst}) we therefore get that $\rho\in C((\eps,T],C^{-n+\alow-2\delta}_b(\mR^d))$. Since $\delta$ and $\eps$ can be arbitrarily small, we get that $\rho\in C((0,T],C^{-n+\gamma}_b(\mR^d))$.

The last thing to verify is that $\rho(t)$ defined in \eqref{eq:rhomild} is measure representable, i.e., $\rho(t) \in \mathfrak{C}^{-n}_b(\mR^d)$ for $t\in (0,T]$. In the terms with $\rho_0$ and $V_2$ a simple use of Fubini's theorem suffices (we use the fact that $\rho_0$ is measure representable). For the term with $V_1$ we note that by the definition of $X$ and the fact that $V_1\in (C_b((0,T),C^n_b(\mR^d)))^d$ we have $\Xi := K_{t-\cdot}\ast \Div(V_1(\cdot)\widetilde{\rho}(\cdot))\in C_b((0,T),\mathfrak{C}^{-n}_b(\mR^d))$ (note that $\langle \Xi(s), \xi\rangle = \langle \widetilde{\rho}(s),V_1(s)D_yK^{\ast}_{t-s}\ast \xi\rangle$). Therefore, $\Xi\colon (0,T)\to \mathfrak{C}^{-n}_b(\mR^d)$ is measurable and separably valued (any continuous function on an interval is separably valued), hence strongly measurable, see \cite[Corollary~1.1.10]{MR3617205}. By \cite[Proposition~1.2.3]{MR3617205} and \ar{the completeness of $\mathfrak{C}^{-n}_b(\mR^d)$ we find that $\int_0^t\Xi(s)\, ds\in \mathfrak{C}^{-n}_b(\mR^d)$, and by \cite[(1.2)]{MR3617205}, for every $\xi\in C^n_b(\mR^d)$,}
\begin{align*}
\int_0^t\langle \Xi(s),\xi\rangle\, ds = \langle \int_0^t\Xi(s)\, ds,\xi\rangle
\end{align*}
 This proves that $\rho(t)\in \mathfrak{C}^{-n}_b(\mR^d)$, therefore we can conclude that $S$ maps $X$ to itself.

 Thanks to linearity, verification that $S$ is a contraction for small $T$ is very similar. From Banach's fixed point theorem we infer that $S$ has a fixed point in $X$. By a Gr\"onwall inequality argument  similar to the one in the proof of \eqref{eq:c3bound}, we get
 \begin{align}\label{eq:cnegGronwall}
     \|\rho(T)\|_{C^{-n}_b(\mR^d)} \leq C(V_1,T)(\|\rho_0\|_{C^{-n}_b(\mR^d)} + \sup\limits_{t\in [0,T]}\|V_2(t)\|_{TV}),
 \end{align}
 with $C(V_1,T)$ locally bounded in $T\in [0,\infty)$.
 Therefore the solution can be extended in time indefinitely. We call the so-obtained $\rho$ a mild solution and by Lemma~\ref{lem:milddistrcneg} we get that it is a very weak solution as well.
 
 Note that by the above arguments we get \eqref{eq:cnegexplode} for small times, say on $[0,T_1]$, with $T_1$ independent of $\rho_0$. In order to get the estimates for larger times we note that $\|t^{-\frac{\gamma}{\alow}}\rho(t)\|_{C^{\gamma-n}_b(\mR^d)}\approx C(T)\|\rho(t)\|_{C^{\gamma-n}_b(\mR^d)}$ for $t\in [T_1/2,T]$. Therefore the estimate for any $t\geq T_1$ can be obtained by using short time estimates on the interval $[t-T_1/2,t+T_1/2]$ and \eqref{eq:cnegGronwall}.
\end{proof}
\begin{corollary}\label{cor:l1cneg}
Let $(z,\rho)$ be the solution of the linear system \eqref{eq:linsyst} with $b,c,z_T=0$ and $\rho_0\in \mathfrak{C}^{-k-1}_b(\mR^d)$, $k\geq 1$. Then for every $\gamma\in(0,\alow)$ and $t\in (t_0,T)$, $\rho(t)$ can be extended to a functional $\widetilde{\rho}(t)\in C^{\gamma-k-1}_b(\mR^d)$ in the way that $\widetilde{\rho}\in C((t_0,T],C^{\gamma-k-1}_b(\mR^d))$ and \begin{align*}\|t^{\frac{\gamma}{\alow}}\widetilde{\rho}(t)\|_{C^{\gamma-k-1}_b(\mR^d)}\leq C(V)\|\rho_0\|_{C^{-k-1}_b(\mR^d)}.\end{align*}
In particular,
\begin{align*}\|\widetilde{\rho}\|_{L^1(C^{\gamma-k-1}_b(\mR^d))}\leq C(V)\|\rho_0\|_{C^{-k-1}_b(\mR^d)}.\end{align*}
Furthermore, for every $t\in [t_0,T]$ we have $\widetilde{\rho}(t)|_{C^{-k-1}_b(\mR^d)} \in \mathfrak{C}^{-k-1}_b(\mR^d)$.
\end{corollary}
\begin{proof}
  We use Lemma~\ref{lem:l1cneg} with $V_1 = V$ and $V_2 = m\Gamma Dz$ and $n = k+1$ (note that the assumptions are satisfied). Let $\widetilde{\rho}$ be the solution obtained in Lemma~\ref{lem:l1cneg} and let $w$ solve \eqref{eq:xi}. Then $w$ is an admissible test function for $\rho-\widetilde{\rho}$ in the sense of Lemma~\ref{lem:milddistrcneg} and we find that $\langle (\rho - \widetilde{\rho})(\tau),\xi\rangle = 0$. Since $\xi\in C^{k+1+\sigma}_b(\mR^d)$ is arbitrary, we find that $\rho(\tau)$ and $\widetilde{\rho}(\tau)$ are identical as elements of $C^{-k-1-\sigma}_b(\mR^d)$, which means that $\widetilde{\rho}$ is an extension of $\rho$. Note that by Theorem~\ref{th:linsyst} we have $\|m\Gamma Dz\|_{TV} \leq C\|Dz\|_{\infty}\leq C\|\rho_0\|_{C^{-k-1}_b(\mR^d)}$, therefore the estimates indeed depend only on $\|\rho_0\|_{C^{-k-1}_b(\mR^d)}$.
\end{proof}
\subsection{Existence and properties of $\dm{U}$}\label{sec:dm} 
As announced earlier, we will obtain $\dm{U}$ as a solution to a certain linear system of the form \eqref{eq:linsyst}. We first define a candidate for $\dm{U}$ in Proposition~\ref{prop:Jdef} and prove its regularity. Then, in Proposition~\ref{prop:dmU} we show that the candidate function from Proposition~\ref{prop:Jdef} indeed satisfies the definition of $\dm{U}$.

We start with an elementary technical lemma, which we use to reduce the assumptions on $H$.
\begin{lemma}\label{lem:Holdereps}
Let $\sigma\in (0,1)$ and $\eps\in (0,1-\sigma)$ and assume that $f\colon \mR^{m+n}\to\mR$ and $y_1,y_2\colon \mR^m \to \mR^n$. Then
\begin{align*}
    \|f(\cdot,y_1(\cdot)) - f(\cdot,y_2(\cdot))\|_{C^\sigma_b(\mR^d)}\leq 2\|f\|_{C^{\sigma+\eps}_b(\mR^d)}(2+\|y_1 - y_2\|_{C^1_b(\mR^d)}^\sigma)\|y_1 - y_2\|_{\infty}^\eps.
\end{align*}
\end{lemma}
\begin{proof}
Let $\phi(x) = f(x,y_1(x))$ and $h(x) = (0,y_2(x) - y_1(x))$. Then we have
\begin{align*}
   \|f(\cdot,y_1(\cdot)) - f(\cdot,y_2(\cdot))\|_{C^\sigma_b(\mR^d)}=  \|\phi(\cdot + h(\cdot)) - \phi(\cdot)\|_{C^{\sigma}_b(\mR^d)}.
\end{align*}
By an elementary calculation we get $\|\varphi\|_{C^\sigma_b(\mR^{n})}\leq 2\|\varphi\|_{C^{\sigma+\varepsilon}_b(\mR^{n})}^{\frac{\sigma}{\sigma+\varepsilon}} \|\varphi\|_{\infty}^{\frac{\varepsilon}{\sigma+\varepsilon}}$. Therefore,
\begin{align}\label{eq:eqn1}
    \|\phi(\cdot + h(\cdot)) - \phi(\cdot)\|_{C^{\sigma}_b(\mR^d)} \leq 2\|\phi(\cdot + h(\cdot)) - \phi(\cdot)\|_{C^{\sigma+\eps}_b(\mR^d)}^{\frac \sigma{\sigma + \eps}}\|\phi(\cdot + h(\cdot)) - \phi(\cdot)\|_{\infty}^{\frac {\eps}{\sigma + \eps}}.
\end{align}
We will estimate the terms separately. We have
\begin{align*}
    &\|\phi(\cdot + h(\cdot)) - \phi(\cdot)\|_{C^{\sigma+\eps}_b(\mR^d)} \leq \|\phi(\cdot + h(\cdot))\|_{C^{\sigma+\eps}_b(\mR^d)}  +\|\phi\|_{C^{\sigma+\eps}_b(\mR^d)} \leq \|\phi\|_{C^{\sigma+\eps}_b(\mR^d)}(1 + \|h\|_{C^1_b(\mR^d)}^{\sigma + \eps}) + \|\phi\|_{C^{\sigma+\eps}_b(\mR^d)},\\
    &\|\phi(\cdot + h(\cdot)) - \phi(\cdot)\|_{\infty} \leq \|h\|_{\infty}^{\sigma + \eps}\|\phi\|_{C^{\sigma+\eps}_b(\mR^d)}.
\end{align*}
Going back to \eqref{eq:eqn1} we obtain
\begin{align*}
    \|\phi(\cdot + h(\cdot)) - \phi(\cdot)\|_{C^{\sigma}_b(\mR^d)} &\leq 2\big(\|\phi\|_{C^{\sigma+\eps}_b(\mR^d)}(2 + \|h\|_{C^1_b(\mR^d)}^{\sigma + \eps})\big)^{\frac \sigma{\sigma + \eps}}(\|h\|_{\infty}^{\sigma + \eps}\|\phi\|_{C^{\sigma+\eps}_b(\mR^d)})^{\frac \eps{\sigma + \eps}} \\ &\leq 2\|\phi\|_{C^{\sigma+\eps}_b(\mR^d)}(2 + \|h\|_{C^1_b(\mR^d)}^{\sigma})\|h\|_{\infty}^{\eps},
\end{align*}
which ends the proof.
\end{proof}
\begin{proposition}\label{prop:Jdef}
Let $\sigma\in(0,\alow-1)$, assume that \ref{eq:H}, \ref{eq:H2}, \ref{eq:H3}, \ref{eq:K}, \ref{eq:F}, \hyperref[eq:F2]{(F2(1,$\sigma$))}, \ref{eq:G}, \hyperref[eq:G2]{(G2(1,$\sigma$))}, \ref{eq:M1}, \ref{eq:M12}, and \ref{eq:M2} hold, $\rho_0\in \mathfrak{C}^{-2}_b(\mR^d)$, and let $(u,m)$ be the unique solution of the MFG system \eqref{eq:MFG} with starting time $t_0$ and initial measure $m_0$, obtained in Theorem~\ref{th:MFGwp}. Then, the system
\begin{equation}\label{eq:dmusyst}
\begin{cases}-\partial_t z - \mL z + D_pH(x,Du)\cdot Dz = \langle\dm{F}(x,m(t)),\rho(t)\rangle\quad &{\rm in}\ (t_0,T)\times \mR^d,\\
\partial_t \rho - \mL^* \rho - \Div(\rho D_pH(x,Du)) - \Div(mD^2_{pp}H(x,Du)Dz) = 0\quad &{\rm in}\ (t_0,T)\times \mR^d,\\
z(T,x) = \langle \dm{G}(x,m(T)),\rho(T)\rangle,\quad \rho(t_0) = \rho_0.
\end{cases}
\end{equation}
has a unique classical solution $(z,\rho)$, such that $z\in C_b([t_0,T),C^{3+\sigma}_b(\mR^d))$ and $\rho\in C_b((t_0,T],C^{-2}_b(\mR^d))$. Furthermore, if we let $(z_y,\rho_y)$ be the solutions of \eqref{eq:dmusyst} with $\rho_0 = \delta_y$ and define $J\colon [0,T]\times \mR^d\times \mP\times \mR^d\to \mR$ as
\begin{align}\label{eq:Jdef}
    J(t_0,x,m_0,y) = z_y(t_0,x),
\end{align}
then $J(t_0,\cdot,m_0,\cdot)\in C^{3+\sigma}_b(\mR^d, C^{2+\sigma}_b(\mR^d))$, the derivatives up to the order $D^{3,2}_{x,y}$ are uniformly continuous with respect to $t_0$ and $m_0$, and there exists $C$ independent of $(t_0,m_0)$ such that
\begin{align}\label{eq:Jbounds}
    \|J(t_0,\cdot,m_0,\cdot)\|_{C^{3+\sigma}_b(\mR^d, C^{2+\sigma}_b(\mR^d))} \leq C.
\end{align}
Moreover, for any finite signed Borel measure $\mu$, we have $\langle J(t_0,x,m_0,\cdot),\mu\rangle = z(t_0,x)$, where $(z,\rho)$ is the solution of \eqref{eq:dmusyst} with $\rho_0  = \mu$.
\end{proposition}
\begin{proof}
  1) ({\em Existence of the derivatives}) Existence and uniqueness of $J$ and its regularity in $x$ follows from Theorem~\ref{th:linsyst} and the regularity for the MFG system given in Theorem~\ref{th:MFGwp}. Let $e_i$ be a  unit coordinate vector and let $h\neq 0$. Since the system \eqref{eq:dmusyst} is linear, we get 
  \begin{align*}
      \frac{J(t_0,x,m_0,y+he_i) - J(t_0,x,m_0,y)}{h} = w_h(t_0,x),
    \end{align*} where $w_h$ is the solution to the first equation of \eqref{eq:dmusyst} corresponding to the initial condition $\frac 1h(\delta_{y+he_i} - \delta_y)$. Furthermore, $\frac 1h(\delta_{y+he_i} - \delta_y)$ converges to $\partial_i\delta_y$ in $C^{-2-\sigma}_b(\mR^d)$. Therefore if we let $w$ be the solution corresponding to $\rho_0 = \partial_i\delta_y$, by using estimate \eqref{eq:linestimates} (note that the constant $M$ depends only on $\rho_0$ in our case) we obtain that
    \begin{align*}
        &\bigg\|\frac{J(t_0,\cdot,m_0,y+he_i) - J(t_0,\cdot,m_0,y)}{h} - w(t_0,\cdot)\bigg\|_{C^{3+\sigma}_b(\mR^d)}\\ &= \|w_h(t_0,\cdot) - w(t_0,\cdot)\|_{C^{3+\sigma}_b(\mR^d)} \leq C\bigg\|\frac 1h(\delta_{y+he_i} - \delta_y) - \partial_i\delta_y\bigg\|_{C^{-2-\sigma}_b(\mR^d)} \mathop{\longrightarrow}^{h\to 0}\  0.
    \end{align*}
    In particular, $\partial_{y_i} J(t_0,x,m_0,y) = w(t_0,x)$. We can do a similar procedure to get the second derivative in $y$ for $J$. By considering the system \eqref{eq:dmusyst} with $\rho_0 = \frac{\partial^2_{ij}\delta_{y} - \partial^2_{ij}\delta_{y'}}{|y-y'|^{\sigma}}\in C^{-2}_b(\mR^d)$ and using estimates \eqref{eq:linestimates} we find that
    \begin{align*}
        \bigg\|\frac{\partial^2_{y_iy_j} J(t_0,\cdot,m_0,y) - \partial^2_{y_iy_j} J(t_0,\cdot,m_0,y')}{|y - y'|^{\sigma}}\bigg\|_{C^{3+\sigma}_b(\mR^d)} \leq C\bigg\|\frac{\partial^2_{ij}\delta_{y} - \partial^2_{ij}\delta_{y'}}{|y-y'|^{\sigma}}\bigg\|_{C^{-2-\sigma}_b(\mR^d)}. 
    \end{align*}
    The last expression is bounded with respect to $y$ and $y'$, because for every $\phi\in C^{2+\sigma}_b(\mR^d)$ we have
    \begin{align*}
       |\langle\partial^2_{ij}\delta_{y} - \partial^2_{ij}\delta_{y'},\phi\rangle| \leq |D^2\phi(y) - D^2\phi(y')|\leq \|\phi\|_{C^{2+\sigma}_b(\mR^d)}|y-y'|^{\sigma}.
    \end{align*}
    This proves \eqref{eq:Jbounds}.\medskip
    
    2) ({\em Regularity in $m_0$ and $t_0$})\ To show the continuity in $m_0$, we first consider $m_0^1,m_0^2\in \mP$, and we let $(u^1,m^1)$, $(u^2,m^2)$ be the corresponding solutions to the MFG system \eqref{eq:MFG} and $(z^1,\rho^1)$, $(z^2,\rho^2)$ the corresponding solutions to the linear system \eqref{eq:dmusyst} with $\rho_0=\partial^2_{ij}\delta_y$ (or $\delta_y$, $\partial_i \delta_y$,  for $J$ and $\partial_{i} J$ respectively). Then $\frac{\partial^2_{ij}}{\partial y^2}J(t_0,x,m_0^1,y) - \frac{\partial^2_{ij}}{\partial y^2}J(t_0,x,m_0^2,y) = z^1(t_0,y) - z^2(t_0,y)$ and the difference $(z,\rho) := (z^1 - z^2, \rho^1 - \rho^2)$ satisfies the linear system \eqref{eq:linsyst} with the following coefficients:
    \begin{equation}
    \begin{split}\label{eq:coefficients}
        \rho_0 &= 0,\quad V(t,x) = D_p H(x,Du^1),\quad \Gamma(t,x) = D^2_{pp}H(x,Du^1), \\
        b(t,x) &= (D_pH(x,Du^1) - D_pH(x,Du^2))Dz^2 + \bigg\langle\dm{F}(x,m^1(t)) - \dm{F}(x,m^2(t)),\rho^2(t)\bigg\rangle = b_1 + b_2,\\
        c &= (D_pH(x,Du^1) - D_pH(x,Du^2))\rho^2 + (m^1 D^2_{pp}H(x,Du^1) - m^2D^2_{pp}H(x,Du^2))Dz^2 =: c_1 + c_2,\\
        z_T(x) &= \bigg\langle\dm{G}(x,m^1(T)) - \dm{G}(x,m^2(T)),\rho^2(T)\bigg\rangle.
    \end{split}
    \end{equation}
    In order to estimate the terms we repeatedly use the results of Theorems~\ref{th:MFGwp}, \ref{th:Lip}, and \ref{th:linsyst}, which are independent of $m_0$. By using in addition \hyperref[eq:F2]{(F2(1,$\sigma$))} and \hyperref[eq:G2]{(G2(1,$\sigma$))} we get
    \begin{align*}
        \|z_T\|_{C^{3+\sigma}_b(\mR^d)} + \sup\limits_{t\in[t_0,T]}\|b_2(t,\cdot)\|_{C^{2+\sigma}_b(\mR^d)} \leq C\sup\limits_{t\in[t_0,T]}\|\rho_2(t)\|_{C^{-2-\sigma}_b(\mR^d)} d_0(m_0^1,m_0^2) \leq C_1d_0(m_0^1,m_0^2).
    \end{align*}
By Lemma~\ref{lem:Holdereps},\footnote{Note that a rough estimate using the fundamental theorem of calculus on $D_pH$ would require $4+\sigma$ derivatives from $H$. It would however yield a stronger result with $d_0(m_0^1,m_0^2)$ on the right-hand side.}
    \begin{align*}
        \sup\limits_{t\in[t_0,T]} \|b_1(t,\cdot)\|_{C^{2+\sigma}_b(\mR^d)}&\leq C\sup\limits_{t\in[t_0,T]} \|D_pH(x,Du^1) - D_pH(x,Du^2)\|_{C^{2+\sigma}_b(\mR^d)}\|z^2\|_{C^{3+\sigma}_b(\mR^d)}\\
        &\leq \widetilde{C} \sup\limits_{t\in[t_0,T]}\|Du^1 - Du^2\|_{\infty}^\eps \leq C_2 d_0^\eps(m_0^1,m_0^2).
    \end{align*}
    For $c_2$ we write
    \begin{align*}
        \sup\limits_{t\in [t_0,T]}\|c_2(t)\|_{C^{-1}_b(\mR^d)} &\leq  \sup\limits_{t\in [t_0,T]}\|z^2(t)\|_{C^2_b(\mR^d)}\big(\|m^1 (D^2_{pp}H(x,Du^1) - D^2_{pp}H(x,Du^2))\|_{C^{-0}_b(\mR^d)}\\
        &\qquad\qquad\qquad+\|(m^1-m^2) D^2_{pp}H(x,Du^2)\|_{C^{-1}_b(\mR^d)}\big)\\
        &\leq C\sup\limits_{t\in [t_0,T]}\|z^2(t)\|_{C^2_b(\mR^d)}\big(\|D^2_{pp}H(x,Du^1) - D^2_{pp}H(x,Du^2)\|_{\infty} + d_0(m^1(t),m^2(t))\big)\\
        &\leq C_3 d_0(m_0^1,m_0^2).
    \end{align*}
   In order to estimate $c_1$, we use Corollary~\ref{cor:l1cneg}
   with $\gamma=1+\sigma$ and $k=1$:
    \begin{align*}
        \|c_1(t)\|_{C^{\sigma-1}_b(\mR^d)} &\leq \|\rho^2(t)\|_{C^{\sigma-1}_b(\mR^d)}\sup\limits_{t\in [t_0,T]}\|(D_pH(x,Du^1) - D_pH(x,Du^2))\|_{C^{1-\sigma}_b(\mR^d)}\\ &\leq Ct^{-\frac{1+\sigma}{\alow}}\sup\limits_{t\in [t_0,T]}\|u^1 - u^2\|_{C^2_b(\mR^d)}\leq \widetilde{C} t^{-\frac{1+\sigma}{\alow}} d_0(m_0^1,m_0^2).
    \end{align*}
    It follows that
    \begin{align*}
        \|c_1\|_{L^1(C^{\sigma-1}_b(\mR^d))} \leq C_4 d_0(m_0^1,m_0^2).
    \end{align*}
    Furthermore, by Corollary~\ref{cor:l1cneg} we have $\rho^2(t)\in \mathfrak{C}^{-2}_b(\mR^d)$ for $t\in[t_0,T]$, and by Theorem~\ref{th:linsyst}, $\rho^2\in C([t_0,T],C^{-3}_b(\mR^d))$. Note that constants $C_1,C_2,C_3,C_4$ above are independent of $m_0$.
    Therefore, we are in a position to use Theorem~\ref{th:linsyst}, which gives
    \begin{align*}
        \sup\limits_{[t_0,T]}\|z(t,\cdot)\|_{C^{3+\sigma}_b(\mR^d)} \leq Cd_0^\eps(m_0^1,m_0^2),
    \end{align*}
    with $C$ independent of $m_0^1,m_0^2$. This proves the uniform continuity with respect to $m_0$ of derivatives up to order $D^{3,2}_{x,y}$ of $J$.
    
    For continuity in $t_0$ we let $0\leq t_0^1<t_0^2\leq T$ and we consider $(u^1,m^1)$ and $(u^2,m^2)$ which solve the MFG system \eqref{eq:MFG} on the interval $[t_0^1,T]$, $[t_0^2,T]$ respectively, both with initial measure $m_0$. We let $(z^1,\rho^1)$, $(z^2,\rho^2)$ the corresponding solutions to the linear system \eqref{eq:dmusyst} with $\rho_0=\delta_y$ (or $\partial_i \delta_y$, $\partial^2_{ij}\delta_y$ to handle the derivatives in $y$). Then
    \begin{align*}
    J(t_0^1,x,m_0,y) - J(t_0^2,x,m_0,y) = z^1(t_0^1,x) - z^2(t_0^2,x) = z^1(t_0^1,x) - z^1(t_0^2,x) + z^1(t_0^2,x) - z^2(t_0^2,x).
    \end{align*}
    The term $z^1(t_0^1,x) - z^1(t_0^2,x)$ converges to 0 (along with $x$ derivatives up to order $3$) uniformly in $x,y,m_0$ because of the regularity of the solutions to \eqref{eq:dmusyst}. For the remaining term we first note that $(z^1,\rho^1)$ satisfies the system \eqref{eq:dmusyst} on $[t_0^2,T]$ with $\rho_0 = \rho^1(t_0^2)$. Therefore $(z,\rho) := (z^1 - z^2,\rho^1 - \rho^2)$ satisfies \eqref{eq:linsyst} on $[t_0^2,T]$ with coefficients same as in \eqref{eq:coefficients} excluding $\rho_0$, which here is equal to $\rho^1(t_0^2) - \delta_y$.
    The rest of the arguments is similar as in the case of continuity in $m_0$, but here we also use the $t_0$ stability of the MFG system given in Lemma~\ref{lem:timestab}.

    3) ({\em Pairing with measures}) Let $\mu\in \mathcal{M}(\mR^d)$. We let $z(t,x) = \int_{\mR^d} z_y(t,x)\,\mu(dy)$ and $\rho(t) = \int_{\mR^d} \rho_y(t)\,\mu(dy)$, with the latter integral understood in the distributional sense, i.e., $\langle \rho(t),\phi\rangle = \int_{\mR^d} \langle \rho_y(t),\phi\rangle\, \mu(dy)$ for $\phi\in C^{2+\sigma}_b(\mR^d)$. By uniform bounds for derivatives of $z_y$ and the fact that $\mu$ is a finite measure we can use the dominated convergence theorem to see that $z$ satisfies the first equation in \eqref{eq:dmusyst} -- note that we get $\langle\dm{F}(x,m(t)),\rho(t)\rangle$ on the right-hand side immediately from the definition of $\rho$. The derivatives of $\rho_y$ in the second equation also become derivatives of $\rho$ after the integration. Furthermore, by uniform bounds on $Dz_y$ and Fubini's theorem we find that for every $\phi\in C([t_0,T],C^{2+\sigma}_b(\mR^d))$,
    \begin{align*}
        \int_{\mR^d}\int_{t_0}^t\int_{\mR^d} D_{pp}^2 H(x,Du)Dz_y(s,x)D\phi(s,x)\, m(s,dx)\, ds\, \mu(dy) = \int_{t_0}^t\int_{\mR^d} D_{pp}^2 H(x,Du)Dz(s,x)D\phi(s,x)\, m(s,dx)\, ds.
    \end{align*}
    Therefore $\rho$ satisfies the second equation. It is also easy to see that the integrated terminal conditions for $z_y$ become $z(T,x) = \langle \dm{G}(x,m(T)),\rho(T)\rangle$ and $\rho(t_0) = \int_{\mR^d} \delta_y \, \mu(dy) = \mu$ in the distributional sense.
\end{proof}
\begin{proposition}\label{prop:dmU}
Let $U$ be the function defined in Definition~\ref{def:Udef} and let $J$ be the function obtained in Proposition~\ref{prop:Jdef}. Then there exist $C,\eps>0$ such that for all $m_0^1,m_0^2\in \mP$ we have
\begin{align*}
    \|U(t_0,\cdot,m_0^2) - U(t_0,\cdot,m_0^1) - \langle J(t_0,\cdot,m_0^1,\cdot),m_0^2-m_0^1\rangle_y\|_{C^{3+\sigma}_b(\mR^d)} \leq Cd_0^{1+\eps}(m_0^1,m_0^2).
\end{align*}
In particular, $U$ is $C^1$ in the sense of Definition~\ref{def:UC1} and
\begin{equation*}
    \dm{U}(t_0,x,m_0,y) = J(t_0,x,m_0,y).
\end{equation*}
\end{proposition}
\begin{remark}
As before, by having only $d_0^{1+\eps}$ instead of $d_0^2$ we may weaken the assumption on $H$ to have only 4 derivatives instead of $4+\sigma$, see the term $b_2$ in the proof.
\end{remark}
\begin{proof}
  Let $(u^1,m^1)$ and $(u^2,m^2)$ be the solutions of the MFG system with initial measures $m_0^1,m_0^2$ respectively. We have $U(t_0,x,m_0^1) = u^1(t_0,x)$ and $U(t_0,x,m_0^2) = u^2(t_0,x)$. Note that by the last part of Proposition~\ref{prop:Jdef} we have that $\langle J(t_0,\cdot,m_0^1,\cdot),m_0^2-m_0^1\rangle_y = w(t,y)$, where $(w,\mu)$ solves the system \eqref{eq:dmusyst} with $(u,m) = (u^1,m^1)$, under initial condition $\mu(t_0)= m_0^2 - m_0^1$. 
  
  By subtracting the respective equations and using the fundamental theorem of calculus we find that the pair $(z,\rho):= (u^2 - u^1 - w,m^2 - m^1 - \mu)$ solves the system \eqref{eq:linsyst} with coefficients
  \begin{align*}
     &V(t,x) = D_pH(x,Du),\quad \Gamma(t,x) = D_{pp}^2H(x,Du),\quad \rho_0 = 0,\\
     &c = (m^2 - m^1)D_{pp}^2H(\cdot,Du)(Du^2 - Du^1)\\
     &\quad\quad\quad+ m^2\int_0^1 \big(D_{pp}^2H(\cdot,\lambda Du^2 + (1-\lambda) Du^1) - D_{pp}^2H(\cdot,Du^1)\big)(Du^2 - Du^1)\, d\lambda =: c_1 + c_2,\\
     &b(t,x) = \int_0^1\int_{\mR^d} \bigg(\dm{F}(x,\lambda m^2(t) + (1-\lambda)m^1(t),y) - \dm{F}(x,m^1(t),y)\bigg)(m^2(t,dy) - m^1(t,dy))\, d \lambda\\
     &\quad\quad\quad - \int_0^1 (D_pH(x,\lambda Du^2 + (1-\lambda)Du^1) - D_pH(x,Du^1))(Du^2 - Du^1)\, d\lambda   =: b_1(t,x) + b_2(t,x),\\
     &z_T(x) = \int_0^1\int_{\mR^d} \bigg(\dm{G}(x,\lambda m^2(T) + (1-\lambda)m^1(T),y) - \dm{G}(x,m^1(T),y)\bigg)(m^2(T,dy) - m^1(T,dy))\, d \lambda.
  \end{align*}
  Therefore, by Theorem~\ref{th:linsyst} it suffices to show that the appropriate norms of $b,c,z_T$ are bounded by $Cd_0^{1+\eps}(m_0^1,m_0^2)$. By Theorems~\ref{th:MFGwp} and \ref{th:Lip}, \ref{eq:H}, \hyperref[eq:F2]{(F2(1,$\sigma$))}, and \hyperref[eq:G2]{(G2(1,$\sigma$))} we find that 
  \begin{align*}
      \|c_1(t)\|_{C^{-1-\sigma}_b(\mR^d)} &\leq \|c_1(t)\|_{C^{-1}_b(\mR^d)} \leq c\sup\limits_{t\in[t_0,T]} \big(d_0(m^1(t),m^2(t)) \cdot \|Du^2(t) - Du^1(t)\|_{C^1_b(\mR^d)}\big) \leq Cd_0^2(m_0^1,m_0^2),\\
       \|c_2(t)\|_{C^{-1-\sigma}_b(\mR^d)} &\leq\|c_2(t)\|_{C^{-0}_b(\mR^d)} \leq \bigg\|\int_0^1 \big(D_{pp}^2H(\cdot,\lambda Du^2 + (1-\lambda) Du^1) - D_{pp}^2H(\cdot,Du^1)\big)(Du^2 - Du^1)\, d\lambda\bigg\|_{\infty}\\
       &\leq \bigg\|\int_0^1\int_0^1 \lambda D_{ppp}^3H(\cdot,s(\lambda Du^2 + (1-\lambda) Du^1) + (1-s)Du^1)|Du^2 - Du^1|^2\, d\lambda\, ds\bigg\|_{\infty}\\
       &\leq c\sup\limits_{t\in[t_0,T]} \|Du^2(t) - Du^1(t)\|_{\infty}^2 \leq Cd_0^2(m_0^1,m_0^2),\\
       \|b_1(t)\|_{C^{2+\sigma}_b(\mR^d)} &\leq \sup\limits_{\substack{t\in[t_0,T]\\ \lambda\in [0,1]}}\bigg\|\dm{F}(x,\lambda m^2(t) + (1-\lambda)m^1(t),y) - \dm{F}(x,m^1(t),y)\bigg\|_{C^{2+\sigma}_b(\mR^d,C^{1}_b(\mR^d))}d_0(m^1(t),m^2(t))\\
       &\leq c d_0^2(m_0^1,m_0^2),\\
       \|b_2(t)\|_{C^{2+\sigma}_b(\mR^d)} &\leq \|Du^2(t,\cdot) - Du^1(t,\cdot)\|_{C^{2+\sigma}_b(\mR^d)} \sup\limits_{\lambda\in [0,1]} \|(D_pH(x,\lambda Du^2 + (1-\lambda)Du^1) - D_pH(x,Du^1))\|_{C^{2+\sigma}_b(\mR^d)}\\
       &\stackrel{(\ast)}{\leq} c \|Du^2(t,\cdot) - Du^1(t,\cdot)\|_{C^{2+\sigma}_b(\mR^d)}^{1+\eps}\\
       &\leq Cd_0^{1+\eps}(m_0^1,m_0^2),\\
       \|z_T\|_{C^{3+\sigma}_b(\mR^d)} &\leq \sup\limits_{\lambda\in [0,1]}\bigg\|\dm{G}(x,\lambda m^2(T) + (1-\lambda)m^1(T),y) - \dm{G}(x,m^1(T),y)\bigg\|_{C^{3+\sigma}_b(\mR^d,C^{1}_b(\mR^d))}d_0(m^1(T),m^2(T))\\
       &\leq Cd_0^2(m_0^1,m_0^2),
  \end{align*}
  where in $(\ast)$ we used Lemma~\ref{lem:Holdereps}. We finish the proof by applying Theorem~\ref{th:linsyst}.
\end{proof}


\section{The master equation}\label{sec:master}
We finally prove well-posedness for the master equation. This is the main result of our paper. It extends previous results to 
nonlocal (and mixed local-nonlocal) diffusions  and   
from compact domains to the whole space. 
We use the scheme of \cite{MR3967062}: having obtained all the regularity needed for $\dm{U}$, we compute the time derivative of $U$ to get existence of solutions. Then we reconstruct the solutions of the MFG system from the master equation in order to get uniqueness.

\begin{definition}
We say that $V\colon [0,T]\times \mR^d \times \mP \to \mR$ is a classical solution of the master equation \eqref{eq:master}, if $\partial_t V, D_x V$,  $D^2_x V$ exist, and are bounded and uniformly continuous in all variables, 
$\dm{V}(t,\cdot,m,\cdot)\in C^{2,2}_{b,x,y}(\mR^{2d})$ with all the derivatives bounded and uniformly continuous in all variables,  
and \eqref{eq:master} is satisfied pointwise.
\end{definition}

\begin{theorem}\label{th:ME}
Let $\sigma\in (0,\alow-1)$ and assume that \ref{eq:H}, \ref{eq:H2}, \ref{eq:H3}, \ref{eq:K}, \ref{eq:F}, \hyperref[eq:F2]{(F2(1,$\sigma$))}, \ref{eq:G}, \hyperref[eq:G2]{(G2(1,$\sigma$))}, and \ref{eq:M1}, \ref{eq:M12}, \ref{eq:M2} hold.
Then $U$ defined in \eqref{eq:Udef} is the unique classical solution of the master equation \eqref{eq:master}.
\end{theorem}

\begin{remark}
Most of the work in this paper, including new ideas and techniques, was done in the previous sections where we established sufficient regularity of the function $U$. This involved proving new or refined results for mild solutions,  for individual equations and linearized systems,  in a range of function spaces including negative H\"older spaces. Our assumptions on the  nonlocal  operator $\mL$ are quite optimal for the subcritical regime.
This setting is made possible by applying
 the arguments of \cite{2021arXiv210406985C} to get a theory
without moment assumptions 
 -- the first in the context of master equations.
 Considerable effort and original techniques have  been put into making the many approximation and compactness arguments rigorous. 
This includes introducing and working with measure representable functionals and the use of $L^1$ compactness theory for negative H\"older spaces. The classical solution setting for the master equation is known to impose excess regularity assumptions on the data. We were able to reduce these by obtaining more optimal estimates through repeated use of interpolation
as well as by adapting weighted in time regularity results of \cite{MR4420941} to our setting.

 \end{remark}

%

\begin{proof}
1)\quad We start with the existence part. Fix $t_0>0$, $m_0\in \mP$, and $x\in \mR^d$, and let $h>0$.\footnote{We only compute the right derivative and prove that it is continuous. Then, by, e.g., \cite[Corollary 2.1.2]{MR710486} $U$ is continuously differentiable.} Then
\begin{align}
    \frac{U(t_0+h,x,m_0) - U(t_0,x,m_0)}{h} = &\frac{U(t_0+h,x,m(t_0+h)) - U(t_0,x,m_0)}{h}\label{eq:dt1}\\ &- \frac{U(t_0+h,x,m(t_0+h)) - U(t_0+h,x,m_0)}{h}.\label{eq:dt2}
\end{align}
Obviously the left-hand side of \eqref{eq:dt1} converges to $\partial_t^+ U(t_0,x,m_0)$ as $h\to 0^+$, provided that the limit exists. We will show that both \eqref{eq:dt2} and the right-hand side of \eqref{eq:dt1} converge and sum up to the right-hand side of the master equation. By the uniqueness for the MFG system \eqref{eq:MFG} and the definition of $U$ we have $U(t_0+h,x,m(t_0+h)) = u(t_0+h,x)$. By this and the equation for $u$ in \eqref{eq:MFG},
\begin{align*}
    \frac{U(t_0+h,x,m(t_0+h)) - U(t_0,x,m_0)}{h} \mathop{\longrightarrow}\limits_{h\to 0^+} & -\mL_x u(t_0,x) + H(x,D_xu(t_0,x)) - F(x,m_0)\\
    &=-\mL_x U(t_0,x,m_0) + H(x,D_xU(t_0,x,m_0)) - F(x,m_0).
\end{align*}
In order to handle \eqref{eq:dt2} we use the fundamental theorem of calculus \eqref{eq:fund}:
\begin{align*}
    &\frac{U(t_0+h,x,m(t_0+h)) - U(t_0+h,x,m_0)}{h}\\ = &\frac 1h\int_0^1\int_{\mR^d} \dm{U}(t_0+h,x,\lambda m(t_0+h)+(1-\lambda)m_0,y)\, (m(t_0+h)-m_0)(dy)\, d\lambda.
\end{align*}
By Propositions~\ref{prop:Jdef} and 
\ref{prop:dmU} the function $\phi(t,y)= \dm{U}(t_0+h,x,sm(t_0+h)+(1-s)m_0,y)$ (note that $\phi$ is constant in $t$) is twice differentiable in $y$ with the derivatives bounded. Therefore it can be passed as a test function (see \eqref{eq:distFP}) in the equation for $m$ in \eqref{eq:MFG}, yielding
\begin{align}
    &\frac 1h\int_0^1\int_{\mR^d} \dm{U}(t_0+h,x,\lambda m(t_0+h)+(1-\lambda)m_0,y)\, (m(t_0+h)-m_0)(dy)\, d\lambda\nonumber\\
    = &\frac 1h\int_0^1\int_{t_0}^{t_0+h}\int_{\mR^d} \bigg( \mL_y\dm{U}(t_0+h,x,\lambda m(t_0+h) + (1-\lambda)m_0,y)\nonumber\\ &\hspace{80pt}- D_y\dm{U}(t_0+h,x,\lambda m(t_0+h) + (1-\lambda)m_0,y)\cdot D_pH(y,Du(t,y))\bigg)\,m(t,dy)\, dt\,  d\lambda.\label{eq:weakFP}
\end{align}
We will show that, as $h\to 0^+$, the last expression converges to
\begin{align}\label{eq:weakFPlim}
    \int_{\mR^d} \bigg( \mL_y\dm{U}(t_0,x,m_0,y)\ - D_y\dm{U}(t_0,x,m_0,y)\cdot D_pH(y,Du(t_0,y))\bigg)\,m(t_0,dy).
\end{align}
We discuss the limiting procedure only for the term $\mL_y \dm{U}$, the other works similarly. We first split:
\begin{align*}
    &\frac 1h\int_0^1\int_{t_0}^{t_0+h}\int_{\mR^d}  \mL_y\dm{U}(t_0+h,x,\lambda m(t_0+h) + (1-\lambda)m_0,y)\,m(t,dy)\, dt\,  d\lambda\\
    =&\frac 1h\int_0^1\int_{t_0}^{t_0+h}\int_{\mR^d}  \bigg(\mL_y\dm{U}(t_0+h,x,\lambda m(t_0+h) + (1-\lambda)m_0,y) - \mL_y\dm{U}(t_0,x,m_0,y)\bigg)\,m(t,dy)\, dt\,  d\lambda\\
    +&\frac 1h\int_{t_0}^{t_0+h}\int_{\mR^d}  \mL_y\dm{U}(t_0,x,m_0,y)\,m(t,dy)\, dt = I_1(h) + I_2(h).
\end{align*}
By Propositions~\ref{prop:Jdef} and \ref{prop:dmU} and the fact that $m\in C([0,T],\mP)$, we get that $\mL_y\dm{U}$ is uniformly continuous in $t$ and $m$, hence the integrand in $I_1$ converges to 0 uniformly in $y$. Therefore $I_1(h)\to 0$ as $h\to 0^+$.

In $I_2$ we do another splitting:
\begin{align*}
    I_2(h) = \frac 1h\int_{t_0}^{t_0+h}\int_{\mR^d}  \mL_y\dm{U}(t_0,x,m_0,y)(m(t,dy) - m_0(dy))\, dt  + \int_{\mR^d}  \mL_y\dm{U}(t_0,x,m_0,y)\, m_0(dy).
\end{align*}
By Theorem~\ref{th:MFGwp} and Remark~\ref{rem:d0weak} we get that $m(t)$ converges to $m_0$ weakly as $t\to t_0^+$. Therefore, since $\mL_y\dm{U}\in C_b(\mR^d)$ we get that there is a modulus of continuity $\omega$ such that
\begin{align*}
    \bigg|\frac 1h\int_{t_0}^{t_0+h}\int_{\mR^d}  \mL_y\dm{U}(t_0,x,m_0,y)(m(t,dy) - m_0(dy))\, dt\bigg|\leq \frac 1h\int_{t_0}^{t_0+h} \omega(h) \mathop{\longrightarrow}\limits^{h\to 0^+} 0.
\end{align*}
It follows that
\begin{align*}
    I_2(h) \mathop{\longrightarrow}\limits^{h\to 0^+} \int_{\mR^d}  \mL_y\dm{U}(t_0,x,m_0,y)\, m_0(dy).
\end{align*}
By the above arguments and a similar procedure for the second term in \eqref{eq:weakFP} we find that \eqref{eq:weakFP} converges to \eqref{eq:weakFPlim}, which shows that $U$ satisfies the master equation.
2)\quad Now we proceed to the uniqueness. Let $V$ be a classical solution to the master equation. We will show that $V=U$. Fix $t_0\in [0,T)$ and $m_0\in \mP$ and consider the following Fokker--Planck equation:
\begin{align}\label{eq:FP-ME}
    \begin{cases} 
    \partial_t \wtm - \mL^* \wtm - \Div(\wtm D_pH(x,D_xV(t,x,\wtm))) = 0,\quad {\rm in}\ [t_0,T]\times \mR^d,\\
    \wtm(t_0) = m_0.
    \end{cases}
\end{align}
We will show that this equation has a distributional solution in class $C([t_0,T],\mP)$ by using the Schauder--Tychonoff fixed point theorem.  Let $\psi$ be a tightness measuring function for $m_0$ and $\nu\textbf{1}_{B(0,1)^c}$ from Lemma~\ref{lem:tight} and define
$$X = \bigg\{m\in C([t_0,T],\mP): m(0) =m_0,\, \sup\limits_{t\in[t_0,T]} \int_{\mR^d} \psi(x)\, m(t,dx) \leq C_1,\, \sup\limits_{t_0\leq s,t\leq T}\frac{d_0(m(t),m(s))}{|t-s|^{ 1/2}} \leq C_2\bigg\},$$
where the constants $C_1,C_2$ are independent of $m$ and will be specified later. We define the map $T$ on $X$ as follows: $T(m) = \mu$ if $\mu\in C([t_0,T],\mP)$ is a distributional solution of the problem
\begin{align}\label{eq:FPSchauder}
    \begin{cases}
    \partial_t \mu - \mL^* \mu - \Div(\mu D_pH(x,D_xV(t,x,m))) = 0\quad {\rm in}\ [t_0,T]\times \mR^d,\\
    \mu(t_0) = m_0.
    \end{cases}
\end{align}
\begin{enumerate}[label = \arabic*.]
    \item \textit{$T$ is well-defined.} This follows from Theorem~\ref{th:FP} with 
    \begin{align}\label{eq:bdef}b(t,x) = D_pH(x,D_xV(t,x,m(t))).\end{align} We have $b\in (C([t_0,T],C^1_b(\mR^d)))^d$ because of \ref{eq:H}, the definition of the classical solution thanks to which $V,D_xV,$ $D^2_{xx}V$ are bounded and uniformly continuous in all variables, and the definition of $X$ which gives the continuity of $m$ with respect to $t$.
    \item $T\colon X\to X$. We already have $\mu\in C([t_0,T],\mP)$ and $\mu(t_0) = m_0$. Tightness and H\"older continuity also follow from Theorem~\ref{th:FP}. Indeed, it states that
    \begin{align}&d_0(\mu(t),\mu(s)) \leq c_0(1+\|b\|_{\infty})|t-s|^{\frac 12},\quad {\rm and}\label{eq:Holdermu}\\
    &\int_{\mR^d} \psi(x)\, \mu(t,dx)\nonumber\\ &\leq \int_{\mR^d} \psi(x)\, m_0(dx) + \|D\psi\|_{\infty}\bigg(2 + cT\big(\|A\|_2 + \|B\|_{\infty} + \|b\|_{\infty} + \int_{|z|<1} |z|^2\,\nu(dz)\big)\bigg) + T\int_{|z|\geq 1} \psi(z)\, \nu(dz),\label{eq:tightmu}\end{align}
    for some constants $c_0,c$ independent of $m$. Furthermore, $\|b\|_{\infty}\leq C$, for some $C>0$ independent of $m$, therefore by setting
    $$C_1 = \int_{\mR^d} \psi(x)\, m_0(dx) + \|D\psi\|_{\infty}\bigg(2 + cT\big(\|A\|_2 + \|B\|_{\infty}+ C + \int_{|z|<1} |z|^2\,\nu(dz)\big)\bigg) + T\int_{|z|\geq 1} \psi(z)\, \nu(dz),$$
    $$C_2 = c_0(1+C),$$ 
    we get that $\mu\in X$.
    \item \textit{$X$ is convex and compact in $C([t_0,T],\mP)$.} Convexity is straightforward. For compactness we use a metric version of the Arzel\`a--Ascoli theorem, see Kelley \cite[Chapter~7 \S 17]{MR0370454} with $C = C([t_0,T],\mP)$ and $F=X$. By the H\"older condition \eqref{eq:Holdermu} we get equicontinuity, whereas by tightness (\eqref{eq:tightmu} and Lemma~\ref{lem:tight}) we obtain pointwise relative compactness. Therefore, by the Arzel\`a--Ascoli theorem $X$ is relatively compact in $C([t_0,T],\mP)$. To see that $X$ is compact, it suffices to show that it is closed, see Remark~\ref{rem:d0weak} (b). To this end, take a sequence $(m_n)\subset X$ converging to $m$ in $C([t_0,T],\mP)$. The H\"older bound is preserved, because of the uniform bound for $m_n$. The tightness estimate \eqref{eq:tightmu} is obtained from the fact that for every $t\in[t_0,T]$, $r>0$,
    $$\int_{B_r} \psi(x)\, m(t,dx) = \lim\limits_{n\to\infty} \int_{B_r} \psi(x)\, m_n(t,dx)\leq C_1.$$
    Thus $m\in X$, which proves that $X$ is compact.
    \item \textit{$T$ is continuous.} Assume that $(m_n)\subset X$ converges to $m$ in $C([0,T],\mP)$. Since $X$ is closed, $m\in X$. Let $\mu_n$ and $\mu$ be the corresponding solutions of \eqref{eq:FPSchauder}. Continuity of $T$ means that $\mu_n\to \mu$ in $C([0,T],\mP)$. In order to show this convergence we consider $\wtmu_n = \mu - \mu_n$. We have (in the distributional sense)
    \begin{align}\label{eq:FPmun}
    \begin{cases}
        \partial_t \wtmu_n - \mL^* \wtmu_n - \Div(\wtmu_n D_pH(x,D_xV(t,x,m(t)))) = \Div(\mu_n F_n(t,x)),\\
        \wtmu_n(t_0) = 0,
    \end{cases}
    \end{align}
    where
    $$F_n(t,x) = D_pH(x,D_xV(t,x,m_n(t))) - D_pH(x,D_xV(t,x,m(t))).$$
    By \ref{eq:H} and Lemma~\ref{lem:dmfacts}, 
    \begin{align}\label{eq:fnbound}\|F_n\|_{\infty} \leq C_H K_V
    \sup\limits_{t\in[t_0,T]}d_0(m(t),m_n(t))  \qquad \text{where}\qquad K_V=\sup_{t,x,m}\Big\|D_x\dm{V}(t,x,m,\cdot)\Big\|_{C^1_b(\mR^d)},\end{align}
    and $K_U\leq \|D_x\dm{V}\|_{\infty}+\|D_yD_x\dm{V}\|_{\infty}<\infty$ by Lemma \ref{lem:Schwarz}.
    The weak formulation of \eqref{eq:FPmun} reads
    \begin{align}\label{eq:weakfpme}
        \int_{\mR^d} \phi(t,x)\, \wtmu_{n}(t,dx) = \int_{t_0}^t\int_{\mR^d} (\partial_t \phi + \mL \phi - b\cdot D\phi)\, \wtmu_n(s,dx)\, ds + \int_{t_0}^t\int_{\mR^d} D\phi(s,x)F_n(s,x)\, \mu_n(s,dx)\, ds,
    \end{align}
    where $b$ is the same as in \eqref{eq:bdef} (in particular, $b\in (C([t_0,T],C^1_b(\mR^d)))^d$) and $\phi\in C^{1,2}_b([t_0,T]\times \mR^d)$ is arbitrary. Let $\phi_0 \in C^3_b(\mR^d)\cap \Lip$. By Lemma~\ref{lem:cd} there exists $\phi$ such that the integrand in the first term on the right-hand side of \eqref{eq:weakfpme} vanishes, $\phi(t) = \phi_0$, and $\|D\phi\|_{\infty} \leq \widetilde{C}$ with $\widetilde{C}$ independent of $\phi_0$, because $\|\phi_0\|_{C^1_b(\mR^d)} \leq 1$. By this and \eqref{eq:fnbound} we get
    \begin{align*}
        \bigg|\int_{\mR^d} \phi_0(x)\, \wtmu_{n}(t,dx)\bigg| &\leq \int_{t_0}^t\int_{\mR^d} |D\phi(s,x)F_n(s,x)|\, \mu_n(s,dx)\, ds\\ &\leq \|D\phi\|_{\infty}\|F_n\|_{\infty}(t-t_0)\leq C(T-t_0)\|D\phi\|_{\infty}\sup\limits_{t\in [t_0,T]} d_0(m_n(t),m(t)).
    \end{align*}
    By Lemma~\ref{lem:testapprox} it follows that \begin{align*}
        d_0(\mu_n(t),\mu(t)) = \sup\limits_{\phi\in \Lip}\bigg|\int_{\mR^d} \phi_0(x)\, \wtmu_{n}(t,dx)\bigg| \leq C(T-t_0)\sup\limits_{t\in [t_0,T]}d_0(m_n(t),m(t)),
    \end{align*}
    therefore $\mu_n$ converges to $\mu$ in $C([t_0,T],\mP)$, which proves the continuity of $T$.
\end{enumerate}
By the Schauder--Tychonoff theorem we conclude that the Fokker--Planck equation \eqref{eq:FP-ME} has a solution $\wtm$.

For $t\in[t_0,T]$ and $x\in\mR^d$, let $v(t,x) = V(t,x,\wtm(t))$. For $t\in (t_0,T)$ and $h$ small we have
\begin{align}\label{eq:vsplit}
    \frac{v(t+h,x) - v(t,x)}{h} = \frac{V(t+h,x,\wtm(t+h)) - V(t+h,x,\wtm(t))}{h} + \frac{V(t+h,x,\wtm(t)) - V(t,x,\wtm(t))}{h}.
\end{align}
The second term converges to $\partial_t V(t,x,\wtm(t))$ as $h\to 0$. For the first term we use the fundamental theorem of calculus from Lemma~\ref{lem:fund} and the distributional formulation of the Fokker--Planck equation for $\wtm$:
\begin{align*}
    &\frac{V(t+h,x,\wtm(t+h)) - V(t+h,x,\wtm(t))}{h}\\ &= \frac 1h\int_0^1\int_{\mR^d} \dm{V}(t+h,x,\lambda \wtm(t+h) + (1-\lambda)\wtm(t),y)(\wtm(t+h) - \wtm(t))(dy)\, d\lambda\\
    &=\frac 1h\int_0^1\int_{t}^{t+h}\int_{\mR^d} \bigg( \mL_y\dm{V}(t+h,x,\lambda \wtm(t+h) + (1-\lambda)\wtm(t),y)\nonumber\\ &\hspace{50pt}- D_y\dm{V}(t+h,x,\lambda \wtm(t+h) + (1-\lambda)\wtm(t),y)\cdot D_pH(y,D_xv(s,y))\bigg)\,m(s,dy)\, dt\,  d\lambda.
\end{align*}
By the definition of the classical solution, $V$ and $v$ are regular enough so that we can repeat the argument for $U$ in the existence part to get
\begin{align*}
    &\lim\limits_{h\to 0}\frac{V(t+h,x,\wtm(t+h)) - V(t+h,x,\wtm(t))}{h}\\
    &=\int_{\mR^d}\bigg(\mL_y\dm{V}(t,x,\wtm,y) - D_y\dm{V}(t,x,\wtm,y)\cdot D_pH(x,D_xV(t,y,\wtm)))\bigg)\, \wtm(t,y)\, dy.
\end{align*}
By this, \eqref{eq:vsplit} and the master equation for $V$ we obtain
\begin{align*}
    \partial_t v(t,x) &= \partial_t V(t,x,\wtm) + \int_{\mR^d}\bigg(\mL_y\dm{V}(t,x,\wtm,y) - D_y\dm{V}(t,x,\wtm,y)\cdot D_pH(x,D_xV(t,y,\wtm)))\bigg)\, \wtm(t,y)\, dy\\
    &=-\mL V(t,x,\wtm) + H(x,D_xV(t,x,\wtm)) - F(x,\wtm)\\ &= -\mL v(t,x) + H(x,D_xv(t,x)) - F(x,\wtm).
\end{align*}
Therefore $v$ satisfies the Hamilton--Jacobi equation of \eqref{eq:MFG} and so $(v,\wtm)$ satisfies the MFG system with the starting time $t_0$ and initial measure $m_0$. By the uniqueness of solutions to \eqref{eq:MFG} we get that $V(t_0,x,m_0) = v(t_0,x) = U(t_0,x,m_0)$ for every $t_0\in [0,T)$, $x\in \mR^d$ and $m_0\in\mP$, which ends the proof of uniqueness.
\end{proof}
\section*{Acknowledgements} \ar{We thank Pierre Cardaliaguet for enlightening discussions and the anonymous referees for their helpful remarks.} E. R. Jakobsen received funding from the Research Council of Norway under Grant Agreement No. 325114 “IMod. Partial differential equations, statistics and data: An interdisciplinary approach to data-based modelling”. A. Rutkowski was supported by the grant 2019/35/N/ST1/04450 of the National Science Center (Poland). The research was carried out while A. Rutkowski was an ERCIM Alain Bensoussan fellow at NTNU. 
\appendix
\section{Proof of Lemma~\ref{lem:intrep}} \label{sec:c0cb}
\begin{proof} Let $C^{-n}_0(\mR^d) = (C^n_0(\mathbb{R}^d))^\ast$ and let $K = 1 + d + \ldots +d^n$. \ar{For $\mu\in \mathcal{M}(\mathbb{R}^d)^K$ we use the convention
\begin{align*}
    \mu = (\mu_{\displaystyle 0},\, \mu_{\displaystyle(1)},\ldots,\mu_{\displaystyle(d)},\,\mu_{\displaystyle(1,1)},\,\mu_{\displaystyle(1,2)},\ldots,\,\mu_{\displaystyle(d,d)},\,\mu_{\displaystyle(1,1,1)},\ldots,\,\mu_{\underbrace{(1,\ldots,1)}_n},\ldots,\,\mu_{\underbrace{(d,\ldots,d)}_n}),
\end{align*}
so for a multi-index $\alpha$ with $|\alpha|\leq n$, $\mu_\alpha\in \mathcal{M}(\mR^d)$ is the corresponding coordinate of $\mu$.}

\noindent 1)\quad 
We first make an auxiliary claim: for every $\rho \in C^{-n}_0(\mR^d)$ there exists $\mu \in \mathcal{M}(\mathbb{R}^d)^K$ such that $\|\rho\|_{C^{-n}_0(\mR^d)} = \|\mu\|_{\mathcal{M}(\mathbb{R}^d)^K}$ and
\begin{align}\label{repr}
\langle \rho,\phi\rangle = \sum\limits_{|\alpha|\leq n} \int_{\mathbb{R}^d} \partial^{\alpha}\phi\, d\mu_\alpha,\quad \phi \in C^n_0(\mathbb{R}^d).
\end{align}

In order to prove the claim, we first note that $C^n_0(\mathbb{R}^d)$ can be embedded isometrically into $C_0(\mathbb{R}^d)^K$; below we identify $C^n_0(\mathbb{R}^d)$ with that embedding. Thus, if $\rho\in C^{-n}_0(\mR^d)$, then by the Hahn--Banach theorem it can be extended to $\rho^*\in (C_0(\mathbb{R}^d)^K)^\ast$ with $\|\rho^*\|_{(C_0(\mathbb{R}^d)^K)^\ast} = \|\rho\|_{C^{-n}_0(\mR^d)}$. By the Riesz representation theorem for $C_0(\mathbb{R}^d)$ there exists $\mu\in \mathcal{M}(\mathbb{R}^d)^K$ such that $\|\rho^\ast\|_{(C_0(\mathbb{R}^d)^K)^*} = \|\mu\|_{TV} := \sum_{|\alpha|\leq n} \|\mu_\alpha\|_{TV}$ and 
\begin{align}
\langle \rho^\ast,\phi\rangle = \sum\limits_{|\alpha|\leq n} \int_{\mathbb{R}^d} \partial^{\alpha}\phi\, d\mu_\alpha,\quad \phi \in C_0(\mathbb{R}^d)^K.
\end{align}
Then \eqref{repr} holds and $\|\mu\|_{TV} = \|\rho^*\|_{(C_0(\mathbb{R}^d)^K)^\ast} = \|\rho\|_{C^{-n}_0(\mR^d)}$, which proves the claim.\medskip

\noindent 2)\quad We now prove the statement of the Lemma. Let $(\rho_k)\subset \mathfrak{C}^{-n}_b(\mR^d)$, $\rho\in C^{-n}_b(\mR^d)$, and assume that $\|\rho_k - \rho\|_{C^{-n}_b(\mR^d)} \to 0$ as $k\to\infty$. Since $\rho_k \in \mathfrak{C}^{-n}_b(\mR^d)$, there exists $\mu^k\in \mathcal{M}(\mathbb{R}^d)^K$ such that
\begin{align}\label{reprn}
\langle \rho_k,\phi\rangle = \sum\limits_{|\alpha|\leq n} \int_{\mathbb{R}^d} \partial^{\alpha}\phi\, d\mu^k_\alpha,\quad \phi \in C^n_b(\mathbb{R}^d).
\end{align}
By 1) there is $\tilde{\mu}^k$, such that $\|\tilde{\mu}^k\|_{TV} = \|\rho_k|_{C^n_0(\mathbb{R}^d)}\|_{C^{-n}_0(\mR^d)}$ and \begin{align}\label{reprn2}
\langle \rho_k,\phi\rangle = \sum\limits_{|\alpha|\leq n} \int_{\mathbb{R}^d} \partial^{\alpha}\phi\, d\tilde{\mu}^k_\alpha = \sum\limits_{|\alpha|\leq n} \int_{\mathbb{R}^d} \partial^{\alpha}\phi\, d\mu^k_\alpha,\quad \phi \in C^n_0(\mathbb{R}^d).
\end{align}
 We will show that the above equalities hold for all $\phi \in C^n_b(\mR^d)$.  For $R>1$, let $\tau_R$ be a smooth cut-off function, i.e., $0\leq\tau_R\leq 1$, $\tau_R = 1$ on $B_R$, $\tau_R = 0$ on $B_{R+1}^c$, and $\| \tau_R\|_{C^{n}_b(\mathbb{R}^d)} \leq M$, with $M$ independent of $R$. Then for every $\phi\in C^n_b(\mathbb{R}^d)$ and $|\alpha|\leq n$ we have $\partial^{\alpha}(\tau_R\phi) \to \partial^{\alpha}\phi$ pointwise as $R\to \infty$. Therefore, by the dominated convergence theorem,
\begin{align}
\lim\limits_{R\to\infty} \int_{\mathbb{R}^d} \partial^{\alpha}(\tau_R\phi)\, d\mu^k_\alpha = \int_{\mathbb{R}^d} \partial^{\alpha}\phi\, d\mu^k_\alpha.
\end{align}
Since $\tau_R\phi\in C^n_0(\mathbb{R}^d)$, this together with \eqref{reprn2} gives 
\begin{align}\label{reprn3}
\langle \rho_k,\phi\rangle = \sum\limits_{|\alpha|\leq n} \int_{\mathbb{R}^d} \partial^{\alpha}\phi\, d\tilde{\mu}^k_\alpha,\quad \phi \in C^n_b(\mathbb{R}^d).
\end{align}
Furthermore, we have
\begin{align*}
\|\tilde{\mu}^k\|_{TV} = \|\rho_k|_{C^n_0(\mathbb{R}^d)}\|_{C^{-n}_0(\mR^d)} \leq \|\rho_k\|_{C^{-n}_b(\mR^d)} \leq \|\tilde{\mu}^k\|_{TV},
\end{align*}
where the second inequality follows from taking the supremum over a larger set and the last inequality can be checked by hand using \eqref{reprn3}. Thus, $\|\rho_k\|_{C^{-n}_b(\mR^d)} = \|\tilde{\mu}^k\|_{TV}$. Therefore, since $\rho_k \to \rho$ in $C^{-n}_b(\mR^d)$ we find that $(\tilde{\mu}^k)$ is a Cauchy sequence in $\mathcal{M}(\mR^d)^K$ \ar{with the total variation norm}, so it converges to some $\mu\in \mathcal{M}(\mR^d)^K$. Then we have $\|\rho\|_{C^{-n}_b(\mR^d)} = \|\tilde{\mu}\|_{TV}$, and since norm convergence implies weak$^*$ convergence, \eqref{reprn3} yields
\begin{align*}
\langle \rho,\phi\rangle = \sum\limits_{|\alpha|\leq n} \int_{\mathbb{R}^d} \partial^{\alpha}\phi\, d\mu_\alpha,\quad \phi \in C^n_b(\mathbb{R}^d),
\end{align*}
which ends the proof.
\end{proof}
\section{Continuity in time}
In this whole section we assume \ref{eq:K}. We note that continuity in time was not verified in the Banach fixed point arguments in \cite{MR4309434}. The results of this section cover this issue as well.

We will often use the Duhamel formula on subintervals $[s,t]\subseteq [0,t]$. Here is the first version of this result.
\begin{lemma}\label{lem:Duhst}
Assume that $z_0\in L^\infty(\mR^d)$ and $g\in L^\infty([0,T]\times\mR^d)$. If we let
\begin{align}\label{eq:Duhamel1}
    z(t,x) = K_t\ast z_0(x) + \int_0^t\int_{\mR^d} K(t-\tau,x-y)g(\tau,y)\, dy\, d\tau,\quad t\in[0,T],\ x\in\mR^d,
\end{align}
then for every $s\in[0,t)$,
\begin{align}\label{eq:Duhamel2}
    z(t,x) = (K_{t-s}\ast z(s))(x) + \int_s^t \int_{\mR^d}K(t-\tau,x-y)g(\tau,y)\, dy\, d\tau,\quad x\in \mR^d.
\end{align}
\end{lemma}
\noindent Note that \eqref{eq:Duhamel2} corresponds to solving the same equation on a shorter time interval.
\begin{proof}
  We can assume that $s\in (0,t)$. We use the Fubini--Tonelli theorem and the Chapman--Kolmogorov equations for $K_t$ (see, e.g., \cite[Theorem~3.1.5]{MR2072890}):
  \begin{align*}
   \int_{\mR^d} K(t-s,x-y) K(s,y-z)\, dy = K(t,x-z),\quad 0<s<t,\ x,z\in\mR^d.
  \end{align*}
  By \eqref{eq:Duhamel1},
  \begin{align*}
      (K_{t-s}\ast z(s))(x) &= K_{t-s}\ast K_s\ast z_0(x) + \int_{\mR^d}K(t-s,x-y)\int_0^s\int_{\mR^d} K(s-\tau,y-z)g(\tau,z)\, dz\, d\tau \, dy\\
      &= K_t\ast z_0(x) + \int_0^s\int_{\mR^d} K(t-\tau,x-z)g(\tau,z)\, dz\, d\tau,
  \end{align*}
  which proves our result.
\end{proof}
\begin{lemma}\label{lem:contsmooth}
Let $k\geq 2$, $\sigma \in (0,\alow-1)$, and $\eps\in (0,\sigma)$. Assume that $f\in C([0,T],C^{k-1}_b(\mR^d))$ and $z_0\in C^{k+\sigma}_b(\mR^d)$. If we define
\begin{align*}
    z(t,x) = (K_t\ast z_0)(x) + \int_0^t \int_{\mR^d} K(t-s,x-y)f(s,y)\, dy\, ds,\quad t\in[0,T],\ x\in\mR^d, 
\end{align*}
 then $z\in   B([0,T],C^{k+\sigma}_b(\mR^d))\cap C([0,T],C^{k+\sigma-\eps}_b(\mR^d))$. Furthermore, for $0\leq s< t\leq T$ we have
\begin{align*}
    \|z(t,\cdot) - z(s,\cdot)\|_{C^{k+\sigma-\eps}_b(\mR^d)} \leq (\sup\limits_{t\in [0,T]}\|z(s,\cdot)\|_{C^{k+\sigma}_b(\mR^d)} + \sup\limits_{t\in [0,T]}\|f(s,\cdot)\|_{C^{k-1}_b(\mR^d)})\omega(t-s),
\end{align*}
where the modulus of continuity $\omega$ is independent of $z_0,z,$ and $f$.
\end{lemma}
\begin{proof}
     The fact that $z\in   B([0,T],C^{k+\sigma}_b(\mR^d))$ is shown in the proof of Lemma~\ref{lem:cd}. Let $0\leq s<t\leq T$. By using Duhamel's formula on $[s,t]$ (see Lemma~\ref{lem:Duhst}) we get
  \begin{align*}
      z(t,x) - z(s,x) = (K_{t-s}\ast z(s))(x) - z(s,x) + \int_s^t \int_{\mR^d} K(t-\tau,x-y) f(\tau,y)\, dy\, d\tau. 
  \end{align*}
  Therefore, for $1\leq l\leq k$,
  \begin{align*}
      D^l z(t,x) - D^l z(s,x) &= \bigg[(K_{t-s}\ast D^l z(s))(x) - D^l z(s,x)\bigg]\\ &+ \int_s^t \int_{\mR^d} D_xK(t-\tau,x-y)D^{l-1}f(\tau,y)\, dy\, d\tau =: I_1^l(x) + I_2^l(x).
  \end{align*}
  We will only estimate the $(\sigma-\eps)$-H\"older seminorms of the $k$-th derivative, the rest being similar but easier. Let $h\in \mR^d\setminus \{0\}$. We have 
  \begin{align}\label{eq:i1i2}
      |D^k z(t,x+h) - D^kz(t,x) - D^k z(s,x+h) + D^k z(s,x)| \leq |I_1^k(x+h) - I_1^k(x)| + |I_2^k(x+h) - I_2^k(x)|.
  \end{align}
  We first estimate the term with $I_1^k$. For $|h|\geq 1$ we have $\frac{1}{|h|^{\sigma-\eps}} \leq 1$, so we get
  \begin{align*}
      \sup\limits_{|h|\geq 1}\frac{|I_1^k(x+h) - I_1^k(x)|}{|h|^{\sigma-\eps}} &\leq \sup\limits_{|h|\geq 1}\int_{\mR^d} K(t-s,y) |D^k z(s,x+h-y) - D^kz(s,x-y) - D^kz(s,x+h) + D^kz(s,x)|\, dy\\
      &\leq 4\|z(s,\cdot)\|_{C^{k+\sigma}_b(\mR^d)}\int_{\mR^d} K(t-s,y)(1\wedge|y|^{\sigma})\, dy,
  \end{align*}
  which converges to 0 as $t-s\to 0^+$, because $K(t-s) \to \delta_0$. Now, let $0<|h|<1$. Then
  \begin{align*}
      |I_1^k(x+h) - I_1^k(x)| &\leq \int_{\mR^d} K(t-s,y) |D^k z(s,x+h-y) - D^kz(s,x-y) - D^kz(s,x+h) + D^kz(s,x)|\, dy\\
      &\leq 2\|z(s,\cdot)\|_{C^{k+\sigma}_b(\mR^d)}\bigg(\int_{B(0,|h|)} K(t-s,y)|y|^{\sigma}\, dy + |h|^{\sigma}\int_{B(0,|h|)^c} K(t-s,y)\, dy\bigg)\\
      &\leq 2\|z(s,\cdot)\|_{C^{k+\sigma}_b(\mR^d)}|h|^{\sigma-\eps}\bigg(\int_{B(0,|h|)} K(t-s,y)|y|^{\eps}\, dy + |h|^{\eps}\int_{B(0,|h|)^c} K(t-s,y)\, dy\bigg)\\
      &= 2\|z(s,\cdot)\|_{C^{k+\sigma}_b(\mR^d)}|h|^{\sigma-\eps}\int_{\mR^d} K(t-s,y)(|h|^{\eps}\wedge |y|^{\eps})\, dy.
  \end{align*}
  Therefore, 
  \begin{align*}
  \sup\limits_{0<|h|<1} \frac{|I_1^k(x+h) - I_1^k(x)|}{|h|^{\sigma-\eps}} \leq 2\|z(s,\cdot)\|_{C^{k+\sigma}_b(\mR^d)}\int_{\mR^d} K(t-s,y)(1\wedge |y|^{\eps})\, dy
  \end{align*}
  and the remaining integral converges to 0 as $t-s \to 0^+$. Hence we get
  \begin{align*}
      \sup\limits_{h\in \mR^d\setminus \{0\}} \frac{|I_1^k(x+h) - I_1^k(x)|}{|h|^{\sigma-\eps}} \leq \sup\limits_{s\in [0,T]}\|z(s,\cdot)\|_{C^{k+\sigma}_b(\mR^d)}\omega(t-s),
  \end{align*}
  where the modulus of continuity $\omega$ does not depend on $z_0,z,$ and $f$.
  
  \noindent The term with $I_2^k$ in \eqref{eq:i1i2} is more straightforward: by Remark~\ref{rem:K} we get
  \begin{align*}
      \frac{|I_2^k(x+h) - I_2^k(x)|}{|h|^{\sigma-\eps}} &\leq  \sup\limits_{t\in [0,T]}\|f(t,\cdot)\|_{C^{k-1}_b(\mR^d)}\int_s^t \int_{\mR^d}\frac{|D_x K(t - \tau,x+h-y) - D_x K(t-\tau,x-y)|}{|h|^{\sigma-\eps}}\, dy\\
      &\leq C \sup\limits_{t\in [0,T]}\|f(t,\cdot)\|_{C^{k-1}_b(\mR^d)}\int_s^t (t-\tau)^{-\frac{1+\sigma-\eps}{\alow}}\, d\tau,
  \end{align*}
  where the last constant does not depend on $h,s,$ and $t$. Thus,
  \begin{align*}
      \sup\limits_{h\in \mR^d\setminus \{0\}} \frac{|I_2^k(x+h) - I_2^k(x)|}{|h|^{\sigma-\eps}} \leq \sup\limits_{t\in [0,T]}\|f(t,\cdot)\|_{C^{k-1}_b(\mR^d)}\omega(t-s),
  \end{align*}
  where the modulus of continuity $\omega$ does not depend on $z_0,z,$ and $f$.
  This ends the proof.
\end{proof}

Below we give results of the similar type for negative H\"older norms.
\begin{lemma}\label{lem:rhotimeDuh}
Let $\gamma,\delta>0$ and assume that $V_1\in (C([0,T],C^\gamma_b(\mR^d)))^d$, $V_2 \in (L^{\infty}([0,T],\mathcal{M}(\mR^d)))^d$, $\rho_0 \in C^{-\gamma}_b(\mR^d)$, and $\widetilde{\rho}\in L^{\infty}((0,T),C^{-\gamma}_b(\mR^d))$. If we define $\rho$ for $\xi\in C^{\gamma}_b(\mR^d)$ and $t\in [0,T]$ by the formula
\begin{align*}
    \langle \rho(t),\xi\rangle = \langle \rho_0, K_t^\ast \ast  \xi\rangle + \int_0^t \langle \widetilde{\rho}(s), V_1(s,\cdot)(D_yK_{t-s}^\ast \ast\xi)(\cdot)\rangle\, ds + \int_0^t \int_{\mR^d} D_yK^\ast_{t-s}\ast\xi(y)\, V_2(s,dy)\, ds,
\end{align*}
then $\rho\in C([0,T],C^{-\gamma-\delta}_b(\mR^d))$.
\end{lemma}
\begin{proof}
    Note that a counterpart of \eqref{eq:Duhamel2} holds: for every $\xi\in C^{\gamma+\delta}_b(\mR^d)$ we have
    \begin{align}\label{eq:weakDuhamelst}
        \langle \rho(t),\xi\rangle = \langle \rho(s),K_{t-s}^\ast\ast\xi\rangle + \int_s^t \langle \widetilde{\rho}(\tau), V_1(\tau,\cdot)(D_yK_{t-\tau}^\ast \ast\xi)(\cdot)\rangle\, d\tau + \int_s^t \int_{\mR^d} D_yK^\ast_{t-\tau}\ast\xi(y)\, V_2(\tau,dy)\, d\tau.
    \end{align}
    Indeed, the equality follows from testing $\rho(s)$ with $K_{t-s}^\ast \ast \xi$ and the fact that by the Chapman--Kolmogorov equations, 
    \begin{align*}
        (D_y K_{s-\tau}^\ast)\ast K_{t-s}^\ast = D_y(K_{s-\tau}^\ast\ast K_{t-s}^\ast) = D_yK_{t-\tau}^\ast.
    \end{align*}
    The remainder of the proof consists in a dual estimate of $\rho(t) - \rho(s)$, which uses the fact that $\|K_{t-s}\ast \xi - \xi\|_{C^{\gamma}_b(\mR^d)} \leq C\omega(t-\tau)\|\xi\|_{C^{\gamma+\delta}_b(\mR^d)}$, obtained in the proof of Lemma~\ref{lem:contsmooth}. We skip the details.
\end{proof}

\noindent In the next estimate we work with very weak solutions appearing in Theorem~\ref{th:linsyst}. \begin{lemma}\label{lem:rhotime}
Suppose that all the assumptions of Theorem~\ref{th:linsyst} hold and that the pair $(z_\lambda,\rho_\lambda)$ solves system \eqref{eq:sigmasyst}. Then $\rho_\lambda \in C([t_0,T],C^{-k-1-\sigma}_b(\mR^d))$.
\end{lemma}
\begin{proof}
Let $t_0\leq t_2<t_1\leq T$ and $\xi\in C^{k+1+\sigma}_b(\mR^d)$. For $i=1,2$ we define $\phi_i$ to be the solution of the following problem:
\begin{align*}
    \begin{cases}
    \partial_t \phi_i(t,x) + \mL \phi_i(t,x) - V(t,x)D\phi_i(t,x) = 0,\quad &(t,x)\in (t_0,t_i)\times \mR^d,\\
    \phi_i(t_i,x) = \xi(x),\quad &x\in \mR^d.
    \end{cases}
\end{align*}
By Lemma~\ref{lem:cd} we have that $\phi_i\in B([t_0,T],C^{k+1+\sigma}_b(\mR^d))\cap C([t_0,T],C^{k+1+\sigma-\delta}_b(\mR^d))$ for arbitrarily small $\delta>0$. By equations for $\rho_\lambda$ and $\phi_i$ we get
\begin{align*}
    |\langle \rho_\lambda(t_1)-\rho_\lambda(t_2),\xi\rangle| &\leq \lambda|\langle \rho_0,\phi_1(t_0) - \phi_2(t_0) \rangle| +\bigg|\int_{t_0}^{t_1} \langle D\phi_1(s),c(s)\rangle\, ds - \int_{t_0}^{t_2} \langle D\phi_2(s),c(s)\rangle\, ds\bigg| \\
    &+\bigg|\int_{t_0}^{t_1}\int_{\mR^d} D\phi_1\Gamma Dz_\lambda(s,x)\, m(s,dx)\, ds -  \int_{t_0}^{t_2}\int_{\mR^d} D\phi_2\Gamma Dz_\lambda(s,x)\, m(s,dx)\, ds\bigg| =: I_1 + I_2 + I_3.
\end{align*}
 Furthermore, we have
 \begin{align*}
     I_1 &\leq \|\rho_0\|_{C^{-k-1-\sigma+\delta}_b(\mR^d)}\|\phi_1(t_0) - \phi_2(t_0)\|_{C^{k+1+\sigma-\delta}_b(\mR^d)},\\
     I_2 &\leq \bigg|\int_{t_0}^{t_2}\langle D\phi_1(s) - D\phi_2(s),c(s)\rangle\, ds\bigg| + \bigg|\int_{t_2}^{t_1} \langle D\phi_1(s),c(s)\rangle\, ds\bigg|\\
     &\leq \|c\|_{L^1(C^{-k-\sigma+\delta}_b(\mR^d))}\big(\sup\limits_{t\in [t_0,t_2]}\|\phi_1(t) - \phi_2(t)\|_{C^{k+1+\sigma-\delta}_b(\mR^d)} + (t_1 - t_2)\sup\limits_{t\in [t_0,t_1]}\|\phi_1(t)\|_{C^{k+1+\sigma}_b(\mR^d)}\big),\\
     I_3 &\leq \bigg|\int_{t_0}^{t_2}\int_{\mR^d} (D\phi_1 - D\phi_2)\Gamma Dz_\lambda(s,x)\, m(s,dx)\, ds -  \int_{t_2}^{t_1}\int_{\mR^d} D\phi_1\Gamma Dz_\lambda(s,x)\, m(s,dx)\, ds\bigg|\\
     &\leq C\|Dz_\lambda(s)\|_{\infty}\big(T\|D\phi_1 - D\phi_2\|_{\infty} + (t_1-t_2)\|D\phi_1\|_{\infty}\big).
 \end{align*}
 By the above estimates, the statement of the lemma is true provided that \begin{align*}
     &\sup\limits_{t\in [t_0,t_2]}\|\phi_1(t)\|_{C^{k+1+\sigma}_b(\mR^d)}\leq C\|\xi\|_{C^{k+1+\sigma}_b(\mR^d)},\ {\rm and}\\ &\sup\limits_{t\in [t_0,t_2]}\|\phi_1(t) - \phi_2(t)\|_{C^{k+1+\sigma-\delta}_b(\mR^d)} \leq  C\|\xi\|_{C^{k+1+\sigma}_b(\mR^d)}\omega(t_1-t_2),
    \end{align*}
    for some modulus of continuity $\omega$. The first inequality follows from Lemma~\ref{lem:cd}. For the second inequality, note that $\phi = \phi_1 - \phi_2$ solves the problem
    \begin{align*}
    \begin{cases}
    \partial_t \phi(t,x) + \mL \phi(t,x) - V(t,x)D\phi(t,x) = 0,\quad &(t,x)\in (t_0,t_2)\times \mR^d,\\
    \phi(t_2,x) = \phi_1(t_2,x) - \xi(x) = \phi_1(t_2,x) -\phi_1(t_1,x),\quad &x\in \mR^d.
    \end{cases}
    \end{align*}
    By Lemma~\ref{lem:cd} (b) we find that
    \begin{align*}
        \sup\limits_{t\in [t_0,t_2]}\|\phi(t)\|_{C^{k+1+\sigma-\delta}_b(\mR^d)}\leq C \|\phi_1(t_2,x) -\phi_1(t_1,x)\|_{C^{k+1+\sigma-\delta}_b(\mR^d)} \leq C\|\xi\|_{C^{k+1+\sigma}_b(\mR^d)}\omega(t_2-t_1),
    \end{align*}
    with $\omega$ independent of $\xi$. This ends the proof.
\end{proof}
\section{From mild to very weak solutions}
\begin{lemma}\label{lem:mildweak}
Assume that $V_1 \in (L^\infty([0,T]\times \mR^d))^d$, $V_2\in (L^\infty([0,T],L^1(\mR^d)))^d$, and $\rho_0 \in L^1(\mR^d)$. If $\rho\in B([0,T],L^1(\mR^d))$ is product measurable and satisfies
\begin{align}\label{eq:mildsol}
    \rho(t,x) = K_t\ast \rho_0(x) + \int_0^t\int_{\mR^d} D_y K(t-s,x-y)(V_1\rho + V_2)(s,y)\, dy\, ds,
\end{align}
for all $t\in [0,T]$, almost everywhere in  $x\in \mR^d$, then for every $\phi \in C([0,T], C^2_b(\mR^d))$ such that $\partial_t \phi\in C_b([0,T]\times \mR^d)$, and every $t\in[0,T]$ we have
\begin{equation}
    \begin{split}\label{eq:weakl1}
    &\int_{\mR^d} \phi(t,x)\rho(t,x)\, dx - \int_{\mR^d} \phi(0,x)\rho_0(x)\, dx\\ =  &\int_0^t\int_{\mR^d} (\partial_t\phi(s,x) + \mL^*\phi(s,x))\rho(s,x)\, dx\, ds + \int_0^t\int_{\mR^d} D_x\phi(s,x)(V_1\rho+ V_2)(s,x)\, dx\, ds.
    \end{split}
\end{equation}
\end{lemma}
\begin{proof}
We show the result directly for $\phi \in C([0,T],C^2_c(\mR^d))$ such that $\partial_t \phi \in C_c([0,T]\times\mR^d)$. Other test functions can be approximated using the smooth cut-off functions \ar{$\chi_R$ defined at the beginning of Section~\ref{sec:prelim}. Indeed, since $\phi \chi_R$, $\partial_t(\phi\chi_R)$, $D_x (\phi\chi_R)$, {$D^2_x (\phi\chi_R)$,}  and $\mL^\ast (\phi\chi_R)$ are uniformly bounded and converge to $\phi$, $\partial_t\phi$, $D_x \phi$, {$D^2_x \phi$,} and $\mL^\ast \phi$\footnote{Note that for $x$ fixed and $R>1+|x|$ we have $|\mL^\ast\big(\phi(\cdot,t)-\phi(\cdot,t)\chi_R\big)(x)|\leq C(\|\phi(\cdot,t)-\phi(\cdot,t)\chi_R\|_{C_b^2(B(0,R))}+2\|\phi(\cdot,t)\|_{C_b(\mathbb R^d)}\nu(B(0,R)^c)) \lesssim 0+\nu(B(0,R)^c)\to 0$ as $R\to \infty$. 
} pointwise as $R\to \infty$, and $\rho$ and $V_2$ are bounded in $L^1(\mR^d)$, we are in a position to use the dominated convergence theorem.}

In order to prove the statement for $\phi \in C([0,T],C^2_c(\mR^d))$, we multiply \eqref{eq:mildsol} by $\phi(t,x)$, integrate over $x$, and subtract $\int \phi(0)\rho_0$, getting
\begin{align}
    &\int_{\mR^d}\rho(t,x)\phi(t,x)\, dx - \int_{\mR^d} \phi(0,x)\rho_0(x)\, dx\nonumber\\
    \label{eq:mildint} = &\int_{\mR^d}K_t\ast \rho_0(x)\phi(t,x)- \phi(0,x)\rho_0(x)\, dx + \int_{\mR^d}\int_0^t\int_{\mR^d} D_y K(t-s,x-y)(V_1\rho + V_2)(s,y)\phi(t,x)\, dy\, ds\, dx.
\end{align}
The goal is then to show that the right-hand side in the equation above is equal to the right-hand side of \eqref{eq:weakl1}. We start by investigating the first integral of \eqref{eq:mildint}. By the fundamental theorem of calculus we find that for $h\in (0,t)$,
\begin{align*}
    \int_{\mR^d} K_t\ast \rho_0(x)\phi(t,x)-K_h\ast \rho_0(x)\phi(h,x)\, dx = \int_{\mR^d} \int_h^t\partial_s\big((K_s\ast \rho_0(x))\phi(s,x)\big)\, ds\, dx.
\end{align*}
Since $K$ is a fundamental solution, for every $s\in (h,t)$ we have
\begin{align*}
    \partial_s\big((K_s\ast \rho_0(x))\phi(s,x)\big) = (\partial_s K_s)\ast \rho_0(x)\phi(s,x) + K_s\ast\rho_0(x) \partial_s\phi(s,x) = (\mL K_s)\ast \rho_0(x)\phi(s,x) + K_s\ast\rho_0(x) \partial_s\phi(s,x),
\end{align*}
so by the Fubini-Tonelli theorem, \ref{eq:K} and the fact that $\rho_0\in L^1(\mR^d)$ we may interchange the order of integration and get
\begin{align}\label{eq:rho0}
    \int_{\mR^d} K_t\ast \rho_0(x)\phi(t,x)-K_h\ast \rho_0(x)\phi(h,x)\, dx = \int_h^t\int_{\mR^d} \partial_s\big((K_s\ast \rho_0(x))\phi(s,x)\big)\, dx\, ds.
\end{align}
Furthermore, we have
\begin{equation}\label{eq:dct}
\begin{split}
   \bigg|\int_{\mR^d} \partial_s\big((K_s\ast \rho_0(x))\phi(s,x)\big)\, dx\bigg| &\leq \bigg|\int_{\mR^d}(\mL K_s)\ast \rho_0(x)\phi(s,x)\, dx\bigg| + \bigg| \int_{\mR^d} K_s\ast \rho_0(x)\partial_s \phi(s,x)\, dx\bigg|\\
    &=\bigg|\int_{\mR^d}K_s\ast \rho_0(x)\mL^*\phi(s,x)\, dx\bigg| + \bigg| \int_{\mR^d} K_s\ast \rho_0(x)\partial_s \phi(s,x)\, dx\bigg| \leq C(\phi),
\end{split}
\end{equation}
In the second equality we used \ref{eq:K}, $\phi(s)\in C^2_c(\mR^d)$, and $\rho_0\in L^1(\mR^d)$.\footnote{We have $K_t\in L^2(\mR^d)$ and $\mL K_t\in L^2(\mR^d)$, hence $\mathcal{F}((\mL K_s) \ast \rho_0) = \mathcal{F}(\mL K_s) \mathcal{F}\rho_0 = \Psi \mathcal{F}(K_s \ast \rho_0) = \mathcal{F}(\mL(K_s\ast \rho_0))$, where $\mathcal{F}$ is the Fourier transform and $\Psi$ is the Fourier multiplier of $\mL$. Thus, $(\mL K_s)\ast\rho_0 = \mL(K_s \ast \rho_0)$. Furthermore, since $K_s\ast\rho_0\in C_0^\infty(\mR^d)$, we get $\int_{\mR^d} (\mL K_s)\ast\rho_0(x) \phi(s,x)\, dx = \int_{\mR^d}K_s \ast \rho_0(x) \mL^\ast\phi(s,x)\, dx$, see the discussion in Section~\ref{sec:hk}.} Note that $C(\phi)$ is independent of $s$.

We let $h\to 0^+$ in \eqref{eq:rho0}. By \eqref{eq:dct}, the dominated convergence theorem, and the strong $L^1$-continuity of the semigroup corresponding to $K_t$ \cite[Theorem~3.4.2]{MR2072890}, we obtain
\begin{align*}
    &\int_{\mR^d}\big(K_t\ast \rho_0(x)\phi(t,x)- \phi(0,x)\rho_0(x)\big)\, dx = \int_0^t\int_{\mR^d} \partial_s(K_s\ast \rho_0(x)\phi(s,x))\, dx\, ds\\
    = &\int_0^t \int_{\mR^d}K_s \ast \rho_0(x)(\partial_s\phi(s,x) + \mL^*\phi(s,x))\, dx\, ds =: I_1.
\end{align*}
Now we will investigate the second integral of \eqref{eq:mildint}:
\begin{align}\label{eq:gradterm}\int_{\mR^d}\int_0^t\int_{\mR^d} D_y K(t-s,x-y)(V_1\rho + V_2)(s,y)\phi(t,x)\, dy\, ds\, dx.\end{align}
By using Tonelli's theorem, \ref{eq:K}, and integrating in the order $dx\, dy\, ds$, 
we see that it converges absolutely, therefore we may interchange the order of integration and get that \eqref{eq:gradterm} is equal to
\begin{align*}
   &\int_0^t \int_{\mR^d}\int_{\mR^d} D_x K(t-s,x-y)(V_1\rho + V_2)(s,y)\phi(t,x)\, dx\, dy\, ds\\ = &\int_0^t \int_{\mR^d}\int_{\mR^d} K(t-s,x-y)(V_1\rho + V_2)(s,y)D_x\phi(t,x)\, dx\, dy\, ds\\
   = &\int_0^t \int_{\mR^d}(V_1\rho + V_2)(s,y)\bigg(\int_{\mR^d} K(t-s,x-y)D_x\phi(t,x)\, dx - D_y\phi(s,y)\bigg)\, dy\, ds \\
   + &\int_0^t \int_{\mR^d} (V_1\rho + V_2)(s,y)D_y\phi(s,y)\, dy\, ds =: J_1+J_2.
\end{align*}
Since $D_x\phi \in C([0,T], C_c(\mR^d))$, by using \ref{eq:K}, the strong continuity of the semigroup corresponding to $K^*_t$ (recall that $K_t(x-y) = K^\ast_t(y-x)$), the Fubini--Tonelli theorem, and arguments similar to \eqref{eq:rho0} and \eqref{eq:dct},\footnote{Note, in particular, that because of \ref{eq:K} and $\phi\in C([0,T],C^2_c(\mR^d))$, for every $\tau\in(s,t)$ we have $\int_{\mR^d} \partial_{\tau} D_xK(\tau-s)  \phi(\tau) = \int_{\mR^d} D_x\mL K(\tau-s) \phi(\tau) = \int_{\mR^d}  D_x K(\tau-s) \mL^*\phi(\tau)$ and the last of these integrals is bounded from above by $C(\tau-s)^{-\frac 1{\alow}}$, which is integrable with respect to $\tau$.} we find that 
\begin{align*}
    &\int_{\mR^d} K(t-s,x-y)D_x\phi(t,x)\, dx - D_y\phi(s,y) = \int_s^t\int_{\mR^d} \partial_{\tau}\big(D_xK(\tau - s, x-y)\phi(\tau,x)\big)\, dx\, d\tau\\
    &= \int_s^t\int_{\mR^d} \partial_{\tau} D_xK(\tau - s, x-y)\phi(\tau,x)\, dx\, d\tau + \int_s^t\int_{\mR^d}  D_x K(\tau - s, x-y)\partial_{\tau} \phi(\tau,x)\, dx\, d\tau.
\end{align*}
Therefore,
\begin{align*}
    J_1 = &\int_0^t \int_{\mR^d}\int_s^t\int_{\mR^d}(V_1\rho + V_2)(s,y) \partial_{\tau} K(\tau - s, x-y)D_x\phi(\tau,x)\, dx\, d\tau\, dy\, ds\\
    + &\int_0^t \int_{\mR^d}\int_s^t\int_{\mR^d}(V_1\rho + V_2)(s,y)  D_xK(\tau - s, x-y)\partial_{\tau} \phi(\tau,x)\, dx\, d\tau\, dy\, ds =: J_{1,1} + J_{1,2}.
\end{align*}
We have 
\begin{align*}
    J_{1,1} = &\int_0^t \int_{\mR^d}\int_s^t\int_{\mR^d}(V_1\rho + V_2)(s,y) \mL K(\tau - s, x-y)D_x\phi(\tau,x)\, dx\, d\tau\, dy\, ds\\
    = &\int_0^t \int_{\mR^d}\int_s^t\int_{\mR^d}(V_1\rho + V_2)(s,y) D_x K(\tau - s, x-y)\mL^* \phi(\tau,x)\, dx\, d\tau\, dy\, ds\\
    =&\int_0^t \int_{\mR^d} \mL^* \phi(\tau,x) \int_0^\tau\int_{\mR^d} (V_1\rho + V_2)(s,y) D_x K(\tau - s, x-y)\, dy\, ds\, dx\, d\tau.
\end{align*}
The change in the order of integration is allowed because the second integral is absolutely convergent. Furthermore,
\begin{align*}
    J_{1,2} = \int_0^t \int_{\mR^d} \partial_{\tau} \phi(\tau,x) \int_0^\tau\int_{\mR^d} (V_1\rho + V_2)(s,y) D_x K(\tau - s, x-y)\, dy\, ds\, dx\, d\tau.
\end{align*}
Hence, switching the role of $s$ and $\tau$ we get
\begin{align*}
   J_1= J_{1,1} + J_{1,2}= \int_0^t \int_{\mR^d} (\partial_{s} \phi(s,x) +\mL^*\phi(s,x)) \int_0^s\int_{\mR^d} (V_1\rho + V_2)(\tau,y) D_x K(s-\tau, x-y)\, dy\, d\tau \, dx\, ds.
\end{align*}
We finish the proof by noticing that \eqref{eq:mildsol} yields
\begin{align*}
    (I_1 + J_1) + J_2 = \int_0^t\int_{\mR^d} (\partial_s\phi(s,x) + \mL^*\phi(s,x))\rho(s,x)\, dx\, ds + \int_0^t\int_{\mR^d} D_x\phi(s,x)(V_1\rho+ V_2)(s,x)\, dx\, ds.
\end{align*}
\end{proof}
In the next result we discuss mild solutions in negative order H\"older spaces. The arguments are quite similar, but we present certain parts of the proof in full detail in order to pinpoint the class of admissible test functions. This is crucial for identification of solutions $\rho$ obtained in Theorem~\ref{th:linsyst} and Lemma~\ref{lem:l1cneg}.
\begin{lemma}\label{lem:milddistrcneg}
Suppose that the assumptions of Lemma~\ref{lem:l1cneg} hold true. Let $\rho\in C_b((0,T],C^{-n}_b(\mR^d))$ be the mild solution obtained by fixed point argument in the proof of Lemma~\ref{lem:l1cneg}, i.e.,
\begin{align*}
     \langle\rho(t),\xi\rangle &= \langle \rho_0,K_t^\ast \ast \xi\rangle + \int_0^t \bigg\langle \rho(s),V_1(s,\cdot) (D_y K_{t-s}^\ast \ast \xi)(\cdot)\bigg\rangle\, ds\\
      &+ \int_0^t \int_{\mR^d} \int_{\mR^d}D_yK(t-s,x-y)\, V_2(s,dy) \xi(x)\, dx\, ds,\quad t\in [0,T], \ \xi \in C^n_b(\mR^d).
\end{align*}
Then $\rho$ is a distributional
solution of \eqref{eq:FPcneg}, that is, for every $\phi\in C([0,T],C^n_b(\mR^d))$ such that $D^n\phi$ is uniformly continuous, $\partial_t \phi \in C_b((0,T),C_b(\mR^d))$, and $\partial_t \phi + \mL \phi \in B([0,T],C^{n-1}_b(\mR^d))$
we have
\begin{align*}
    \langle \rho(t),\phi(t)\rangle - \langle \rho_0,\phi(0)\rangle = \int_0^t \langle (\partial_t\phi + \mL\phi - V_1D\phi)(s),\rho(s)\rangle\, ds + \int_0^t\int_{\mR^d} D\phi(s,y)\, V_2(s,dy)\, ds,\quad t\in [0,T].
\end{align*}
\end{lemma}
\begin{proof}
  By using the mild formulation we get
  \begin{align*}
      \langle \rho(t),\phi(t)\rangle - \langle \rho_0,\phi(0)\rangle &= \langle \rho_0,K_t^\ast \ast \phi(t) - \phi(0)\rangle + \int_0^t \langle \rho(s),V_1(s) D_y K_{t-s}^\ast \ast \phi(t)\rangle\, ds\\
      &+ \int_0^t \int_{\mR^d} \int_{\mR^d}D_yK(t-s,x-y)\, V_2(s,dy) \phi(t,x)\, dx\, ds \\
      &= \langle \rho_0,K_t^\ast \ast \phi(t) - \phi(0)\rangle + \int_0^t \langle \rho(s)V_1(s) + V_2(s), D_y K_{t-s}^\ast \ast \phi(t)\rangle\, ds =: I_1 + I_2.
  \end{align*}
  Then,
  \begin{align}\label{eq:i1start}
      I_1 = \lim\limits_{s\to 0^+} \langle \rho_0,K_t^\ast \ast \phi(t) - K_s^\ast\ast \phi(s)\rangle = \lim\limits_{s\to 0^+}  \int_s^t \partial_t \langle \rho_0, K_\tau^\ast \ast \phi(\tau)\rangle\, d\tau.
  \end{align}
  We will now justify the above equalities. In the first of them we have used that $\phi\in C([0,T],C^{n}_b(\mR^d))$ and that $\phi(0)$ and its spatial derivatives of order up to $n$ are bounded and
 uniformly continuous in $x$, so that $K_s^\ast\ast\phi(s) \to \phi(0)$ in $C^n_b(\mR^d)$. 
  To see that the second equality in \eqref{eq:i1start} is valid, we will show that the time derivative appearing there exists and is bounded and continuous. For $\tau \in (s,t)$ and $h\neq 0$ small we have
\begin{align}\label{eq:diffquot}
\frac{K_{\tau+h}^\ast\ast\phi(\tau+h) - K_\tau^\ast\ast\phi(\tau)}{h} =  \frac{K_{\tau}^\ast\ast(\phi(\tau+h) - \phi(\tau))}{h} +  \frac{(K_{\tau+h}^\ast-K_\tau^\ast)\ast\phi(\tau+h)}{h}.
  \end{align}
  By \ref{eq:K} and the fact that $\partial_t\phi\in C_b((0,T),C_b(\mR^d))$
  the first term in \eqref{eq:diffquot} converges to $K_\tau^\ast\ast \partial_t\phi (\tau)$ in $C^n_b(\mR^d)$ as $h\to 0$.
  Indeed, we have $|(\phi(\tau+h,x)-\phi(\tau,x))/h - \partial_t\phi(\tau,x)| = |\partial_t\phi (\tau_x,x) - \partial_t\phi(\tau,x)| \leq \omega(h)$, so $\|(\phi(\tau+h)-\phi(\tau))/h - \partial_t\phi(\tau)\|_{\infty} \to 0$ as $h\to 0$. Therefore \begin{align}\label{eq:cbl1}
\bigg\|\frac{K_{\tau}^\ast\ast(\phi(\tau+h) - \phi(\tau))}{h}-K_\tau^\ast\ast \partial_t\phi (\tau)\bigg\|_{C^n_b(\mR^d)} \leq \Big\|\frac{\phi(\tau+h)-\phi(\tau)}{h} - \partial_t\phi(\tau)\Big\|_{\infty} \sum\limits_{k=0}^n \|D^k K^\ast_\tau\|_{L^1(\mR^d)} \mathop{\longrightarrow}\limits_{h\to 0} 0.
  \end{align}
  By \ref{eq:K}, the fact that $(\partial_t - \mL^\ast)K_t^\ast = 0$, and the fact that
 $\phi\in C([0,T],C_b(\mR^d))$ 
 we find that the second term of \eqref{eq:diffquot} converges to $\partial_t K_\tau^\ast\ast \phi(\tau)$ in $C^n_b(\mR^d)$ as $h\to 0^+$. 
In order to see that we split: \begin{align}\label{eq:split2}
     \frac{(K_{\tau+h}^\ast-K_\tau^\ast)\ast\phi(\tau+h)}{h} = \frac{(K_{\tau+h}^\ast-K_\tau^\ast)}{h}\ast (\phi(\tau+h) - \phi(\tau)) + \frac{(K_{\tau+h}^\ast-K_\tau^\ast)\ast\phi(\tau)}{h}.
 \end{align}
 and since 
$\frac{(K_{\tau+h}^\ast-K_\tau^\ast)}{h} - \partial_\tau K^\ast_\tau =  \int_0^1 \partial_\tau (K^*_{\tau+\lambda h} - K^*_\tau)\, d\lambda=   \int_0^1 \int_0^1 \partial_\tau^2 K^\ast_{\tau + \lambda\mu h} \,\lambda h\,d\mu\, d\lambda$,
 it follows that for $0\leq k\leq n$,
 \begin{align*}
\bigg\|D^k\frac{(K_{\tau+h}^\ast-K_\tau^\ast)}{h} - D^k\partial_\tau K^\ast_\tau\bigg\|_{L^1(\mR^d)} 
\leq 
h\int_0^1\int_0^{1} \|D^k\mL^2 K^\ast_{\tau + \lambda\mu h}\|_{L^1(\mR^d)}\, d\mu\, d\lambda,
 \end{align*}
 which goes to 0 as $h\to 0$. Thus we see that the first term on the right-hand side of \eqref{eq:split2} vanishes and the second term converges to $\partial_t K_\tau^\ast\ast \phi(\tau)$ by an estimate similar to \eqref{eq:cbl1}.
 This proves that $\partial_t \langle \rho_0, K_\tau^\ast \ast \phi(\tau)\rangle = \langle \rho_0,  \partial_t(K_\tau^\ast \ast \phi(\tau))\rangle$ exists and is continuous and bounded, therefore \eqref{eq:i1start} is valid. 
 
 Using $(\partial_t - \mL^\ast)K_\tau^\ast = 0$, $K_\tau^\ast\in C_0^2(\mR^d)$, $K_\tau^\ast,\mL^\ast K_\tau^\ast \in L^1(\mR^d)$, and 
$\phi(\tau)\in C^2_b(\mR^d)$ we get $\partial_t K_\tau^\ast\ast \phi(\tau) = K_\tau^\ast\ast \mL \phi(\tau)$.  
Therefore, for every $\tau\in (0,T]$
  we find that
  \begin{align*}
      I_1 = \lim\limits_{s\to 0^+}  \int_s^t  \langle \rho_0, K_\tau^\ast\ast (\partial_t\phi (\tau)+  \mL \phi(\tau))\rangle\, d\tau.
  \end{align*}
  Furthermore, since
 $\partial_t\phi + \mL\phi \in B([0,T],C^{n-1}_b(\mR^d))$,
    by \ref{eq:K} we get
  \begin{align*}
      \|K_{\tau}^\ast\ast (\partial_t + \mL)\phi(\tau)\|_{C^n_b(\mR^d)} \leq C\tau^{-\frac 1{\alow}} \|(\partial_t + \mL)\phi(\tau)\|_{C^{n-1}_b(\mR^d)},\quad \tau\in (0,T).
  \end{align*}
  Therefore, by the dominated convergence theorem,
  \begin{align*}
      I_1  = \int_0^t \langle \rho_0, K_\tau^\ast\ast (\partial_t\phi (\tau)+  \mL \phi(\tau))\rangle\, d\tau.
  \end{align*}
  In $I_2$ we add and subtract a term with $D_y\phi(s)$:
  \begin{align*}
      I_2 &= \int_0^t \langle \rho(s)V_1(s) + V_2(s), D_y K_{t-s}^\ast \ast \phi(t) - D_y\phi(s)\rangle\, ds\\ &+ \int_0^t \langle \rho(s)V_1(s)+V_2(s), D_y\phi(s)\rangle\, ds =: J_1 + J_2.
  \end{align*}
  We have
  \begin{align}
      J_1 &= \int_0^t \lim\limits_{h\to 0^+}\langle \rho(s)V_1(s)+V_2(s), D_y K_{t-s}^\ast \ast \phi(t) - D_y K_{h}^\ast \ast \phi(s+h)\rangle\, ds\nonumber \\
      &= \int_0^t \lim\limits_{h\to 0^+}\int_{s+h}^t \partial_\tau \langle \rho(s)V_1(s)+V_2(s), D_y K_{\tau-s}^\ast \ast \phi(\tau)\rangle\,d\tau\, ds.\label{eq:dtau}
  \end{align}
  The first equality is true because for every $s\in (0,T]$ we have $V_2(s)\in C^{-0}_b(\mR^d)$ and $\rho(s)V_1(s)\in C^{-n+1}_b(\mR^d)$ (by the mild solution estimates of Lemma~\ref{lem:l1cneg}). Therefore, since $\phi\in C([0,T],C^n_b(\mR^d))$ with $D^n\phi$ uniformly continuous, we get that $K_h^\ast\ast D_y\phi(s+h) \to D_y\phi(s)$ in $C^{n-1}_b(\mR^d)$. The second equality is justified by showing that the derivative exists and is continuous and bounded, similar to what was done in $I_1$: for $\tau\in (s,t)$ and $h$ small we have
  \begin{align*}
      \frac{D_y K_{\tau+h-s}^\ast \ast \phi(\tau+h) - D_y K_{\tau-s}^\ast \ast \phi(\tau)}{h} = \frac{(D_y K_{\tau+h-s}^\ast - D_y K_{\tau-s}^\ast) \ast \phi(\tau+h)}{h} + \frac{D_y K_{\tau-s}^\ast \ast (\phi(\tau+h) -  \phi(\tau))}{h}.
  \end{align*}
  For the same reasons as before this expression converges in $C^{n-1}_b(\mR^d)$ to
     $\partial_\tau D_yK_{\tau-s}^\ast\ast \phi(\tau) + D_yK_{\tau - s}^\ast\ast \partial_{\tau}\phi(\tau)\in C([s,t],C^{n-1}_b(\mR^d))$, therefore the $\tau$-derivative in \eqref{eq:dtau} exists and is continuous and bounded. Furthermore, by the dominated convergence theorem, \ref{eq:K}, $(\partial_t - \mL^\ast)K_\tau^\ast = 0$,
     and $\partial_t\phi + \mL\phi \in B([0,T],C^{n-1}_b(\mR^d))$,
     we get that
     \begin{align*}
         J_1 &= \int_0^t \int_{s}^t  \langle \rho(s)V_1(s) + V_2(s), \partial_\tau D_yK_{\tau-s}^\ast\ast \phi(\tau) + D_yK_{\tau - s}^\ast\ast \partial_{\tau}\phi(\tau)\rangle\,d\tau\, ds\\
         &=\int_0^t \int_{s}^t  \langle \rho(s)V_1(s) + V_2(s), D_yK_{\tau-s}^\ast\ast( \mL \phi(\tau) + \partial_{\tau}\phi(\tau))\rangle\,d\tau\, ds\\
         &=\int_0^t \int_0^\tau \langle \rho(s)V_1(s) + V_2(s), D_yK_{\tau-s}^\ast\ast( \mL \phi(\tau) + \partial_{\tau}\phi(\tau))\rangle\, ds \,d\tau.
     \end{align*}
     We finish the proof by adding $I_1,J_1$, and $J_2$ and by again using the definition of mild solution.
\end{proof}

\section{Stability of the MFG system with respect to the initial time}\label{sec:leftder}

Let $h>0$ and let $(u_h,m_h)$ be the solution to the MFG system with initial measure $m_0$ on the time interval $(t_0-h,T)$. We also denote by $(u,m)$ the solution on $(t_0,T)$. We will use the following lemma on stability of the MFG system under the changes of the initial time.

\begin{lemma}\label{lem:timestab}
Let $\sigma\in (0,\alow)$ and $\eps\in(0,\sigma)$. There exist a constant $C>0$ and a modulus of continuity $\omega$ independent of $m_0$, such that for every $t_0\in [0,T)$ and $h\in (0,1)$, there is a constant $C>0$ independent of $h$ such that
\begin{align}
    &\sup\limits_{t\in [t_0,T]} \|u_h(t,\cdot)-u(t,\cdot)\|_{C^{3+\sigma}_b(\mR^d)} +\sup\limits_{t\in [t_0,T]} d_0(m_h(t),m(t)) \leq Ch^{\frac{1}{2}},\label{eq:ts1}\\
    &\sup\limits_{t\in [t_0-h,t_0]} \|u_h(t,\cdot)-u(t_0,\cdot)\|_{C^{3+\sigma-\eps}_b(\mR^d)} + \sup\limits_{t\in [t_0-h,t_0]} d_0(m_h(t),m_0) \leq \omega(h).\label{eq:ts2}
\end{align}
\end{lemma}
\begin{proof}
Note that $(u,m)$ and $(u_h,m_h)$ solve the MFG system on $(t_0,T)$ with the initial measures $m_0$ and $m_h(t_0)$ respectively. Therefore, by using Theorem~\ref{th:Lip} we get that
\begin{align*}
    \sup\limits_{t\in [t_0,T]} \|u_h(t,\cdot)-u(t,\cdot)\|_{C^{3+\sigma}_b(\mR^d)} + \sup\limits_{t\in [t_0,T]}d_0(m_h(t),m(t))\leq Cd_0(m_0,m_h(t_0)).
\end{align*}
Furthermore, by Theorem~\ref{th:FP} (b) we get that 
\begin{align*}
    d_0(m_0,m_h(t_0)) \leq c_0(1+ \|D_pH(\cdot,u_h,Du_h)\|_{\infty})h^{\frac 1{2}}.
\end{align*}
Since $u_h$ and $Du_h$ are bounded independently of $h$ (see Theorem~\ref{th:MFGwp}), we obtain \eqref{eq:ts1}.

In order to get \eqref{eq:ts2} we first use the triangle inequality. For every $t\in[t_0-h,t_0]$,
\begin{align*}
     \|u_h(t,\cdot)-u(t_0,\cdot)\|_{C^{3+\sigma-\eps}_b(\mR^d)} + d_0(m_h(t),m_0)&\leq \|u_h(t,\cdot)-u_h(t_0,\cdot)\|_{C^{3+\sigma-\eps}_b(\mR^d)} + d_0(m_h(t),m_h(t_0))\\
     &+\|u_h(t_0,\cdot)-u(t_0,\cdot)\|_{C^{3+\sigma-\eps}_b(\mR^d)} + d_0(m_h(t_0),m_0).
\end{align*}
The estimate for the last two terms follows from \eqref{eq:ts1}. By Theorem~\ref{th:FP} (b) we get 
\begin{align*}
    d_0(m_h(t),m_h(t_0)) \leq c h^{\frac 1{2}},
\end{align*}
and by the last part of Theorem~\ref{th:HJ},
\begin{align*}
    \|u_h(t,\cdot)-u_h(t_0,\cdot)\|_{C^{3+\sigma-\eps}_b(\mR^d)} \leq\omega(t_0-t) \leq \omega(h),
\end{align*}
which ends the proof of \eqref{eq:ts2}.
\end{proof}

\bibliographystyle{abbrv}
\bibliography{bib-file}
\end{document}